\documentclass{amsart}

\usepackage{amssymb}
\usepackage{amsfonts}
\usepackage{amsmath,amsthm}
\usepackage{natbib}

\usepackage{setspace}
\usepackage{color}
\usepackage{xcolor}
\usepackage{breakcites}
\usepackage{amsaddr}

\newtheorem{theorem}{Theorem}[section]

\newtheorem{lemma}[theorem]{Lemma}
\date{\today}
\newtheorem{corollary}[theorem]{Corollary}
\newtheorem{remark}[theorem]{Remark}
\newtheorem{assumption}{Assumption}
\newtheorem{example}[theorem]{Example}

\newcommand{\EE}{\ensuremath{\mathbb{E}}}

\newcommand{\ul}{\ensuremath{\lfloor t/\Delta_n\rfloor}}

\newcommand{\ulT}{\ensuremath{\lfloor T/\Delta_n\rfloor}}

\newcommand{\indicator}{\ensuremath{\mathbb{I}}}

\allowdisplaybreaks

\newcommand{\dennis}[1]{\textcolor{red}{#1}}

\title[Functional Data Analysis for SPDE]{A feasible central limit theorem for realised covariation of SPDEs in the context of functional data
}

\author{Fred Espen Benth}
\address{Department of Mathematics, University of Oslo, P.O. Box 1053, Blindern, 0316, OSLO, Norway}

\author{Dennis Schroers}
\address{Department of Mathematics, University of Oslo, P.O. Box 1053, Blindern, 0316, OSLO, Norway}

\author{Almut E.~D.~Veraart}
\address{Department of Mathematics, Imperial College London, 180 Queen’s Gate, London, SW7 2AZ, UK}

\begin{document}

\maketitle

\begin{abstract}
   This article establishes an asymptotic theory for volatility estimation in an infinite-dimensional setting.
   We consider  mild solutions of semilinear stochastic partial differential equations and derive a stable central limit theorem for the \textit{semigroup adjusted realised covariation} ($SARCV$), which is a consistent estimator of the integrated volatility and a generalisation of the realised quadratic covariation to Hilbert spaces.
   Moreover, we 
    introduce \textit{semigroup adjusted multipower variations} ($SAMPV$) and establish their weak law of large numbers; using SAMPV, we construct a consistent estimator of the 
   asymptotic covariance of the  mixed-Gaussian limiting process appearing in the central limit theorem for the SARCV, resulting in  a feasible asymptotic theory.
    Finally, we outline how our results can be applied even if observations are only available on a discrete space-time grid. 
\end{abstract}

\keywords{Keywords: central limit theorem, high-frequency estimation,  functional data, SPDE, power variations, volatility, $C_0$-semigroups, Hilbert-Schmidt operators}

\newpage
\tableofcontents

\section{Introduction}
Estimation of volatility is of great importance for capturing the second-order structure of a random dynamical system. 
In this work, we develop a feasible asymptotic distribution theory for the estimation of the integrated volatility operator 
$\int_0^t\Sigma_s ds:=\int_0^t\sigma_s\sigma_s^* ds$ corresponding to a stochastic partial differential equation (SPDE) in a separable Hilbert space $H$ of the form
\begin{equation}\label{SPDE}
    dY_t= (\mathcal A Y_t +\alpha_t)dt+ \sigma_t dW_t,\quad t\in [0,T],
\end{equation}
based on discrete observations of its mild solution within a finite time-interval $[0,T]$ for $T>0$.
Here $\mathcal A$ is the generator of a strongly continuous semigroup $\mathcal S:=(\mathcal S(t))_{t\geq 0}$ on $H$, $W$ is a cylindrical Wiener process, $\alpha$ and $\sigma$ are the drift- and volatility processes, respectively (see Section \ref{sec: Limit Theorems for the SARCV} below for a detailed specification). Such SPDEs constitute a well-established framework for describing spatio-temporal dynamics with applications in,  e.g., finance, physics, biology, meteorology and mechanics (cf. the textbooks \cite{DPZ2014}, \cite{PZ2007}, \cite{WeiRockner2015} or \cite{GM2011}).
In the context of infill-asymptotics and in the presence of time-discrete observations $$Y_0,Y_{\Delta_n},...,Y_{\ulT},\quad \Delta_n:= \frac 1n$$ of a realisation of a solution to \eqref{SPDE}, the role of integrated volatility is similar to the one of the covariance operator in the analysis of i.i.d. functional data.
 This becomes particularly evident if $\sigma$ is independent of $W$. In this case
integrated volatility is the conditional covariance of the driving noise, that is,
$$\int_0^t \sigma_s dW_s\big|\sigma\sim\mathcal N\left(0,\int_0^t \Sigma_s ds\right),\quad t\geq 0.$$
Hence, a feasible estimation theory for integrated volatility in this setting could allow standard functional data analysis methods to be applied to the analysis of observations of solutions to SPDEs.


Our theory is based on the \textit{semigroup-adjusted realised covariation} ($SARCV$), given  
for $n\in\mathbb N$ by
\begin{equation}\label{SARCV}
    SARCV_t^n:=\sum_{i=1}^{\ul}\tilde{\Delta}_i^n Y^{\otimes 2}:=\sum_{i=1}^{\ul} \left(Y_{i\Delta_n}-\mathcal S(\Delta)Y_{(i-1)\Delta_n}\right)^{\otimes 2},
\end{equation}
which was shown to be a consistent estimator of the integrated volatility $\int_0^t\Sigma_s ds$ in \cite{Benth2022}. Here $h^{\otimes 2}=\langle h,\cdot\rangle h$ denotes the usual tensor product. In this paper, we consider the more involved task of proving, under suitable regularity conditions, the functional central limit theorem
\begin{align*}
    \Delta_n^{-\frac 12}\left( SARCV_t^n-\int_0^t \Sigma_s ds\right)\stackrel{\mathcal L-s}{\Longrightarrow} \mathcal N(0,\Gamma_t),
\end{align*}
where $\stackrel{\mathcal L-s}{\Longrightarrow}$ 
stands for the stable convergence in law as a process in the Skorokhod space $\mathcal D([0,T],\mathcal H)$. $\mathcal N(0,\Gamma_t)$ is an infinite-dimensional continuous mixed Gaussian process\footnote{Recall that a centred Hilbert space-valued random variable  $X$ is mixed Gaussian with random covariance $C:H\to H$ if conditional on $C$ the random variable $\langle X, h\rangle$ a one-dimensional centred Gaussian distributed  random variable with variance $\langle C h,h\rangle$ for all $h\in H$.
} with values in $\mathcal H$, the space of Hilbert-Schmidt operators on $H$, and with a conditional covariance operator $\Gamma_t$, called the \textit{asymptotic variance}. The above central limit theorem is not feasible, as the asymptotic variance is a priori unknown, so we also derive a consistent estimator for $\Gamma$. As this can be done conveniently by appealing to laws of large numbers for certain adjusted power and bipower variations, we also provide consistency results for general \textit{semigroup-adjusted realised multipower variations} ($SAMPV$) given by
\begin{align}\label{SAMPV}
   SAMPV^n_t(m_1,...,m_k):= \sum_{i=1}^{\ul-k+1} \bigotimes_{j=1}^k\tilde{\Delta}_{i+j-1}^n Y^{\otimes m_j}. 
\end{align}
We refer to the preliminaries below for the  general tensor power notation.

 Compared with the finite-dimensional theory, the semigroup adjustment in the realised covariation and the multipower variations might seem unusual. Nevertheless, the results presented here should be understood as a direct generalisation of the theory for multivariate semimartingales to the setting of semilinear SPDEs as in \eqref{SPDE}.
This is because the semigroup adjustment just becomes relevant if $(Y_t)_{t\in [0,T]}$ is not a semimartingale, which is a purely infinite-dimensional issue.
In fact, if $H$ is finite-dimensional, $(Y_t)_{t\in [0,T]}$ is automatically a semimartingale and dropping the semigroup adjustment in \eqref{SARCV} still yields a consistent estimator, namely, the quadratic covariation
\begin{equation}\label{realised covariation (nota adjusted)}
    RV_t^n=\sum_{i=1}^{\ul} (Y_{i\Delta_n}-Y_{(i-1)\Delta_n})^{\otimes 2}.
\end{equation} One can equivalently think of choosing the semigroup to equal the identity operator on $H$ (i.e. $\mathcal S\equiv I$) for the sake of the limit theorems and move the part of the (in this case) strong solution belonging to the original generator $\mathcal A$ in equation \eqref{SPDE} into the drift $\alpha$.



For over two decades, there have been many contributions to the asymptotic theory for stochastic volatility estimation in a finite-dimensional set-up. 
These include the  articles \cite{BNS2002, BNS2003, BNS2004, ABDL2003} and \cite{Jacod2008}, amongst many others, and  the textbooks  \cite{JacodProtter2012} and  \cite{Ait-SahaliaJacod2014}, focusing on the semimartingale set-up. Moreover, recently, attention has also turned towards finite-dimensional volatility estimation in the context when the observed process is not necessarily a semimartingale, see e.g.~\cite{CNW2006}, \cite{BNCP11}, \cite{BNCP10b}, \cite{CorcueraHPP2013},  \cite{CorcueraNualartPodolskij2015}, \cite{ChongMies2020}, \cite{ChongDelerue2022}, \cite{Szymanski2022} and \cite{VG2019, PV2019}.

There are two recent strands of research that are related to the infinite-dimensional case:
during the last decade, some effort went into the
 generalisation of ARCH and GARCH models for functional data, appearing at a possibly high frequent rate in \cite{hoermann2013}, \cite{aue2017}, \cite{cerovecki2019}, \cite{sun2020} and \cite{kuhnert2020}. At the same time, a lot of recent research has been devoted to the intricate problem of estimating volatility based on observations of finite-dimensional realisations of second-order stochastic partial differential equations (cf. \cite{Bibinger2020}, \cite{Chong2020}, \cite{ChongDalang2020}, \cite{Bibinger2019}, \cite{Hildebrandt2021}, \cite{altmeyer2021}, \cite{cialenco2022}, \cite{Cialenco2020}, \cite{Altmeyer2022}, \cite{Pasemann2021} to mention some). We refer to \cite{Cialenco2018} for a survey. 
 In that sense, volatility estimation has been approached either discretely in time or discretely in space. So
in contrast to the high research activity in both of these areas, to the best of the authors' knowledge, there appear to be no results at the intersection that allow  making inference on a coherent and potentially smooth spatio-temporal volatility structure as we do here.
Such results, however, may be desirable in many situations. We discuss some applications and relevant types of data in the following subsection.

The presentation of our results is divided into six sections, where after a short consideration of data and some brief preliminaries following this introduction, we outline the setting for the guiding example of term structure models in Section \ref{sec: Term Structure Setting} which makes the otherwise rather abstract operator-theoretic notation more concrete.
We present a detailed discussion on limit theorems and applications of the $SARCV$ in Section \ref{sec: Limit Theorems for the SARCV}, where we also include a short section on the estimation of conditional covariances in Subsection \ref{sec: Conditional Covariance} and establish the  corresponding feasible limit theory  (accounting for the unknown random covariance structure in the basic central limit theorem for this estimator) in Subsection \ref{sec: Feasible central limit theorems for the (SARCV)}. 
A discussion about the convergence behaviour of the na{\"i}ve quadratic variation is added in Subsection \ref{sec: Necessity of the Adjustment}. Afterwards, we outline, how the limit theory can be applied in the case of discrete observations in time and space in Section \ref{sec: discrete data in time and space}.
Section \ref{sec: Limit theorems for multipower variations} addresses the laws of large numbers for the general semigroup-adjusted multipower variations $SAMPV(m_1,...,m_k)$. Section \ref{sec: Outline of Proofs} outlines the proofs of the limit theorems, which are given in full length in an appendix. We summarise and further discuss the results in the concluding   Section \ref{sec: Conclusion}.

\subsection{Considerations on data}

As the $SARCV$ and the $SAMPV$ take into account the Hilbert space-valued data $(Y_{i\Delta_n},i=1,...,\ul)$, the theory presented here is part of the realm of \textit{functional data analysis}. Functional data, which are usually sampled discretely, are often smoothed in order to obtain an element in some suitable function space. In our case, this means that practically every datum $Y_{i\Delta_n}$ should be considered as a smoothed version of discretely sampled data. Assuming that data are of high resolution in the spatial dimension as well, one can obtain fully feasible consistency results and central limit theorems for the integrated volatility operators from our results (see Section \ref{sec: discrete data in time and space} for how this can be done for a regular sampling grid). This means, however, that (at least locally when estimating functionals of the integrated volatility) we need to have dense samples in both space and time.

Taking into account the effort that went into the development of volatility estimation in the case of sampling the solution of an SPDE at a fixed finite number of points in space and a high frequent rate in time, it might be worth underlining the following: the wording ``high frequent'' can be misleading, as this is primarily a matter of scale. 

For instance, in financial forward and futures markets, where one wants to capture price variations for contracts with times-to-maturity of more than a year, intra-daily patterns of variation might, for some purposes, not be as insightful as e.g., intra-monthly ones.
Another example is meteorological data, where in several regions we find a considerable number of weather stations measuring for instance wind, temperature or rainfall at fixed time intervals such as every hour. This leads to a reasonable volume of spatio-temporal data for a week or a month rather than a day. Moreover, reducing volatility estimation on techniques that allow making inference based on fixed multivariate samples of the SPDE might make it hard to capture spatial features like slope and curvature induced by the dynamics of neighbouring stations via  the asymptotic analysis. Dynamics that are dependent on this kind of \textit{derivative information} are of course not just relevant to meteorological applications but are for instance considered important to describe the dynamics of term structure models in finance (c.f.~\cite{Cont2005}). Smooth features of the volatility operator can be conveniently accessed in the functional data framework we elaborate on here and derivative information are inherent in the estimator itself  (due to the adjustment).

On the other hand, in contrast to possibly prevalent perception, there are intraday high-frequency financial data that should eventually be considered functional. One example can be found in the modern structure of intraday energy markets. 
In the European intraday energy markets, participants can continuously trade contracts for energy delivery each day (from late afternoon til midnight) for all 96 quarter-hours of the day ahead. Interpreting this as a discretisation of the curve of all potential forward contracts of the next day, this can, due to no-arbitrage arguments, be considered as a semimartingale in a Hilbert space of functions.
We underline, that our results are new also in the semimartingale case $\mathcal S=I$, leading to an infinite-dimensional theory for realised covariation of $H$-valued semimartingales. Arguably, in that way, it becomes possible to estimate components of the recently treated infinite-dimensional stochastic volatility models (c.f.~\cite{Cox2020, Cox2021}, \cite{BenthSimonsen2018}, \cite{BenthHarang2020}, \cite{BenthSgarra2021}).

\subsection*{Preliminaries and notation}
Throughout this work, $H$, is a separable Hilbert space. The corresponding inner product and norm are denoted by $\langle \cdot,\cdot\rangle_H$ and $\|\cdot\|_H$ and the identity operator on $H$ by $I_H$, where we will drop the $H$-dependence most of the time and simply write $\langle \cdot,\cdot\rangle$, $\|\cdot\|$ and $I$.
If $G$ is another separable Hilbert space,  $h\in H$ and $g\in G$,
 we write $L(G,H)$ for the space of bounded linear operators from $G$ to $H$ and $L(H):=L(H,H)$. We write $\|\cdot\|_{\text{op}}$ for the operator norm on these spaces. $L_{\text{HS}}(G,H)$ denotes the Hilbert space of Hilbert-Schmidt operators from $G$ into $H$, that is $B\in L(G,H)$ such that
$$
\Vert B\Vert_{L_{\text{HS}}(U,H)}^2:=\sum_{n=1}^{\infty}\| B e_n\|^2<\infty,
$$ 
for an orthonormal basis $(e_n)_{n\in\mathbb N}$ of $G$. If $G=H$, we write $\mathcal H:=L_{\text{HS}}(H,H)$. 
The operator $h\otimes g:= \langle h,\cdot\rangle g$ is a Hilbert-Schmidt and even nuclear operator from $H$ to $G$. Recall that $B$ is nuclear, if  $\sum_{n=1}^{\infty}\| B e_n\|<\infty$ for some orthonormal basis $(e_n)_{n\in\mathbb N}$ of $G$.
Moreover, we shortly write $h^{\otimes p}=h\otimes(h\otimes(\cdots\otimes(h\otimes h))$ and $\bigotimes_{j=1}^k h_j:=h_1\otimes ...\otimes h_k:=h_1\otimes( ...\otimes (h_{k-1}\otimes  h_k))$.
We write recursively $\mathcal H^2=\mathcal H=L_{\text{HS}}(H,H)$ 
and $\mathcal H^m=L_{\text{HS}}(H,\mathcal H^{m-1})$,  
for $m>2$. Thus, $\mathcal H^m$ is the space of operators spanned by the orthonormal basis $(e_{j_1}\otimes\cdots\otimes e_{j_m})_{j_1,...,j_m\in\mathbb N}$, for an orthonormal basis $(e_j)_{j\in\mathbb N}$ of $H$ with respect to the Hilbert-Schmidt norm.
As $\mathcal H^m$ 
is isometrically isomorphic to the space $L_{\text{HS}}(\mathcal H^p,\mathcal H^q)$ 
if $p+q=m$ and $p,q\geq 2$ (and $L_{\text{HS}}( H,\mathcal H^q)$ or $L_{\text{HS}}(\mathcal H^p,H)$ 
if $p$ or $q$ is equal to $1$), we will alternate between the notations throughout the paper. 
For instance, if $m$ is even, $\mathcal H^m$ can be identified with the space $L_{\text{HS}}(\mathcal H^{\frac m2},\mathcal H^{\frac m2})$, which is why we can speak without loss of generality of symmetric operators on these spaces. 
Recall moreover that 
\begin{equation}
    \Sigma_t := \sigma_t\sigma_t^*\quad \forall t\in [0,T],
\end{equation}
where $\sigma$ is the stochastically integrable Hilbert-Schmidt operator-valued volatility process (c.f. Section \ref{sec: Limit Theorems for the SARCV}).
We will also need the notation $\Sigma_s^{\mathcal S_n}:=\mathcal S(i\Delta_n-s)\Sigma_s\mathcal S(i\Delta_n-s)^*$ for $s\in ((i-1)\Delta_n,i\Delta_n]$. 
We also need different concepts of convergence of stochastic processes. Recall that a sequence of random variables $(X_n)_{n\in\mathbb N}$ defined on a probability space $(\Omega, \mathcal F,\mathbb P)$ and with values in a Polish space $E$ converges stably in law to a random variable $X$ defined on an extension $(\tilde{\Omega}, \tilde{\mathcal F},\tilde{\mathbb P})$ of $(\Omega, \mathcal F,\mathbb P)$ with values in $E$, if for all bounded continuous $f:E\to \mathbb R$ and all bounded random variables $Y$ on $(\Omega,\mathcal F)$ we have
$\mathbb E[Y f(X_n)]\to\tilde {\mathbb E}[Y f(X)]$ as $n\to\infty$,
where $\tilde {\mathbb E}$ denotes the expectation with respect to $\tilde{\mathbb P}$. If, for a Hilbert space-valued process $X^n$, we have that it converges stably in law as a process in the Skorokhod space $\mathcal D([0,T]; H)$, we write $X^n\stackrel{\mathcal L-s}{\Longrightarrow}X$. Here and throughout we always assume the space $\mathcal D([0,T]; H)$ to be endowed with the classical Skorokhod topology, making it a Polish space (c.f. for instance chapter VI in \cite{Billingsley}).
Moreover, by $X^n\stackrel{u.c.p.}{\longrightarrow}{X}$ we mean convergence uniformly on compacts in probability, i.e. for all $\epsilon>0$ it is $\mathbb P[\sup_{t\in [0,T]} \|X^n(t)-X(t)\|>\epsilon]\to 0$ for $T>0$.

\section{A motivating example: term structure models}\label{sec: Term Structure Setting}
In this section, we discuss the example of term structure models from mathematical finance arising in bond and energy markets. Term structure models, which can conveniently be expressed in form of stochastic partial differential equations, relate the time to maturity of financial contracts to their empirical and theoretical characteristics. 
For an introduction to the SPDE approach to modelling forward curve evolutions we refer to \cite{Filipovic2000} in the case of instantaneous forward rates in bond markets and to \cite{BenthKruhner2014} in the case of instantaneous forward prices in energy and commodity markets.

 Forward curves, respectively forward prices, are usually considered to take their values in some suitable Hilbert space of functions. Besides the space of square-integrable functions $L^2(0,1)$, reproducing kernel Hilbert spaces (RKHS) and in particular Sobolev spaces such as 
 $$H^1(0,1):=\{ h:[0,1]\to\mathbb R: h\text{ is absolutely continuous and } h'\in L^2(0,1)\}$$
equipped with the norm $\|h\|:=h(0)^2+\int_0^1 (h'(x))^2dx$ 
are a reasonable choice for a state space of instantaneous forward curves. 
The compact interval $[0,1]$ contains all observable times to maturity (normalised by the maximal time to maturity observable). The arbitrage-free dynamics of forward curves can then be expressed in terms of the Heath-Jarrow-Morton-Musiela equation 
\begin{align*}
    df_t= \left(\partial_x f_t+\alpha(\sigma_s)\right)ds + \sigma_s dW_s,
\end{align*}
where $\sigma$ is a general Hilbert-Schmidt operator valued process from a noise space $U$ into $H=H^1(0,1)$ and $\alpha:L_{\text{HS}}(U,H)\to H$ is a continuous mapping (c.f. \cite[Section 4.3]{Filipovic2000}) for forward rates and vanishes entirely for commodity and energy price curves (c.f. e.g. \cite{BenthKruhner2018}).
In the space $L^2(0,1)$ of square-integrable functions, $\partial_x$ is defined on its domain $D(\partial_x)= \{h\in H^1(0,1): h(1)=0\}$ and according to \cite[Section 2.11]{Engel1999} generates the nilpotent semigroup of left shifts in $L^2(0,1)$ given by 
\begin{equation}\label{Nilpotent Shift}
    \mathcal S(t)h(x):=\begin{cases} h(x+t), & x+t\leq 1,\\
0, & x+t>1.
\end{cases}
\end{equation}
In the Sobolev space, the differential operator $\partial_x$ can be defined on its domain $D(\partial_x)= \{h\in H^1(0,1): h'\in H^1(0,1)\}$ and combining Corollary 5.1.1 in \cite{Filipovic2000} and \cite[Section 2.3]{Engel1999} it is then the generator of the strongly continuous semigroup of left shifts on $H^1(0,1)$ given by 
\begin{equation}\label{Sobolev Shift}
    \mathcal S(t)h(x):=\begin{cases} h(x+t), & x+t\leq 1,\\
h(1), & x+t>1.
\end{cases}
\end{equation}
We may choose the noise space to be $U=L^2(0,1)$, such that we can interpret $\sigma_s$ as a Hilbert-Schmidt operator from $L^2(0,1)$ into itself or that it maps into $H^1(0,1)\hookrightarrow L^2(0,1)$ and is Hilbert-Schmidt with respect to the norm on $H^1(0,1)$ if $H=H^1(0,1)$. As such, it is given as a kernel operator $$ \sigma_s f(x)= \int_0^1 q_s(x,y) f(y) dy,\quad \forall s\geq 0, x\in [0,1]. $$
 In the case that $H=H^1(0,1)$ we alternatively could have chosen $U=H^1(0,1)$, as by Theorem 9 in \cite{Berlinet2011} we have that in an RKHS on $[0,1]$ with kernel $k$, every continuous linear operator $L$ is given by a kernel operator with kernel $l(x,y)= \langle k(x,\cdot),L^*k(\cdot,y) \rangle$ in the sense that
$$L f(x)= \langle f, l(\cdot,x)\rangle,\quad \forall x\in [0,1].$$
 We will come back to the estimation of integrated volatility in this setting for $H=H^1(0,1)$ in Section \ref{sec: discrete data in time and space}.

\section{Limit theorems for the $SARCV$}\label{sec: Limit Theorems for the SARCV}

Throughout this work we fix $(Y_t)_{t\in [0,T]}$ for $T>0$ to be the mild solution of the SPDE \eqref{SPDE}, i.e.~$Y$ is a continuous adapted stochastic process defined on a filtered probability space $(\Omega,\mathcal F,(\mathcal F_t)_{t\in[0,T]},\mathbb P)$ with right-continuous filtration $(\mathcal F_t)_{t\in [0,T]}$ taking values in the separable Hilbert space $H$ and is given by the stochastic Volterra process
\begin{align}\label{mild Ito process}
    Y_t=\mathcal S(t)Y_0+\int_0^t\mathcal S(t-s)\alpha_s ds+\int_0^t \mathcal S(t-s)\sigma_s dW_s,\quad t\in [0,T].
\end{align}
Here, $\mathcal S:=(\mathcal S(t))_{t\geq 0}$ is a strongly continuous semigroup on $H$ generated by $\mathcal A$ and $W$ is a cylindrical Wiener process potentially on another separable Hilbert space $U$ (with covariance operator $I_U$). Moreover, $\alpha$ is an almost surely Bochner integrable adapted stochastic process with values in $H$ and $\sigma$ is a Hilbert-Schmidt operator-valued process that is stochastically integrable   with respect to $W$, i.e. for $\Omega_T:=[0,T]\times \Omega$,
$$\sigma\in \left\{\Phi:\Omega_T\to L_{\text{HS}}(U,H): \Phi\text{ predictable and }\mathbb P\left[\int_0^T\|\Phi(s)\|_{L_{\text{HS}}(U,H)}^2ds<\infty\right]=1\right\}$$
(c.f. for instance Chapter 2.5 in \cite{WeiRockner2015} for the definition of the stochastic integral in this context). Both coefficients $\alpha$ and $\sigma$ can in principle be state (or even path) dependent, provided that there is a mild solution of the form \eqref{mild Ito process} to the equation. We refer to $(Y_t)_{t\in[0,T]}$ as a mild It\^o process.

We present first our result on the asymptotic behaviour of the semigroup-adjusted realised covariation ($SARCV$), 
as it is the most important example of the (semigroup-adjusted) power variations. The law of large numbers for general multipower variations is postponed to the next section.

\subsection{Infeasible central limit theorems for the $SARCV$}
 As it was shown in \cite{Benth2022}, the law of large numbers needs no further assumption on $Y$
\footnote{There are two minor differences with respect to the limit theory established in \cite{Benth2022}:
 First, the driver $W$ was assumed to have a covariance that is of trace class. However,
    considering the stochastic integral of a Hilbert-Schmidt operator-valued process with respect to a cylindrical Wiener noise or the stochastic integral of a process with values in $L_{HS}(Q^{ 1/2}U,H)$ with respect to the corresponding trace class ($Q$-)Wiener process in $U$, does not make a difference. The stochastic integral can (on an extension of the probability space) in both cases be translated into one or the other, due to the martingale representation theorems (c.f. Section 2.2.5 in \cite{GM2011}).
   Second, the drift was assumed to be almost surely square-integrable. 
   Here, in this paper, we do not aim to derive a rate of convergence via the laws of large numbers and are in that regard able to drop these conditions.} 
:
\begin{theorem}\label{T: LLN for the SARCV}
For a mild It{\^o} process $Y$ of the form \eqref{mild Ito process}, we have
\begin{align*}
    SARCV^n\stackrel{u.c.p.}{\longrightarrow}\left(\int_0^t \Sigma_s ds\right)_{t\in[0,T]}.
\end{align*}
\end{theorem}
The derivation of a corresponding central limit theorem, that is, the asymptotic normality of 
\begin{equation}
    \tilde X_t^n:=SARCV_t^n- \int_0^t\Sigma_s ds:= \sum_{i=1}^{\ul}(\tilde{\Delta}_i^nY)^{\otimes 2} - \int_0^t\Sigma_sds,
\end{equation}
is more involved. First of all, already in finite dimensions some further conditions have to be imposed, which is why we give an analogue of the fairly mild Assumption 5.4.1(i) from \cite{JacodProtter2012}:
\begin{assumption}\label{As: Integrability Assumption on Drift and Volatility}
The coefficients $\alpha$ and $\sigma$ satisfy the following local integrability condition:
$$\mathbb P\left(\int_0^T \|\alpha_s\|^2+\|\sigma_s\|_{L_{\text{HS}}(U,H)}^{4} ds <\infty \right)=1.$$
\end{assumption}
The law of large numbers, Theorem \ref{T: LLN for the SARCV}, is very general, as there are no additional assumptions imposed on $Y$. However, the subtle difference to the convergence of realised variation in the finite-dimensional case is hidden in the rate of convergence. Even if Assumption \ref{As: Integrability Assumption on Drift and Volatility} holds, the speed of convergence may become arbitrarily slow and might not be of magnitude $\mathcal O_p(\sqrt{\Delta_n})$ anymore, where $\mathcal O_p$ denotes {\it boundedness in probability} (c.f. Example \ref{Ex: counterexample for the Spatial regularity assumption} below). The latter is however an important condition to obtain a general infinite-dimensional central limit theorem with respect to some uniform operator topology such as the one induced by the Hilbert-Schmidt  norm. In order to overcome this issue, we  impose further assumptions which increase the regularity of the sample paths of the process or consider limit theorems for the mild solution process evaluated at functionals $h$ that induce some regularity of the respective finite-dimensional process $\langle Y_t,h\rangle$\footnote{One might hope to find a uniform rate $c_n$ such that $c_n^{-1}(SARCV_t^n-\int_0^t \Sigma_s ds)$ converges in distribution to a nontrivial law with respect some 
operator-topology. This is not possible in the general context we are examining:
Example \ref{Ex: counterexample for the Spatial regularity assumption} describes a case, in which for certain irregular functionals $\sqrt n\langle (SARCV_t^n-\int_0^t \Sigma_s ds) h,g\rangle$ diverges. On the other hand, for another choice of functionals ($h,g\in D(\mathcal A)$ for instance) we obtain convergence in distribution to a centered Gaussian law. 
}.

To this purpose, we introduce the notion of Favard spaces. Here, for $\gamma \in (0,1)$ the $\gamma$-Favard space $F_{\gamma}^{\mathcal S}$ is defined by
$$F_{\gamma}^{\mathcal S}=F_{\gamma}^{\mathcal S}(H):=\left\{h\in H:\|h\|_{F_{\gamma}^{\mathcal S}(N)}:=\sup_{t\in [0,N]}\left\|t^{-\gamma}\left(I-\mathcal S(t)\right)h\right\|<\infty, \forall N>0\right\}. $$
As $D(\mathcal A)\subset F^{\mathcal S}_{\gamma}$, these spaces always form dense subsets of $H$ and become Banach spaces when equipped with the norm $\sup_{N\geq 0}\|\cdot\|_{F_{\gamma}^{\mathcal S}(N)}$ as long as the semigroup has a negative growth bound (c.f. \cite{Engel1999}, Chapter II.5). An example of practical importance for a subset of a $1/2$-Favard space are the evaluation functionals in a Sobolev space (this is outlined further in Section \ref{sec: discrete data in time and space}).

For functionals in the $\frac 12$-Favard space, we have the following central limit theorem in the weak operator topology:
\begin{theorem}\label{T: Central limit theorem for functionals of the quadratic covariation}
Define the covariance operator process $\Gamma_t$ for $t\in [0,T]$ on $\mathcal H$ by
$$\Gamma_t B:= \int_0^t\Sigma_s(B+B^*)\Sigma_sds,\quad B\in\mathcal H.$$
Let $B\in \mathcal H$ be an operator with a finite-dimensional range of the form
$B=\sum_{l=1}^K\mu_l h_l\otimes g_l$
for $h_l,g_l\in F_{\frac 12}^{\mathcal S^*}, \mu_l\in\mathbb R$ for $l=1,...,K$, $K\in\mathbb N$
and let Assumption \ref{As: Integrability Assumption on Drift and Volatility} hold. Then 
$$\left(\Delta_n^{-\frac 12}\langle \tilde{X}_t^n ,B\rangle_{\mathcal H}\right)_{t\in[0,T]} \stackrel{\mathcal L-s}{\Longrightarrow}\left(\mathcal N(0, \langle \Gamma_t B,B\rangle)\right)_{t\in[0,T]},$$
where the limiting process on the right is, conditionally on $\mathcal F$, a continuous centered Gaussian process with independent increments defined on a very good filtered extension $(\tilde{\Omega},\tilde{\mathcal F},\tilde{\mathcal F}_t,\tilde{\mathbb P})$ of $(\Omega,\mathcal F,\mathcal F_t,\mathbb P)$.
\end{theorem}
For the notion of a very good filtered extension we refer to  \cite[Sec.~2.4.1]{JacodProtter2012}.
Let us now give two examples of operators $B$ that can be chosen in Theorem \ref{T: Central limit theorem for functionals of the quadratic covariation} to make inference on term structure models.
\begin{example}\label{Ex: Point evaluations and local averages}
We consider examples of practical importance: local averages and evaluation functionals.
\begin{itemize}
    \item[(a)](Local averages) Consider the case that $H=L^2(0,1)$ and $\mathcal S$ is the nilpotent shift semigroup defined in \eqref{Nilpotent Shift}. We have for $t\in [0,1]$ that $\mathcal S^*(t)f(x)= \indicator_{[t,1]}(x) f(x-t)$. Then it holds, for $0< b\leq 1$ and $t<b$, that 
    \begin{align*}
         \|(\mathcal S(t)^*-I)\indicator_{[0,b]}\|_{L^2(0,1)}^2
        =  (\min(b+t,1)-b)+t,
    \end{align*}
            which shows that $\indicator_{[0,b]}\in F_{\frac 12}^{\mathcal S^*}$ but  $\indicator_{[0,b]}\notin  F_{ \gamma}^{\mathcal S^*}$ for any $\gamma> 1/2$. Since Favard-spaces are vector spaces, this yields in particular, that by virtue of Theorem \ref{T: Central limit theorem for functionals of the quadratic covariation} we can analyse one-dimensional (or multivariate) stochastic processes that arise as local averages over certain areas of a mild solution. That is, we can readily analyse time series $\bar y_{i\Delta_n}^{a,b}, i=0,...,\ulT$ where
            $$ \bar y_{i\Delta_n}^{a,b}:=\frac 1{b-a}\int_a^b Y_{i\Delta_n}(x)dx = \frac 1{b-a}\langle Y_{i\Delta_n},\indicator_{a,b}\rangle_{ L^2(0,1)}.$$
             For forward curves in term structure models this kind of sampling structure appears naturally as differences of yield curve values or (log-)bond prices which can be observed in the market, since for a zero coupon bond price at time $t$ with time to maturity $x+t$ we have
            $$P_t(x)=e^{-\int_0^x f_t(y)dy}.$$ In energy markets we also observe prices as weighted averages of instantaneous forward prices in the form of energy-swap contracts guaranteeing delivery of energy over a certain time (c.f. \cite{BSBK}).
            A practically relevant class of operators are, hence, weighted sums of indicator functionals of the form
            $$\sum_{i,j=1}^d w_{i,j} \indicator_{[a_i,b_i]}\otimes\indicator_{[a_j,b_j]},$$
            for some intervals $[a_i,b_i]\subset [0,1]$ and $w_{i,j}\in \mathbb R$ for $i,j=1,...,d$.
            \item[(b)](Evaluation functionals) For $H=H^1(0,1)$ we can define evaluation functionals $\delta_x$ by $\delta_x f=f(x)$  for all $x\in [0,1]$. These functionals satisfy $\delta_x\in F_{\frac 12}^{\mathcal S^*}$, while  $\delta_x\notin  F_{ \gamma}^{\mathcal S^*}$ for any $\gamma> 1/2$ if $x\in [0,1)$. This is shown in Lemma \ref{L: Evaluation functionals in Sobolev spaces have 1/2 regularity} below where statistical estimation within this framework is elaborated in a fully discrete setting. We can, hence, analyse one-dimensional (or multivariate) stochastic processes that arise as evaluations of mild solutions of first-order stochastic partial differential equations at a finite number of points.
            A practically relevant class of operators are, thus, weighted sums of evaluation functionals of the form
            $$B= \sum_{i,j=1}^d w_{i,j}\delta_{x_i}\otimes \delta_{x_j},$$
             for some elements $x_i\subset [0,1]$ and $w_{i,j}\in \mathbb R$ for $i,j=1,...,d$.
\end{itemize}
\end{example}

In order to derive a stable central limit theorem for the $SARCV$ with respect to the Hilbert-Schmidt norm, we need to impose regularity assumptions on the volatility process itself, namely:
\begin{assumption}\label{As: Spatial regularity}
One of the two following conditions holds:
\begin{itemize}
    \item[(i)] 
    $$\int_0^T\sup_{t\in[0,T]}\mathbb E [\| t^{-\frac 12}(I-\mathcal S(t))\sigma_s\|_{\text{op}}^2]ds<\infty;$$
\item[(ii)]
$$\mathbb P\left[\int_0^T\sup_{t\in[0,T]}\| t^{-\frac 12}(I-\mathcal S(t))\sigma_s\|_{\text{op}}^2ds<\infty\right]=1.$$
\end{itemize}
\end{assumption}
\begin{remark}
Observe that if the semigroup has negative growth bound and, thus, $F_{\frac 12}^{\mathcal S}$ is a Banach space, Assumption \ref{As: Spatial regularity}(i) and (ii) can be rewritten as 
\begin{itemize}
    \item[(i)] $\sigma\in L^2\left([0,T],F_{\frac 12}^{\mathcal S}\left(L^2(\Omega,L\left(U,H)\right)\right)\right)$
    \item[(ii)]$\mathbb P\left[\sigma\in L^2\left([0,T],L\left(U,F_{\frac 12}^{\mathcal S}(H)\right)\right)\right]=1$.
\end{itemize}
\end{remark}
Now we state the associated central limit theorem.
\begin{theorem}\label{T:Infeasible Central Limit Theorem}
Let $\Gamma$ be as in Theorem \ref{T: Central limit theorem for functionals of the quadratic covariation}. Under Assumptions \ref{As: Integrability Assumption on Drift and Volatility} and \ref{As: Spatial regularity} we have that 
\begin{equation}
    (\Delta_n^{-\frac 12}\tilde{X}_t^n)_{t\in[0,T]}\stackrel{\mathcal L-s}{\Longrightarrow} (\mathcal N(0,\Gamma_t))_{t\in[0,T]},
\end{equation}
where the limiting process on the right is, conditionally on $\mathcal F$, a continuous centered $\mathcal H$-valued Gaussian process with independent increments defined on a very good filtered extension $(\tilde{\Omega},\tilde{\mathcal F},\tilde{\mathcal F}_t,\tilde{\mathbb P})$ of $(\Omega,\mathcal F,\mathcal F_t,\mathbb P)$.
\end{theorem}
Assumption \ref{As: Spatial regularity} is a sharp regularity criterion for the validity of the central limit theorem in the Hilbert-Schmidt norm:
\begin{example}\label{Ex: counterexample for the Spatial regularity assumption}
Assumption \ref{As: Spatial regularity} is sharp in the sense that for all $\mathfrak H<\frac 12$ we can always find a deterministic and constant volatility $\sigma$, such that \begin{equation}\label{too weak spatial regularity}
     \sup_{t\in[0,T]}\| t^{-\mathfrak H}(I-\mathcal S(t))\sigma\|_{L_{\text{HS}}(U,H)}<\infty,
 \end{equation} 
    but convergence in distribution of $\sqrt n\tilde X_t^n$ cannot take place, even with respect to the weak operator topology. Such a specification can be done for instance in the following way: Take $H=L^2[0,2]$, $(\mathcal S(t))_{t\geq 0}$ the nilpotent semigroup of left-shifts, such that for $x\in [0,2],t\geq 0$ it is $\mathcal S(t)f(x)=\indicator_{[0,2]}(x+t)f(x+t)$ and 
    $\sigma=e\otimes X,$
    where $e\in H$ such that $\|e\|=1$ and $X$ is an appropriately chosen path of a rough fractional Brownian motion. That is,  $X(x)=B^{\mathfrak H}_x(\omega)$ for a fractional Brownian motion $(B^{\mathfrak H}_x)_{x\geq 0}$ with Hurst parameter $\mathfrak H<\frac 12$, defined on another probability space $(\bar{\Omega},\bar{\mathcal F},\bar{\mathbb P})$ and $\omega\in \bar{\Omega}$ is such that $B^{\mathfrak H}_x$ is $\mathfrak H$-Hölder continuous and guarantees divergence of $\sqrt n\tilde X_t^n$. Clearly, $B^{\mathfrak H}(\omega)$ is globally $\mathfrak H$-Hölder continuous on $[0,2]$ and we can find a $C> 0$ such that
    $$\left\| \frac{(I-\mathcal S(t))}{t^{\mathfrak H}}\sigma\right\|_{L_{\text{HS}}(U,H)}^2=\int_0^{2-t} \left(\frac{B^{\mathfrak H}_x(\omega)-B^{\mathfrak H}_{x+t}(\omega)}{t^{\mathfrak H}}\right)^2dx+\int_{2-t}^2 \left(\frac{B^{\mathfrak H}_x(\omega)}{t^{\mathfrak H}}\right)^2dx\leq C
    $$
    Hence, we have that \eqref{too weak spatial regularity} holds. However, it is intuitively clear, that the lower $\mathfrak H$ is chosen, the worse the impact on the regularity of $Y$ is, which eventually leads to divergence of $\sqrt n\tilde X_t^n$ for the rough case $\mathfrak H<\frac 12$.
    We give a detailed verification of this counterexample as well as how to choose the appropriate $\omega$ in Appendix \ref{sec: Remaining proofs}.
    
     In order to account for such irregularities, one often scales the increments in a particular way and still obtains a feasible limit theory, such as was done for second-order stochastic partial differential equations in \cite{Bibinger2020} or \cite{Chong2020} and for Brownian semistationary processes in  \cite{CNW2006}, \cite{BNCP11}, \cite{BNCP10b}, \cite{CorcueraHPP2013},  \cite{CorcueraNualartPodolskij2015} and \cite{VG2019, PV2019}. However,  by the law of large numbers, Theorem \ref{T: LLN for the SARCV} we deduce that these rescaling arguments would lead to inconsistent estimators.
    \end{example}
To get an intuition about the regularity that is induced by Assumption \ref{As: Spatial regularity}, observe the following
\begin{remark}
Assumption \ref{As: Spatial regularity}(i) (and \ref{As: Spatial regularity}(ii)) increases the regularity of $Y$ in space and time: In fact, suppose that the volatility has bounded second moment, that is, 
$\sup_{s\in[0,T]}\mathbb E\left[\left\|\sigma_s\right\|_{L_{\text{HS}}(U,H)}^2\right]<\infty.$
The assumption then says that the stochastic convolution is weakly mean-square $\frac 12$-regular in time, as for each $h\in H$ and $0\leq u<t\leq T$
\begin{align}\label{Induced weak time regularity}
   & \mathbb E\left[ \left(\langle\int_0^t \mathcal S(t-s)\sigma_s dW_s-\int_0^u\mathcal S(u-s)\sigma_s dW_s,h\rangle\right)^2\right]^{\frac 12}\notag\\
     &\qquad\leq 
     \left(\int_0^u\mathbb E\left[ \left\|\left((\mathcal S(t-u)-I) \mathcal S(u-s)\sigma_s \right)^*h\right\|^2\right]ds\right)^{\frac 12}\notag\\
     &\qquad\qquad+\left( \int_u^t\mathbb E\left[\left\|\left(\mathcal S(t-s)\sigma_s \right)^*h\right\|^2\right]ds\right)^{\frac 12}\notag\\
     &\qquad= \mathcal O\left((t-u)^{\frac 12}\right).
\end{align} 
If we are in a reproducing kernel Hilbert space (i.e.~a Hilbert space of functions, say over an interval in $\mathbb R$ such that the evaluation functionals $\delta_x$ are continuous) and the semigroup is the shift semigroup, it is easy to see that the assumption also gives mean-square $\frac 12$-regularity in space: To see this, we write $\delta_x$ for the evaluation functionals in $H$ and observe that
\begin{align}\label{Induced weak space regularity}
   & \mathbb E\left[\left|\int_0^t \mathcal S(t-s)\sigma_s dW_s (x)-\int_0^t \mathcal S(t-s)\sigma_s dW_s (y)\right|^2\right]^{\frac 12}\notag\\
    &\qquad=\mathbb E\left[\left|\delta_0\left(\int_0^t \mathcal S(t-s)\mathcal S(y)(\mathcal S(x-y)-I)\sigma_s dW_s\right)\right|^2\right]^{\frac 12}\notag\\
    &\qquad\leq  \|\delta_0\| \sup_{t\in[0,T]}\|\mathcal S (t)\|_{\text{op}}\left(\int_0^t \mathbb E\left[\|\mathcal S(x-y)-I)\sigma_s \|_{\text{op}}^2\right]ds\right)^{\frac 12}\notag\\
    &\qquad=\mathcal O(|x-y|^{\frac 12}),
\end{align}
by It{\^o}'s formula for $x>y$.
Combining \eqref{Induced weak time regularity} and \eqref{Induced weak space regularity} we find that the random field $(t,x)\mapsto\int_0^t\mathcal S(t-s)\sigma_s dW_s(x)$ has mean-square regularity $\frac 12$ in space and in time.
\end{remark}
\begin{remark}[What if the semigroup adjustment is infeasible?]\label{Rem: Strong solutions and RV vs SARCV}
The semigroup adjustment can readily be implemented in situations in which the semigroup is known and has a simple form (e.g. a simple left-shift as in term structure models).
However, it should be noted that the adjustment might be hard or even impossible to implement in some cases. For instance, a commonly encountered situation is $\mathcal A=\kappa \mathcal A'$ for some known generator $\mathcal A'$ of a strongly continuous semigroup $(\mathcal T(t))_{t\geq 0}$ in $H$ and an unknown parameter $\kappa$. In this case, we have $\mathcal S(t)=\mathcal T(\kappa t)$ and without further knowledge of the parameter $\kappa$, SARCV is an infeasible estimator.

It is, hence, important to characterise situations, in which the semigroup adjustment is superfluous and we can use the simpler infinite-dimensional realised variation \eqref{realised covariation (nota adjusted)}. We give weak regularity conditions on the volatility guaranteeing consistency and asymptotic normality of $RV_t^n$ in section \ref{sec: Necessity of the Adjustment}. A simple, yet very relevant situation is when the volatility has a finite second moment and is contained in the domain of the generator $\mathcal A$ of the semigroup. Assuming the drift to be zero for convenience,  it is well known that in this case the stochastic convolution \eqref{mild Ito process} is a strong solution to the SPDE
$$dY_t=\mathcal AY_tdt+\sigma_tdW_t,\quad Y_0=0,\,t\in[0,T],$$
(which is especially fulfilled if $\mathcal A$ is continuous), c.f. \cite[Theorem 3.2]{GM2011}.
This yields that $Y$ is of the form
$$Y_t=\int_0^t \mathcal A Y_s ds+\int_0^t \sigma_s dW_s,$$
such that we can reinterpret $Y$ to be a mild It{\^o} process of the form \eqref{mild Ito process} with the semigroup to be the identity and $\alpha_t = \mathcal A Y_t$ for the sake of the limit theory. In that way, Assumption \ref{As: Spatial regularity} is trivially fulfilled and the realised covariation $RV^n_t$ (c.f. \eqref{realised covariation (nota adjusted)}) 
is consistent and asymptotically mixed normal.

At the same time, the adjustment with the initial semigroup (generated by $\mathcal A$) also leads to a consistent estimator, since the semigroup is Lipschitz-continuous on the range of the volatility due to the mean value theorem. Thus, $SARCV^n$ converges in probability to the same limit and has the same asymptotic normal distribution as $RV^n$.
However, the assumption that the volatility is in the domain of the generator $\mathcal A$ or the existence of a strong solution is oftentimes too strong and we give some weaker regularity conditions in Section \ref{sec: Necessity of the Adjustment} enabling us to use $RV_t^n$ even in some situations in which $Y$ does not have the pleasant semimartingale structure of a strong solution. Yet, in some important cases also these conditions might be too strong and
the asymptotic equality of the semigroup-adjusted and the nonadjusted variation is not in general fulfilled (c.f. Section \ref{sec: Necessity of the Adjustment}).
\end{remark}
\begin{remark}[Which CLT to use in practise?]
Both results Theorem \ref{T: Central limit theorem for functionals of the quadratic covariation} and \ref{T:Infeasible Central Limit Theorem} are central limit theorems for the same process. While Theorem \ref{T:Infeasible Central Limit Theorem} yields a more general convergence, it comes along with the additional regularity Assumption \ref{As: Spatial regularity}, while Theorem \ref{T: Central limit theorem for functionals of the quadratic covariation} does not impose further assumptions on the mild It{\^o} process $Y$ itself, but rather on the functionals under which we observe it. 

Hence, we might use Theorem \ref{T: Central limit theorem for functionals of the quadratic covariation} in situations in which regularity assumptions on the volatility are not reasonable or cannot be guaranteed to hold and we are interested in testing hypotheses or finding confidence intervals of sufficiently regular functionals of the integrated volatility (in terms of the assumption of the theorem). Two important classes of such functionals (or even linear combinations of these) are presented in Example \ref{Ex: Point evaluations and local averages}.
In term structure models, for instance, we might want to quantify the estimation error of the volatility corresponding to a particular economic parameter. For instance, it is usually important to consider the spread between two forward contracts with maturities far from each other. We are then interested in confidence intervals for the volatility of the process $\langle \delta_x-\delta_y, f_t\rangle)_{t\in [0,T]}$ for the long maturity $x$ and the short term maturity $y$ where $\delta_x$ and $\delta_y$ are evaluation functionals $\delta_x f=f(x)$ in the Sobolev space $H^1(0,1)$ which is defined in Section \ref{sec: Term Structure Setting}.
  In this case, we have to characterise the asymptotic distribution of $\int_0^T \langle \Sigma_s (\delta_x-\delta_y), (\delta_x-\delta_y) \rangle ds$. It turns out, that the evaluation functionals $\delta_x$ and $\delta_y$ are sharply in the space $F_{\frac 12}^{\mathcal S^*}$ for the shift semigroup $\mathcal S$ defined in Section \ref{sec: Term Structure Setting}, such that we can use Theorem \ref{T: Central limit theorem for functionals of the quadratic covariation} with the choice $B=(\delta_x-\delta_y)^{\otimes 2}$ (c.f. Lemma \ref{L: Evaluation functionals in Sobolev spaces have 1/2 regularity} below).

On the other hand, if regularity Assumption \ref{As: Spatial regularity} is reasonable to assume, Theorem \ref{T:Infeasible Central Limit Theorem} makes Theorem \ref{T: Central limit theorem for functionals of the quadratic covariation} obsolete. Infinite-dimensional central limit theorems as Theorem \ref{T:Infeasible Central Limit Theorem} can be used to design hypothesis tests based on nonlinear functionals of integrated volatility via an infinite-dimensional Delta method (c.f. \cite[Section 3.9]{vanderVaart1996}), or to make inference on the eigencomponents of integrated volatility in the same way infinite-dimensional limit theorems guarantee the asymptotic normality of empirical eigenfunctions for covariance operators (c.f. \cite{Kokoszka2013}) and we could also test for functionals that are not in the $1/2$ Favard-space of the dual of the semigroup. The latter is for instance the case for indicator functionals (hence, local averages) in $L^2(0,1)$ and the heat semigroup (c.f. Section \ref{sec: Heat Equation} below),  for which the Favard spaces are sharply embedded into Hölder spaces of continuous functions (c.f. \cite[Proposition 5.33]{Engel1999}).
\end{remark}

\subsection{Estimation of conditional covariance}\label{sec: Conditional Covariance}
As argued in the introduction, estimating integrated volatility corresponds to the estimation of the conditional covariance of the noise process if we assume that the volatility and the Wiener process are independent. As opposed to the semimartingale case, however, it is not the conditional covariance of the increments or adjusted increments of a mild solution of an SPDE. The latter can, nevertheless be estimated within our framework as well and might be used for inference on the dynamics.

As a motivation, we show in the next example how we can build time-series models from HJMM-term structure dynamics.
\begin{example}[HJMM-time series model]
Let us come back to the term structure model described in Section \ref{sec: Term Structure Setting}. 
Assume that the drift and volatility processes are independent of the cylindrical Wiener process and stationary.
We want to build a functional quarterly time-series $(F_i)_{i\in \mathbb N}$ for the forward curve process, that describes the dynamics of the arbitrage-free HJMM-dynamics well and might for instance be used in forecasting. 
Measuring time in years, it is then
\begin{align*}
    F_i:=Y_{\frac i4}= & \mathcal S\left(\frac 14\right) Y_{\frac {i-1}4}+\int_{\frac {i-1}4}^{\frac i4} \mathcal S\left(\frac i4-s\right)\alpha_s ds +\int_{\frac {i-1}4}^{\frac i4} \mathcal S\left(\frac i4-s\right)\sigma_s dW_s\\
    =& \mathcal S\left(\frac 14\right) Y_{\frac {i-1}4}+\mu_i + \epsilon_i,
\end{align*}
where
$$\mu_i:=\int_{\frac {i-1}4}^{\frac i4} \mathcal S\left(\frac i4-s\right)\alpha_s ds,\quad \epsilon_i:= \int_{\frac {i-1}4}^{\frac i4} \mathcal S\left(\frac i4-s\right)\sigma_s dW_s.$$
Defining
$\Sigma_i^*:= \int_0^{\frac 14}\mathcal S(\frac 14-s)\Sigma_{s+\frac {(i-1)}4}\mathcal S(\frac 14-s)^*ds$, we obtain a stationary time-series of covariance operators, such that 
$$\epsilon_i|\sigma\sim  N(0,\Sigma_i^*),\quad i\in\mathbb N,$$
forms a weak white noise sequence.

Assuming the time-series $\mu_i$ to be deterministic and constant and potentially violating the no-arbitrage setting, we can proceed in a straightforward manner: If $\mu$ is deterministic and constant, estimation of mean $\mu$ and covariance $C=\mathbb E[\Sigma_i^*]$ can be based on their empirical counterparts via the adjusted increments $(Y_{\frac i4}-\mathcal S(\frac 14)Y_{\frac {i-1}4})$. 
We might then conduct a dimension reduction of the model by functional principal component analysis.

The conditional heteroscedasticity of the $F_i$ would necessitate a sharper analysis of the time series of conditional covariances $(\Sigma^*_i)_{i\in \mathbb N}$.
We might assume that it follows a particular functional time-series model and treat it as observed rather than latent in the spirit of \cite{ABDL2003}. In the latter case, this is justified by the observation that in the case of continuous semimartingales integrated volatility is the same as the conditional covariance of the increments of the process and is observable under continuous observations. In our case integrated volatility is observable as well by virtue of Theorem \ref{T: LLN for the SARCV} but does not correspond to the conditional covariance of adjusted increments anymore. Fortunately, adjusting our estimator appropriately makes observation of the conditional covariance possible as well. Even better, we can estimate it without imposing the regularity Assumption \ref{As: Spatial regularity}. This result can be found in Corollary \ref{C: Limit theorems for the conditional Covariance} below.

\end{example}
Let us come back to the general setting. For $0\leq U\leq T$,  define 
\begin{equation}
    \int_U^T\Sigma_s^Tds:=\int_U^T \mathcal S(T-s)\Sigma_s\mathcal S(T-s)^* ds.
\end{equation}
In the case that the drift and the volatility are independent of the driving Wiener process this is the conditional covariance of the adjusted increments. I.e.~we have 
$$(Y_T-\mathcal S(T-U)Y_U)|\alpha,\sigma\sim \mathcal N\left(\int_U^T \mathcal S(T-s)\alpha_s ds ,\int_U^T \Sigma_s^T ds\right).$$
In that regard, it is helpful to exploit that the process
\begin{align*}
 Y^T_t:=&\mathcal S(T)Y_0+\int_0^t \mathcal S(T-s)\alpha_sds+\int_0^t \mathcal S(T-s)\sigma_sdW_s\\
 = &\tilde Y_0+\int_0^t \tilde{\alpha}_sds+\int_0^t \tilde{\sigma}_sdW_s,\quad t\in [0,T],
\end{align*}
is a semimartingale on $H$, where $\tilde Y_0:= \mathcal S(T) Y_0$, $\tilde{\alpha}_t=\mathcal S(T-t)\alpha_t$ and $\tilde{\sigma}_t=\mathcal S(T-t)\sigma_t$. Hence, the associated nonadjusted realised covariation is a consistent and asymptotically normal estimator of $\int_0^T\Sigma_s^Tds$. 
Luckily, in the presence of the functional data $(Y_{i\Delta_n},i=1,...,\ulT)$, we can reconstruct the quadratic variation corresponding to $Y^T$ 
by
$$Y_{i\Delta_n}^T-Y_{(i-1)\Delta_n}^T= \mathcal S(T-i\Delta_n) Y_{i\Delta_n}-\mathcal S(T-(i-1)\Delta_n)Y_{(i-1)\Delta_n}.$$
This yields the following limit theorems as a corollary of Theorem \ref{T:Infeasible Central Limit Theorem} and Remark \ref{Rem: Strong solutions and RV vs SARCV}, which do not need Assumption \ref{As: Spatial regularity}:
\begin{corollary}\label{C: Limit theorems for the conditional Covariance}
We have 
\begin{align*}
&\sum_{i=1}^{\ulT} \left(\mathcal S(T-i\Delta_n) Y_{i\Delta_n}-\mathcal S(T-(i-1)\Delta_n)Y_{(i-1)\Delta_n}\right)^{\otimes 2}
\stackrel{ u.c.p.}{\longrightarrow} \int_0^T\Sigma_s^Tds,
\end{align*}
and, if Assumption \ref{As: Integrability Assumption on Drift and Volatility} holds, we also have
\begin{align*}
&\Delta_n^{-\frac 12}\left(\sum_{i=1}^{\ulT} \left(\mathcal S(T-i\Delta_n) Y_{i\Delta_n}-\mathcal S(T-(i-1)\Delta_n)Y_{(i-1)\Delta_n}\right)^{\otimes 2}-\int_0^T\Sigma_s^Tds\right)\\
&\qquad \stackrel{\mathcal L-s}{\Longrightarrow} \mathcal N(0,\int_0^T\mathcal S(T-s)\Sigma_s\mathcal S(T-s)^*(\cdot+\cdot^*)\mathcal S(T-s)\Sigma_s\mathcal S(T-s)^*ds).
\end{align*}
In particular, we obtain that 
\begin{align*}
&\sum_{i=\lfloor U/\Delta_n\rfloor+1}^{\ulT} \left(\mathcal S(T-i\Delta_n) Y_{i\Delta_n}-\mathcal S(T-(i-1)\Delta_n)Y_{(i-1)\Delta_n}\right)^{\otimes 2}
\stackrel{ u.c.p.}{\longrightarrow} \int_U^T\Sigma_s^Tds,
\end{align*}
and under 
Assumption \ref{As: Integrability Assumption on Drift and Volatility} that
\begin{align*}
&\Delta_n^{-\frac 12}\left(\sum_{i=\lfloor U/\Delta_n\rfloor+1}^{\ulT} \left(\mathcal S(T-i\Delta_n) Y_{i\Delta_n}-\mathcal S(T-(i-1)\Delta_n)Y_{(i-1)\Delta_n}\right)^{\otimes 2}-\int_U^T\Sigma_s^Tds\right)\\
&\qquad \stackrel{\mathcal L-s}{\Longrightarrow} \mathcal N(0,\int_U^T\mathcal S(T-s)\Sigma_s\mathcal S(T-s)^*(\cdot+\cdot^*)\mathcal S(T-s)\Sigma_s\mathcal S(T-s)^*ds).
\end{align*}
\end{corollary}

\begin{remark}[Inadequacy of the conditional covariance for dimension reduction]
It should be noted that (conditional) covariances may not be a suitable tool for dimension reduction in situations where the stochastic dynamics imposed by the SPDE should be conserved, unlike in the case of i.i.d. functional data.
 This can be of great importance, as SPDE dynamics often encode important physical or economic principles (such as the absence of arbitrage opportunities in term structure models).

In the energy market, for instance, there is evidence that energy spot prices are not following semimartingale-dynamics (c.f. \cite{BENNEDSEN2017}). Energy spot prices as observed in the market are averages of the lower end of the forward price curve  (see e.g. \cite{BSBK}) and are, thus, bounded linear functionals of these in the Hilbert-space $L^2([0,1])$. This implies in particular, that energy forward curves cannot follow a strong solution to the Heath-Jarrow-Morton-Musiela equation in $L^2(0,1)$ (c.f. section \ref{sec: Term Structure Setting}). Corollary 1 in \cite{filipovic2000b} shows that this excludes the existence of a finite-dimensional submanifold of $L^2(0,1)$ on which the solution to the Heath-Jarrow-Morton-Musiela equation is viable. Hence, given that observed energy spot prices do indeed not follow semimartingale-dynamics, the projection onto a finite-dimensional linear subspace, which is usually done via a functional principal component technique based on the covariance, violates the principle of the absence of arbitrage in the market.

In contrast, the stochastic noise process and, hence, integrated volatility can conveniently be replaced by an approximated and potentially low-dimensional version without harming the stochastic dynamics imposed by the SPDE.
\end{remark}

We next outline how to transform Theorems \ref{T: Central limit theorem for functionals of the quadratic covariation} and \ref{T:Infeasible Central Limit Theorem} (as well as Corollary \ref{C: Limit theorems for the conditional Covariance}) into feasible results. 

\subsection{Feasible central limit theorems for the $SARCV$}\label{sec: Feasible central limit theorems for the (SARCV)}
The central limit Theorems \ref{T: Central limit theorem for functionals of the quadratic covariation} and \ref{T:Infeasible Central Limit Theorem} (and Corollary \ref{C: Limit theorems for the conditional Covariance}) are infeasible in practice, as we do not know the asymptotic variance operator $\Gamma$ a priori. A consistent estimator of this random operator is given by the difference of the corresponding (semigroup-adjusted) fourth power- and the second bipower variation, and therefore it will be possible to derive feasible versions of Theorems \ref{T: Central limit theorem for functionals of the quadratic covariation} and \ref{T:Infeasible Central Limit Theorem}.
For that, we introduce $\hat{\Gamma}^n$ given by
\begin{equation}
    \hat{\Gamma}_t^n:=\Delta_n^{-1}\bigl(SAMPV_t^n(4)-SAMPV_t^n(2,2)\bigr).
\end{equation}
It can be seen by the following laws of large numbers in Theorems \ref{T: Law of large numbers for tensor power variations} and \ref{T: Law of large numbers for Multitensorpower variations} that this defines a consistent estimator of $\Gamma$. I.e.,
 we have in $\mathcal H^4$
\begin{equation}\label{Consistent estimator for the asymptotic variance}
    \hat{\Gamma}^n\stackrel{u.c.p.}{\longrightarrow}\Gamma\quad \text{ as }n\to\infty
\end{equation}
under the following Assumption:
\begin{assumption}\label{As: locally bounded coefficients}
$\alpha$ is locally bounded and $\sigma$ is a c{\`a}dl{\`a}g process w.r.t. $\|\cdot\|_{L_{\text{HS}}(U,H)}$.
\end{assumption}
This assumption corresponds to Assumption $(H)$ in \cite[p.238]{JacodProtter2012}. 
Due to the next result, the estimator $\hat{\Gamma}^n$ behaves well in the sense that it remains in the space of covariance operators:
\begin{lemma}
$\hat{\Gamma}_t^n$ is a symmetric and positive semidefinite nuclear (and therefore Hilbert-Schmidt) operator. 
\end{lemma}
\begin{proof}
That it is a symmetric nuclear operator follows immediately, since it is the difference of two symmetric nuclear operators. Notice that for any real vector $(x_1,...,x_N)$ for some $N\in\mathbb N$ we have
\begin{align*}
    0\leq \sum_{i=1}^{N-1}(x_{i+1}-x_i)^2=  \sum_{i=1}^{N-1}x_{i+1}^2+\sum_{i=1}^{N-1}x_i^2-2 \sum_{i=1}^{N-1}x_{i+1}x_i
    \leq  2\left[\sum_{i=1}^{N}x_i^2- \sum_{i=1}^{N-1}x_{i+1}x_i\right].
\end{align*}
Using this elementary inequality we obtain positive semidefiniteness, since for each $B\in\mathcal H$ 
\begin{align*}
  &\left\langle \Delta_n\hat{\Gamma}_t^n B,B\right\rangle_{\mathcal H}\\
    &= \sum_{i=1}^{\ul} \left\langle(\tilde{\Delta}_i^n Y)^{\otimes 2},B\right\rangle_{\mathcal H}^2 -\sum_{i=1}^{\ul-1} \left\langle(\tilde{\Delta}_i^n Y)^{\otimes 2},B\right\rangle_{\mathcal H}\left\langle(\tilde{\Delta}_{i+1}^n Y)^{\otimes 2},B\right\rangle_{\mathcal H}.
\end{align*}
Hence, $\hat{\Gamma}^n$ is positive semidefinite. 
\end{proof}
The following two results are direct corollaries of the central limit theorems \ref{T: Central limit theorem for functionals of the quadratic covariation} and \ref{T:Infeasible Central Limit Theorem} and the fact that two sequences of random variables defined on the same probability space with values in a Polish space, where one converges stably in law and the other converges in probability, converge jointly stably in law (c.f. \cite[Thm. 3.18 (b)]{Hausler2015}).
We now give the feasible version of the central limit theorem \ref{T: Central limit theorem for functionals of the quadratic covariation}, which can be used to find confidence intervals (e.g. for evaluations in a reproducing kernel Hilbert space setting as in Subsection \ref{sec: discrete data in time and space}):
\begin{corollary}\label{C: feasible CLT for functionals}
 Let Assumption \ref{As: locally bounded coefficients} hold and $B\in \mathcal H$ be an operator with a finite-dimensional range of the form
$B=\sum_{l=1}^K\mu_l h_l\otimes g_l$
for $h_l,g_l\in F_{ 1/2}^{\mathcal S^*}$, $l=1,...,K$, $K\in\mathbb N$. Then 
 \begin{equation*}
    \frac{\Delta_n^{-\frac 12}\langle \tilde X_t^n ,B\rangle_{\mathcal H}}{ \sqrt{\langle\hat{\Gamma}_t B,B\rangle_{\mathcal H}}} \stackrel{d}{\longrightarrow} \mathcal N(0,1),
 \end{equation*}
 conditional on the set $\{\langle\Gamma_t B,B\rangle_{\mathcal H}>0\}\subseteq \Omega$.
\end{corollary}

We also obtain a "feasible" version of Theorem \ref{T:Infeasible Central Limit Theorem}:
\begin{corollary}\label{Feasible central limit theorem for the SARCV}
Under Assumptions \ref{As: Spatial regularity} and \ref{As: locally bounded coefficients}, we obtain
\begin{equation}\label{Joint weak Convergence}
  \left(\Delta_n^{-\frac 12}\tilde{X}_t^n,\hat{\Gamma}_t^n\right)_{t\in[0,T]}\stackrel{\mathcal L-s}{\Longrightarrow} \left(\mathcal N\left(0,\Gamma_t\right), \Gamma_t\right)_{t\in[0,T]},
\end{equation}
where we consider the processes in the space $\mathcal H\times \mathcal H^4$, 
equipped with the metric $$d((B_1,\Psi_1),(B_2,\Psi_2)):= \|B_1-B_2\|_{\mathcal H}+\|\Psi_1-\Psi_2\|_{\mathcal H^4}.$$
\end{corollary}

\subsection{Is the semigroup adjustment necessary?}\label{sec: Necessity of the Adjustment}
Certainly, in many situations, it would be convenient to use the realised quadratic variation instead of the semigroup-adjusted variation.
We shall show below when this is possible but start here with an example where the realised covariation diverges.
\begin{example}\label{Counterexample for RV}
Assume that for an element $e\in H$ such that $\|e\|=1$ and an $H$-valued random variable $X$ the volatility takes the simple form
$$ \sigma_s =  e \otimes \mathcal S(s) X.$$
Moreover, we assume that there is no drift and $Y(0)=0$ and let $X$ (and hence $\sigma_s$) be independent of the driving cylindrical Wiener process $W$ (i.e., no so-called leverage effect). The process $\beta_t:=\langle e, W_t\rangle$ is well-defined and a one-dimensional standard Brownian motion. We obtain
$$Y_t:=\int_0^t \mathcal S(t-s)\sigma_s dW_s= \beta_t \mathcal S(t)X\quad\forall t\in [0,T].$$
 This simple form can be exploited in order to derive counterexamples for the validity of the law of large numbers and the central limit theorem for the quadratic variation. For that, we introduce two cases:
\begin{itemize}
    \item[(i)] (Counterexample for the law of large numbers) $H=L^2[0,2]$, $X(x):=B_x^{\mathfrak H}$, where $B^{\mathfrak H}$ is a fractional Brownian motion with Hurst parameter $\mathfrak H=\frac 14$ and $(\mathcal S(t))_{t\geq 0}$ is the (nilpotent) left-shift semigroup given by
    $$\mathcal S(t) f (x):= f(x+t)\indicator_{[0,2]}(x+t)\quad t\geq 0, x \in [0,2].$$
     \item[(ii)] (Counterexample for the central limit theorem) $H=L^2(\mathbb R)$, $X(x):=\indicator_{[0,1]}(x)$, and $(\mathcal S(t))_{t\geq 0}$ is the left-shift semigroup given by
    $$\mathcal S(t) f (x):= f(x+t)\quad x,t\geq 0.$$
    Observe that in this case Assumptions \ref{As: Integrability Assumption on Drift and Volatility} and \ref{As: Spatial regularity} are satisfied, such that the central limit theorem \ref{T:Infeasible Central Limit Theorem} holds. 
\end{itemize}
We start with the first case and make the following technical observation:
\begin{align*}
\left\|\sum_{i=1}^{n}((\mathcal S(\Delta_n)-I)Y_{(i-1)\Delta_n})^{\otimes 2} \right\|_{\mathcal H}^2
    = &\sum_{i,j=1}^{n}\langle(\mathcal S(\Delta_n)-I)Y_{(i-1)\Delta_n})^{\otimes 2},(\mathcal S(\Delta_n)-I)Y_{(j-1)\Delta_n})^{\otimes 2}\rangle_{\mathcal H}\\
    = &\sum_{i,j=1}^{n}\langle((\mathcal S(\Delta_n)-I)Y_{(i-1)\Delta_n}),((\mathcal S(\Delta_n)-I)Y_{(j-1)\Delta_n})\rangle^2\\
    \geq & \sum_{i=1}^{n}\|(\mathcal S(\Delta_n)-I)Y_{(i-1)\Delta_n})\|^4.
    \end{align*}
Assume now that the realised variation $RV_t^n$ converges in probability to the integrated volatility.    One can show, that  $(RV_t^n-\int_0^t\Sigma_sds-\sum_{i=1}^n ((\mathcal S(\Delta_n)-I)Y_{(i-1)\Delta_n})^{\otimes 2})$ and therefore $\sum_{i=1}^{n}\|(\mathcal S(\Delta_n)-I)Y_{(i-1)\Delta_n})\|^4$ converges in probability to $0$ and that $\sum_{i=1}^{n}\|(\mathcal S(\Delta_n)-I)Y_{(i-1)\Delta_n})\|^4$ is uniformly integrable. This is a technical exercise, which can be found in Appendix \ref{sec: Remaining proofs}. Thus, in the first case, we must necessarily have by Jensen's inequality \begin{align*}
      0= \lim_{n\to\infty}\  \sum_{i=1}^{n}\mathbb E\left[\|(\mathcal S(\Delta_n)-I)Y_{(i-1)\Delta_n})\|^4\right]
      \geq & \lim_{n\to\infty}\  \sum_{i=1}^{n}\mathbb E\left[\|(\mathcal S(\Delta_n)-I)Y_{(i-1)\Delta_n})\|^2\right]^2\\
      = & \lim_{n\to\infty}\Delta_n^{2+4\mathfrak H} \sum_{i=1}^n(i-1)^2>0,
    \end{align*}
    which is a contradiction.
    
Assume now that the realised variation $\sqrt{n}(RV_t^n-\int_0^t\Sigma_sds)$ converges in distribution to a normal distribution.  We now turn to the second example (ii).
In this case, both $\sqrt{n}(RV_t^n-\int_0^t\Sigma_sds)$ and $\sqrt{n}(SARCV_t^n-\int_0^t\Sigma_sds)$ are uniformly integrable, such that their convergence in distribution implies convergence of their means. This is again a technical exercise and the details can be found in Appendix \ref{sec: Remaining proofs}. 
We observe that
\begin{align*}
       \mathbb E\left[RV_t^n-\int_0^t\Sigma_s ds\right] 
    =&  \mathbb E\left[SARCV_t^n
    -\int_0^t\Sigma_s ds\right]+\sum_{i=1}^{\ul}\mathbb E\left[[(\mathcal S(\Delta_n)-I)Y_{(i-1)\Delta_n}]^{\otimes 2}\right].
\end{align*}
Normalising by $\sqrt n$ we find that the first summand converges to $0$, due to the uniform integrability and the central limit theorem \ref{T:Infeasible Central Limit Theorem} (i.e. convergence in distribution to a centred random variable). With the notation $\Delta_i\mathcal S=\mathcal S(i\Delta_n)-\mathcal S((i-1)\Delta_n)$ we find, since $\mathbb E\left[[(\mathcal S(\Delta_n)-I)Y_{(i-1)\Delta_n}]^{\otimes 2}\right]=\int_{0}^{(i-1)\Delta_n} (\Delta_i\mathcal S \indicator_{[0,1]})^{\otimes 2} ds$ that
\begin{align*}
\left\|\mathbb E\left[\sum_{i=1}^{n}[(\mathcal S(\Delta_n)-I)Y_{(i-1)\Delta_n}]^{\otimes 2}\right]\right\|_{\mathcal H}^2
    = &\sum_{i,j=1}^{n}(i-1)(j-1)\Delta_n^2\langle\Delta_i\mathcal S\indicator_{[0,1]},\Delta_j\mathcal S\indicator_{[0,1]}\rangle^2\\
    \geq & \Delta_n^2\sum_{i=1}^{n}(i-1)^2\|\Delta_i\mathcal S\indicator_{[0,1]}\|^4\\
    = &\Delta_n^2\sum_{i=1}^{n}(i-1)^2 2\Delta_n^2.
    \end{align*}
   After normalisation by $n=(\sqrt n)^2$ the expression converges to a positive constant, which verifies that the second case (ii) provides a counterexample for the central limit theorem.
    \end{example}

We can, however, impose assumptions on the regularity of the semigroup on the range of the volatility, such that we again obtain a law of large numbers and a central limit theorem for the realised variations. The assumption for the law of large numbers is 
\begin{assumption}\label{as: strong Spatial regularity}
Let almost surely
$$\lim_{t\to 0}\int_0^T\| t^{-\frac 12}(I-\mathcal S(t))\sigma_s\|_{L_{\text{HS}}(U,H)}^2ds=0.$$
\end{assumption}
\begin{remark}
Assumption \ref{as: strong Spatial regularity} looks similar to Assumption \ref{As: Spatial regularity}. However, in contrast to the weaker Assumption \ref{As: Spatial regularity}, Assumption \ref{as: strong Spatial regularity} excludes some elementary shapes for the volatility such as the one of Example \ref{Counterexample for RV}, for which it is simple to see that $\|(I-\mathcal S(t))\sigma\|_{ L_{\text{HS}}(U,H)}= 2t$.
\end{remark}
Analogously, we obtain a central limit theorem under the following assumption. 
\begin{assumption}\label{as: strong Spatial regularity for the CLT for RV}
Let almost surely
$$\lim_{t\to 0}\int_0^T\| t^{-\frac 34}(I-\mathcal S(t))\sigma_s\|_{L_{\text{HS}}(U,H)}^2ds=0.$$
\end{assumption}
We have the following results.
\begin{theorem}\label{T: Limit Theorems for nonadjusted RV}
\begin{itemize}
    \item[(i)](Law of large numbers) If Assumption \ref{as: strong Spatial regularity} is valid, we have 
\begin{equation}
    RV_t^n\stackrel{u.c.p.}{\longrightarrow}\int_0^t \Sigma_s ds.
\end{equation}
\item[(ii)](Central limit theorem) If Assumptions \ref{As: Integrability Assumption on Drift and Volatility} and \ref{as: strong Spatial regularity for the CLT for RV} are valid, we have 
\begin{equation}
     \Delta_n^{-\frac 12}\left(RV_t^n-\int_0^t \Sigma_s ds\right)\stackrel{\mathcal L-s}{\Longrightarrow}\mathcal N(0,\Gamma_t).
\end{equation}
\end{itemize}
\end{theorem}
We also have a central limit theorem in the weak operator topology as well as a law of large numbers with mild conditions on the functionals:
\begin{theorem}\label{T:Finite dimensional limit theorems for functionals of nonadjusted covariation}\begin{itemize}
    \item[(i)](Law of large numbers) If  $B\in \mathcal H$ is of the form
$B=\sum_{l=1}^K\mu_l h_l\otimes g_l$
for $h_l,g_l\in F_{ 1/2}^{\mathcal S^*}$ for $l=1,...,K$, $K\in\mathbb N$, we have 
      \begin{equation}
    \langle RV_t^n ,B\rangle_{\mathcal H}\stackrel{u.c.p.}{\longrightarrow}\int_0^t \langle\Sigma_s ,B\rangle_{\mathcal H} ds.
\end{equation}
    \item[(ii)](Central limit theorem) If  $B\in \mathcal H$ is of the form
$B=\sum_{l=1}^K\mu_l h_l\otimes g_l$
for $h_l,g_l\in F_{ 3/4}^{\mathcal S^*}$ for $l=1,...,K$, $K\in\mathbb N$ and Assumption \ref{As: Integrability Assumption on Drift and Volatility} holds, we have
\begin{equation}
    \langle\Delta_n^{-\frac 12}\left(RV_t^n-\int_0^t\Sigma_s ds\right),B\rangle_{\mathcal H}\stackrel{\mathcal L-s}{\Longrightarrow} \mathcal N(0,\langle\Gamma_t B,B\rangle_{\mathcal H}).
\end{equation}
\end{itemize}
\end{theorem}

\subsection{Discrete samples in space and time}\label{sec: discrete data in time and space}
We discuss in this subsection the case when we have observations which are discrete in space and time.
Discretisation in space yields many nontrivial challenges (e.g.~owing to asynchronicity or noise).  
Here we want to outline how our results can be used immediately for estimation of the second-order structure of a continuous mild It{\^o} process and therefore we assume throughout this subsection that we have observations of $Y$ on a discrete regular space-time grid. That is, we observe
\begin{equation}\label{fully discrete sampling}
    Y_{i\Delta_n}(j\Delta_n):=Y_{t_i}(x_j), i,j=1,...,n,
\end{equation}
where for notational reasons we fix $T=1$. 
We assume that $H$ is the Sobolev space 
$$H^1(0,1):=\{ h:[0,1]\to\mathbb R: h\text{ is absolutely continuous and } h'\in L^2([0,1])\},$$
 equipped with the norm $\|h\|:=h(0)^2+\int_0^1 (h'(x))^2dx$. This is a reproducing kernel Hilbert space in which the corresponding reproducing kernel is $k(x,y):=1+\min(x,y)$, c.f.~\cite{Berlinet2011}. We  write $\delta_x= k(x,\cdot)$ for both the representer of the evaluation functionals and the evaluation functionals $\delta_x f=f(x)$ in $H$.
 
 Define the operator $\Pi_n:H\to H$ as the orthogonal projection onto $$H_n:=span(\delta_{j\Delta_n}, j=1,...,n).$$ 
 Then, for any $h\in H$, $\Pi_n h$ can readily be recovered from the finite number of evaluations $h(j\Delta_n),\,j=1,...,n$. Indeed, as $\langle \delta_{j\Delta_n},\Pi_n h\rangle=\langle \delta_{j\Delta_n}, h\rangle= h(j\Delta_n),$ $\Pi_n h$ is the unique element in $span(\delta_{j\Delta_n}, j=1,...,n)$ that interpolates the points $h(j\Delta_n),j=1,...,n$. 
 Thus, it is of the form 
 \begin{align}\label{RKHS Projection}
   \Pi_n Y_{i\Delta_n} = \sum_{j=1}^n \alpha_{j,i} k(j\Delta_n,\cdot),
 \end{align}
 where $(\alpha_{1,i},...,\alpha_{n,i})^{\perp}= (\mathbb K_n)^{-1} (Y_{i\Delta_n}(\Delta_n),...,Y_{i\Delta_n}(1))^{\perp}$
 and $\mathbb K_n$ denotes the positive definite matrix $\mathbb K_n =  (k(j_1\Delta_n,j_2\Delta_n))_{j_1,j_2=1,...,n}$. Observe that in this particular case, the kernel matrix has a very simple form as $k(j_1\Delta_n,j_2\Delta_n)=1+\Delta_n\min(j_1,j_2)$ and its inverse is given by the symmetric tridiagonal matrix $\mathbb K_n^{-1}$ which has entries $$(\mathbb K_n^{-1})_{i,j}= \begin{cases}
     -n & |i-j|=1\\
     2n& i=j\notin \{1,n\}\\
     n & i=j=n\\
     2+\frac {n^2-2}{n+1} & i=j=1\\
     0 & |i-j|>1.
 \end{cases}$$
 This method yields the interpolating element in $H$ that is minimal with respect to the norm in $H$ (c.f. \cite[Theorem 58]{Berlinet2011}) and is a very natural choice of reconstructing a curve from discrete data.
 The projections are also suitable for asymptotic theory due to the subsequent lemma.
 \begin{lemma}\label{RKHS projections converge strongly}
 The projections $\Pi_n$ converge strongly to the identity on $H=H^1(0,1)$.
 \end{lemma}
 \begin{proof}
 According to \cite[Theorem 3]{Berlinet2011}  $K_0:=span(\delta_{x}, x\in [0,1])$ is dense in $H^1(0,1)$. 
 For an arbitrary element $h= \sum_{i=1}^{d} \lambda_i \delta_{x_i}\in K_0$ let $\hat h_n= \sum_{i=1}^{d} \lambda_i \delta_{\hat x_i^n}$, where $\hat x_i^n\in \{j\Delta_n, j=1,...,n\}$ which is closest to $x_i$. We then have $\|\delta_{x_i}-\delta_{\hat x_i}\|\leq|x_i-\hat x_i|\leq \Delta_n$ for all $i=1,...,d$ and, thus,
 $\|h-\hat h_n\|\leq \Delta_n \sum_{j=1}^d |\lambda_j|.$
 Now let $h\in H$ and $\epsilon >0$. We can choose a $g\in span(\delta_{x}, x\in [0,1])$ such that $\|h-g\|\leq \frac {\epsilon}2$ and for $g$ we can find an $n_0\in \mathbb N$ such that for each $n\geq n_0$ there is an $h_n\in  span(\delta_{j\Delta_n}, j=1,...,n)$ such that
 $\|g-h_n\|\leq \epsilon/2$. Thus, since $\Pi_n$ is an orthogonal projection, for all $n\geq n_0$ we have
 $$\|(I-P_n) h\|\leq \|h-h_n\|\leq \|h-g\|+\|g-h_n\|\leq \epsilon. $$
 \end{proof}
 
  Let us now derive asymptotic results in the fully discrete setting \eqref{fully discrete sampling}.
 We outline the situation here in two cases, which are of practical importance and well-suited for these observations. In the first case, we have a continuous It{\^o} semimartingale in $H$. This covers suitable frameworks for intraday energy markets, as mentioned in the introductory section.  In the second case, $\mathcal S$ is the semigroup of left shifts, which for instance corresponds to the framework of Heath-Jarrow-Morton term structure models, c.f.~\cite{Filipovic2000}, for interest rates and for energy forward markets, c.f.~\cite{BenthKruhner2014}. For a different sampling scheme we will also include a short discussion on the stochastic heat equation in a separate subsection afterwards.
\begin{itemize}
    \item[(a)](Semimartingale case) The semigroup is equal to the identity (or can be interpreted as such in the case of a strong solution as in Remark \ref{Rem: Strong solutions and RV vs SARCV}). That is, we observe a continuous It{\^o} semimartingale 
$$Y_t=Y_0+\int_0^t \alpha _s ds+\int_0^t \sigma_s dW_s.$$
In that case, we define the operator
$$\hat{\Sigma}^n_t= \Pi_n RV_t^n \Pi_n= \sum_{i=1}^{\ul} (\Pi_n \Delta_i^n Y)^{\otimes 2}.$$
The latter is feasible, as we can derive the values $\Delta_i^n Y(j\Delta_n)=Y_{i\Delta_n}(j\Delta_n)-Y_{(i-1)\Delta_n}(j\Delta_n)$ from data and, hence, can derive $\Pi_n \Delta_i^n Y$ by \eqref{RKHS Projection}.
\item[(b)](Shift case) $\mathcal S$ is the semigroup of left shifts, given by $$\mathcal S(t)h(x):=\begin{cases} h(x+t), & x+t\leq 1,\\
h(1), & x+t>1,
\end{cases}$$ which forms a the strongly continuous semigroup on $H^1(0,1)$. 
In that case, we define the operator
$$\hat{\Sigma}^n_t= \Pi_n SARCV_t^n \Pi_n= \sum_{i=1}^{\ul} (\Pi_n \tilde{\Delta}_i^n Y)^{\otimes 2}.$$
The latter is feasible, as we can derive the values $\tilde{\Delta}_i^n Y(j\Delta_n)=Y_{i\Delta_n}(j\Delta_n)-Y_{(i-1)\Delta_n}((j+1)\Delta_n)$ for $j=1,...,n-1$ and $\Delta_i^n Y(1)=0$ (by the definition of the semigroup) from data and, hence, can derive $\Pi_n \tilde{\Delta}_i^n Y$ by \eqref{RKHS Projection} also in this case.
\end{itemize}
The proof of the next result makes  use 
of Theorem \ref{T: LLN for the SARCV}. 
\begin{lemma}\label{L: Consistency of discretised truncated SARCV}
In both cases (a) and (b), we have
$$\hat{\Sigma}_t^n\overset{u.c.p.}{\longrightarrow} \int_0^t \Sigma_s ds,$$ with respect to the Hilbert-Schmidt norm on $\mathcal H=L_{\text{HS}}(H^1(0,1))$.
\end{lemma}
\begin{proof}
Let $A_n$ denote either $ RV_t^n$ in case (a) or $ SARCV_t^n$ in case (b). Then it is
$$\|\Pi_n A_n \Pi_n- \Pi_n \int_0^t \Sigma_s ds\Pi_n\|_{\mathcal H}\leq \|A_n-\int_0^t \Sigma_s ds\|_{\mathcal H},$$
which converges to $0$ uniformly on compacts in probability in both cases by Theorem \ref{T: LLN for the SARCV}. 
Moreover, $\Pi_n  \Sigma_s  \Pi_n$ converges to $ \Sigma_s$ with respect to the nuclear (and hence the Hilbert-Schmidt) norm for all $s\in [0,1]$, which follows by Lemma \ref{RKHS projections converge strongly} and combining Proposition 4 and Lemma 5 in \cite{Panaretos2019}. The u.c.p. convergence follows by dominated convergence as
$$\sup_{t\in [0,1]}\left\|\int_0^t\Pi_n  \Sigma_s  \Pi_n-\Sigma_s ds\right\|_{\mathcal H}\leq \int_0^1\left\|\Pi_n  \Sigma_s  \Pi_n-\Sigma_s\right\|_{\mathcal H}ds. $$
\end{proof}
Due to the semimartingale property of the processes $(Y_t(x))_{t\in [0,T]}$ in case (a), both by the finite-dimensional limit theory outlined in \cite{JacodProtter2012} or by appealing to Theorem \ref{T: Central limit theorem for functionals of the quadratic covariation} we have the following result.
\begin{corollary}
 In case (a), for $x\in [0,1]$, we have
\begin{align*}
  \sqrt n \left(\sum_{i=1}^{\ul} \left(Y_{i\Delta_n}(x)-Y_{(i-1)\Delta_n}(x)\right)^2-\int_0^t \langle \sigma_s,\delta_x\rangle^2ds\right)\stackrel{\mathcal L-s}{\Longrightarrow} \mathcal N(0,\langle\Gamma_t \delta_{x}^{\otimes 2},\delta_{x}^{\otimes 2}\rangle_{\mathcal H}).
\end{align*}
A feasible version,  conditional on the set $\{\langle\Gamma_t \delta_{x}^{\otimes 2},\delta_{x}^{\otimes 2}\rangle_{\mathcal H}>0\}\subseteq \Omega$,  is given by
\begin{align*}
   &\left(\sum_{i=1}^{n}\left(Y_{i\Delta_n}(x)-Y_{(i-1)\Delta_n}(x)\right)^4\right.\\
   &\left.-\sum_{i=1}^{n-1}\left(Y_{i\Delta_n}(x)-Y_{(i-1)\Delta_n}(x)\right)^2\left(Y_{(i+1)\Delta_n}(x)-Y_{(i)\Delta_n}(x)\right)^2\right)^{-\frac 12}\\
   &\times\left(\sum_{i=1}^{n} \left(Y_{i\Delta_n}(x)-Y_{(i-1)\Delta_n}(x)\right)^2-\int_0^t \langle \sigma_s, \delta_x\rangle^2ds\right)\\
   &\stackrel{d}{\longrightarrow} \mathcal N(0,1).
\end{align*}
\end{corollary}
It's notable that the central limit theorem can be recovered in case (b) as well, due to the following observation: In the case that $H=H^1(0,1)$, the representations $\delta_x$ of evaluation functionals are in the $\frac 12$-Favard spaces of the shift semigroup and its dual. Namely, we have 
\begin{lemma}\label{L: Evaluation functionals in Sobolev spaces have 1/2 regularity}
Let $H=H^1(0,1)$ and $\mathcal S$ be the left shift semigroup. Then 
the representations $\delta_x$, for any $0\leq x \leq 1$, of the evaluation functionals  are elements in the Favard class $F_{1/2}^{\mathcal S}$ and $F_{1/2}^{\mathcal S^*}$, but for $x\in (0,1]$ not in  the $\gamma$-Favard spaces  $F_{\gamma}^{\mathcal S}$ and for $x\in [0,1)$ $F_{\gamma}^{\mathcal S^*}$ with respect to the shift semigroup for $\gamma >\frac 12$. 
\end{lemma}
Let us assume for the moment we are in case (b) for the process
\begin{align*}
    Y_t(x)=&Y_0(x+t)+\int_0^t \alpha_s(x+t-s) ds \int_0^t \langle\sigma_s, \delta_{x+t-s}\rangle dW_s.
\end{align*}
This leads to the following useful limit theorem, which enables us to find confidence bounds for the process $(\int_0^t  \langle\sigma_s, \delta_{x}\rangle^2ds)_{t\in [0,T]}$ based on observations $(Y_{i\Delta_n}(x),Y_{i\Delta_n}(x+\Delta_n)),i=1,...,n$ in case (b):
\begin{corollary} In case (b), we have, for  $x\in [0,1]$, due to the central limit Theorem \ref{T: Central limit theorem for functionals of the quadratic covariation} (respectively, Theorem \ref{C: feasible CLT for functionals})
\begin{align*}
  \sqrt n \left( \sum_{i=1}^{\ul} \left(Y_{i\Delta_n}(x)-Y_{(i-1)\Delta_n}(x+\Delta_n)\right)^2-\int_0^t \langle\sigma_s, \delta_{x}\rangle^2ds\right)
   \stackrel{\mathcal L-s}{\Longrightarrow} \mathcal N(0,\langle\Gamma_t \delta_{x}^{\otimes 2},\delta_{x}^{\otimes 2}\rangle).
\end{align*}
A feasible version,  conditional on the set $\{\langle\Gamma_t \delta_{x}^{\otimes 2},\delta_{x}^{\otimes 2}\rangle_{\mathcal H}>0\}\subseteq \Omega$, is given by 
\begin{align*}
   &\left(\sum_{i=1}^{n}\left(Y_{i\Delta_n}(x)-Y_{(i-1)\Delta_n}(x+\Delta_n)\right)^4\right.\\
   &\left.-\sum_{i=1}^{n-1}\left(Y_{i\Delta_n}(x)-Y_{(i-1)\Delta_n}(x+\Delta_n)\right)^2\left(Y_{(i+1)\Delta_n}(x)-Y_{i\Delta_n}(x+\Delta_n)\right)^2\right)^{-\frac 12}\\
   &\times\left(\sum_{i=1}^{n} \left(Y_{i\Delta_n}(x)-Y_{(i-1)\Delta_n}(x+\Delta_n)\right)^2-\int_0^t  \langle\sigma_s, \delta_{x}\rangle^2ds\right)\\
   &\stackrel{d}{\longrightarrow} \mathcal N(0,1).
\end{align*}
\end{corollary}
 We remark, that even for case (b), Lemma \ref{L: Evaluation functionals in Sobolev spaces have 1/2 regularity} also guarantees that Theorem \ref{T:Finite dimensional limit theorems for functionals of nonadjusted covariation}(i) applies. Hence, it holds that
$$
\sum_{i=1}^n\left(Y_{i\Delta_n}(x)-Y_{(i-1)\Delta_n}(x)\right)^2\stackrel{u.c.p.}{\longrightarrow} \int_0^t  \langle\sigma_s, \delta_{x}\rangle^2 ds.
$$
Therefore, we just need observations $Y_{i\Delta_n}(x),{i=1,...,n}$ to estimate the quadratic variation of the one-dimensional processes $(Y_t(x))_{t\in [0,T]}$ consistently.


\subsubsection{A note on the stochastic heat equation}\label{sec: Heat Equation}
As already mentioned in Remark \ref{Rem: Strong solutions and RV vs SARCV}, the semigroup adjustment can be easily implemented in cases in which we know the semigroup and it has a simple form, which is not always the case. A prototypical example is the stochastic heat equation with an unknown diffusivity $\kappa>0$ taking the form
\begin{align*}
    dY_t=\kappa \partial_{xx} Y_t dt+ Q^{\frac 12}dW_t.
\end{align*}
Here we assume that $\int_0^t Q^{\frac 12}dW_s$ is formally a $Q$-Wiener process taking values in  $H=L^2[0,1]$ with an unknown nuclear covariance operator $Q$. The differential operator $\partial_{xx}$ on the domain $D(\partial_{xx})= \{h\in L^2[0,1]: \|f'\|+\|f''\|<\infty, f(0)=f(1)=0\}$ generates an analytic semigroup on $H$ given by
\begin{align*}
    \mathcal S(t) f= \sum_{j=1}^{\infty} e^{t\lambda_j} \langle e_j, f\rangle e_j,
\end{align*}
where $\lambda_j=\pi^2 j^2 \kappa$ and $e_j(x):=\sqrt 2 \sin (\pi j x)$ (see for instance Example B.12 in \cite{PZ2007}). 
 In this situation, the regularity of the dynamics is very often expressed in terms of Sobolev spaces, which can be formally defined as
\begin{align}
    \dot H^r := D\left(\partial_{xx}^{\frac r2}\right)= \left\{h\in H: \|h\|_{\dot H^r}^2:=\sum_{j=1}^{\infty} \lambda_j^{ r} \langle e_j, h\rangle^2 <\infty\right\}.
\end{align}
Equipped with the norm $\|\cdot \|_{\dot H^r}=\|(-\mathcal A)^{\frac r2} \cdot \|$, these are separable Hilbert spaces. Now, if $W$ is a cylindrical Wiener process on $L^2(0,1)$ and \begin{equation}\label{Heat-Favard condition}
Q^{\frac 12}\in L_{\text{HS}}\left(L^2(0,1),\dot H^r\right),
\end{equation}
it follows by Theorem 6.13 in Section 2.6 of \cite{pazy1983}
\begin{align}\label{Favard-Property for Sobolev Volatility}
     \sup_{t\in [0,T]}t^{-\frac r2}\|(\mathcal S(t)-I)Q^{\frac 12}\|_{L_{\text{HS}}(U,H)} =& \sup_{t\in [0,T]}t^{-\frac r2}\|A^{-\frac r2}(\mathcal S(t)-I)A^{\frac r2}Q^{\frac 12}\|_{L_{\text{HS}}(U,H)} \notag\\
    \leq & C \|A^{\frac r2}Q^{\frac 12}\|_{L_{\text{HS}}(U,H)}\notag \\
    = & C \|Q^{\frac r2}\|_{L_{\text{HS}}(U,\dot H^r)}<\infty.
\end{align}
This yields
\begin{lemma}\label{lem: Structure Lemma for Heat equation}
If in \eqref{Heat-Favard condition} we have
\begin{itemize}
    \item[(a)] $r=1$, then Assumption \ref{As: Spatial regularity} holds and the semigroup-adjusted realised covariation satisfies the infinite-dimensional central limit theorem \ref{T:Infeasible Central Limit Theorem};
    \item[(b) ]$r>1$, then Assumption \ref{as: strong Spatial regularity} holds and the realised variation satisfies the infinite-dimensional law of large numbers Theorem \ref{T: Limit Theorems for nonadjusted RV}(i);
    \item[(b) ]$r>\frac 32$, then Assumption \ref{as: strong Spatial regularity for the CLT for RV} holds and the realised variation satisfies the infinite-dimensional central limit theorem \ref{T: Limit Theorems for nonadjusted RV}(ii).
\end{itemize}
\end{lemma}
As we do not necessarily know $\kappa$, it might not be possible to implement the semigroup adjustment. Even if we knew  $\kappa,$ on the basis of discrete observations we would need to approximate the semigroup appropriately to implement the adjustment such as it is done in \cite{Hildebrandt2021b}.
In this regard, cases (b) and (c) of the previous theorem are particularly appealing, as they hold for the realised variation, which does not take into account an adjustment by the semigroup. Still, also the latter has to be approximated by discrete data. Here we assume that we sample data from the mild solution to the stochastic heat equation as local averages, that is,
we have
$$\bar{y}_{i,j}^{n,m}:=\frac 1{\Delta_m}\int_{(j-1)\Delta_m}^{j\Delta_m} Y_{i\Delta_n}(x)dx,\quad i=0,...,n,j=1,...,m.$$
Observe that in this case, we can have a different spatial and temporal resolution.
Let $\Pi_m$ denote the projection onto the subspace of $L^2[0,1]$ spanned by the orthonormal vectors $\Delta_m\indicator_{[(j-1)\Delta_m,j\Delta_m]}$. 
Then we can recover $\Pi_m \Delta_i^m Y$ from data as this is simply corresponding to the piecewise constant function given by
$$\Pi_m \Delta_i^m Y = \sum_{i=1}^m (\bar{y}_{i,j}^{n,m}-\bar{y}_{i-1,j}^{n,m})\indicator_{[(j-1)\Delta_m,j\Delta_m]}.$$
We can, thus, readily  derive the estimator
$$\hat{\Sigma}_t^{n,m}:=\Pi_m RV_t^n \Pi_m = \sum_{i=1}^{\ul}(\Pi_m \Delta_i^n Y)^{\otimes 2},$$
from data as well. For a sufficiently regular $Q$, we then obtain an infinite-dimensional law of large numbers:
\begin{lemma}\label{lem: Convergence of RV for Heat equation under Favard condition}
Assume \eqref{Heat-Favard condition} holds with $r>1$. Then $\hat{\Sigma}_t^{n,m}$ is a consistent estimator, that is, with respect to the Hilbert-Schmidt norm it is as $n,m\to\infty$
\begin{align*}
    \hat{\Sigma}_t^{n,m}\overset{u.c.p.}{\longrightarrow} tQ.
\end{align*}
\end{lemma}
\begin{proof}
We have
\begin{align*}
 \| \hat{\Sigma}_t^{n,m}-\Pi_m tQ\Pi_m\|\leq  & \|RV_t^n -t Q\|,
\end{align*}
which converges to $0$ by Lemma \ref{lem: Structure Lemma for Heat equation} (b) as $n\to\infty$. As $\Pi_m\to I$ strongly in $L^2(0,1)$ as $m\to\infty$ we also have that
$\|\Pi_m (tQ)\Pi_m- tQ\|_{L_{\text{HS}}(L^1(0,1))}$ converges to $0$ as $m\to\infty$ by combining Proposition 4 and Lemma 5 in \cite{Panaretos2019}.
\end{proof}
We  may also derive a central limit theorem for the one-dimensional observations.
\begin{lemma}
Assume that \eqref{Heat-Favard condition} holds with $r>3/2$ and that $m=m_n$ with $\lim_{n\to\infty} n\Delta_{m_n}=0$. Then for all $h\in H$ it is
\begin{align*}
    \sqrt n \langle\left( \hat{\Sigma}_t^{n,m} -tQ\right)h,h\rangle\overset{\mathcal L-s}{\Longrightarrow} \mathcal N\left(0, 2t \langle Qh,h\rangle^2\right).
\end{align*}
\end{lemma}
\begin{proof}
We decompose
\begin{align*}
   & \sqrt n \left( \hat{\Sigma}_t^{n,m} -tQ\right)\\
    = & \sqrt n \left( RV_t^n -tQ\right)+\left(  \sqrt n \left( \hat{\Sigma}_t^{n,m} -\Pi_mtQ\Pi_m\right)-\sqrt n \left( RV_t^n -tQ\right)\right)+\sqrt n \left( \Pi_m tQ\Pi_m-tQ\right).
\end{align*}
The first term converges stably in law to the limiting Gaussian process as specified in the assertion as $n\to\infty$.
It, thus, remains to show that the other two terms converge to $0$ uniformly on compacts in probability.

For the second summand we denote $h_m = \Pi_m h$ and find that
\begin{align*}
 & \sqrt n \left|\langle \left( \hat{\Sigma}_t^{n,m} -\Pi_m tQ \Pi_m\right) h,h\rangle-\sqrt n \langle \left( RV_t^n - tQ \right) h,h\rangle\right|\\
  \leq & \sqrt n  \left\| RV_t^n - tQ \right\| \|h_m-h\| (\|h_m\|+\|h\|). 
\end{align*}
As the first factor is bounded in probability and $h_m\to h$ as $m\to\infty$, this converges to $0$. For the second summand we have that $\sqrt n(\Pi_m Q\Pi_m -Q)$
we find
\begin{align*}
  \langle\Pi_m Q\Pi_m h -Qh,h\rangle\leq \|(I-\Pi_m)Qh\|+  \|(I-\Pi_m)Qh_m\|= (1)_m+(2)_m. 
\end{align*}
For the first summand we can argue that as $Q$ maps into $$\dot{H}^{\frac 32}\subset\dot{H}^{1}\subset H^1(0,1)$$
by Lemma 3.1 in \cite{thomee2007}, we have that for any $h\in H$ that $\partial_x Qh=(Qh)'\in L^2(0,1)$ and  $(Qh_m)'\in L^2(0,1)$ as well. Hence, for
$Qh^*_m(\cdot):=\sum_{i=1}^m Qh(i\Delta_m) \indicator_{[(i-1)\Delta_m,i\Delta_m]}(\cdot)\in span(\indicator_{[(i-1)\Delta_m,i\Delta_m]}:i=1,...,m)$ we have 
\begin{align*}
    (1)_m^2\leq\|Qh-Qh^*_m\|^2= \left\|\sum_{i=1}^m\left(\int_x^{i\Delta_m}  (Qh)'(y)dy\right) \indicator_{[(i-1)\Delta_m,i\Delta_m]}(x)\right\|^2\leq & \Delta_m \|(Qh)'\|^2
\end{align*}
and in the same way and using Lemma 3.1 in \cite{thomee2007}
$$ (2)_m^2\leq  \Delta_m \|(Qh_m)'\|^2= \Delta_m \|\partial_{xx}^{\frac 12}Qh_m\|^2\leq  \Delta_m \|Q\|_{L_{\text{HS}}(L^2(0,1),\dot H^1)}^2\|h_m\|^2 .$$
Summing up, we get
\begin{align*}
  \sqrt n\langle\Pi_m Q\Pi_m h -Qh,h\rangle\leq & \|(I-\Pi_m)Qh\|+  \|(I-\Pi_m)Qh_m\|\\
  = &\sqrt n \sqrt{\Delta_m}\left(\|(Qh)'\|+\|Q\|_{L_{\text{HS}}(L^2(0,1),\dot H^1)}\|h_m\|\right). 
\end{align*}
This converges to $0$ as $\sqrt {n\Delta_m}\to 0$ as $n\to \infty$ by assumption.
\end{proof}
Analytic semigroups such as the heat semigroups can impose regularity on the sample paths of $Y$ and potentially allow to weaken the conditions of Lemma \ref{lem: Structure Lemma for Heat equation}, which may not be sharp in this setting. 
We postpone a thorough analysis of these conditions in case of analytic semigroups to future research.

\section{A law of large numbers for multipower variations}\label{sec: Limit theorems for multipower variations}
We still have to verify the consistency \eqref{Consistent estimator for the asymptotic variance} of the estimator for the asymptotic variance $\Gamma_t$. Rather than proving only this specific result, we provide general laws of large numbers for power and multipower variations in this section.

For a positive symmetric trace-class operator $\Sigma$, we define the operator $\rho_{\Sigma}(m)$, as the  $m$th \textit{tensor moment} of an $H$-valued random variable $U\sim \mathcal N(0,\Sigma)$, i.e.,  
\begin{equation}
    \rho_{\Sigma}(m)=\mathbb E[U^{\otimes m}].
\end{equation}
 This operator can be characterised by the identity
\begin{equation}\label{Power variation limit operator}
    \langle \rho_{\Sigma_s}(m),h_1\otimes...\otimes h_m\rangle_{\mathcal H^m} = \sum_{p\in\mathcal P(m)}\prod_{(x,y)\in p} \langle \Sigma h_x,h_y\rangle,
\end{equation}
for any collection $h_1,...,h_m\in H$, where the sum is taken over all pairings over $(1,...,m)$, i.e., all ways to disjointly decompose $(1,...,m)$ into pairs. We denote the set of all these pairings by $\mathcal P(m)$, which is then given as
\begin{align*}
\mathcal P(m)
=&\left\{p\subset \lbrace 1,...,m\rbrace^2: \# p=\frac m2 \text{ and if }(x,y),(x',y')\in p,\right.\\
&\qquad\left.\text{ then $x,y,x',y'$ are pairwise unequal and }x< y,x'<y' \right\}.
\end{align*}
In particular, $\rho_{\Sigma}(m)=0$, if $m$ is odd. In the case of power variations, we need 
\begin{assumption}[m]\label{As: for LLN of Power Variations} For a natural number $m\in \mathbb N$
we have 
\begin{equation}
    \mathbb P\left[\int_0^T\|\alpha_s\|^{\frac {2m}{2+m}}ds+\int_0^T \|\sigma_s\|_{L_{\text{HS}}(U,H)}^m ds <\infty\right]=1.
\end{equation}
\end{assumption}
Observe that the assumption above corresponds to Condition 3.4.6 in the finite-dimensional law of large numbers Theorem 3.4.1 in \cite{JacodProtter2012}. 
We now state a law of large numbers for semigroup-adjusted power variations:
\begin{theorem}\label{T: Law of large numbers for tensor power variations}
Let $m\geq2$ be a natural number and Assumption \ref{As: for LLN of Power Variations}(m) be valid.
Then 
\begin{align*}
 \Delta_n^{1-\frac m2}   SAMPV^n(m)\stackrel{u.c.p.}{\longrightarrow} \left(\int_0^t \rho_{\Sigma_s}(m)ds\right)_{t\in [0,T]},
\end{align*}
with respect to the Hilbert-Schmidt norm on $\mathcal H^m$.
\end{theorem}
Let us study some examples:
\begin{example}\label{Ex: LLN for SARCV as a corollary of the LLN for power variations}
If $m=2$, there is just one way to decompose $\{1,2\}$ into pairs, i.e., $\mathcal P(2)$ consists of the pair $\{(1,2)\}$ only. Therefore $\rho_{\Sigma_s}(2)= \Sigma_s$, and
in particular, the law of large numbers reads in this case
\begin{align*}
    SARCV_t^n(2)\stackrel{u.c.p.}{\longrightarrow} \int_0^t \Sigma_s ds,
\end{align*}
which corresponds to the law of large numbers Theorem \ref{T: LLN for the SARCV}.
\end{example}
\begin{example}
If $m=4$, then we find that $\mathcal P(4)$ consists of the pairs $\{(1,2),(3,4)\}$, $\{(1,3),(2,4)\}$ and $\{(1,4),(2,3)\}$. Hence, it follows,
 \begin{align*}
      &\langle \rho_{\Sigma_s}(4),h_1\otimes...\otimes h_4\rangle_{\mathcal H^4}\\
      &\qquad=\langle \Sigma_s h_1,h_2\rangle  \langle \Sigma_s h_3,h_4\rangle +\langle \Sigma_s h_1,h_3\rangle  \langle \Sigma_s h_2,h_4\rangle+\langle \Sigma_s h_1,h_4\rangle  \langle \Sigma_s h_2,h_3\rangle\\
      &\qquad=\langle \Sigma_s^{\otimes 2}+\Sigma_s (\cdot+\cdot^*)\Sigma_s,h_1\otimes h_2\otimes h_3\otimes h_4\rangle.
 \end{align*}
 This yields
 $\rho_{\Sigma_s}(4)=\Sigma_s (\cdot+\cdot^*)\Sigma_s+\Sigma_s^{\otimes 2}.$
\end{example}
For a positive symmetric trace class operator $\Sigma:H\to H$, define for $m,m_1,...,m_k\in\mathbb N$ such that $m=m_1+...+m_k$
$$\rho_{\Sigma}^{\otimes k}(m_1,...,m_k):= \bigotimes_{j=1}^k \rho_{\Sigma}(m_j),$$
which is an operator in $\mathcal H^m$, such that for any collection $(h_{j,l})\subset H$, $j=1,...,k$ and $l=1,...,m_j$
we have 
\begin{align*}
    \langle \rho_{\Sigma}^{\otimes k}(m_1,...,m_k),\bigotimes_{l=1}^{m_1}h_{1,l} \otimes...\otimes \bigotimes_{l=1}^{m_k}h_{k,l}\rangle_{\mathcal H^m}
    = & \prod_{j=1}^{k}\sum_{p\in\mathcal P(m_l)}\prod_{(x,y)\in p} \langle \Sigma_s h_{x,j},h_{y,j}\rangle.
\end{align*}
We have the following law of large numbers for multipower variations:
\begin{theorem}\label{T: Law of large numbers for Multitensorpower variations}
Let Assumption \ref{As: locally bounded coefficients} hold and $m,m_1,m_2,\ldots,m_k$ be natural numbers such that $m_1+...+m_k=m\geq 2$. Then 
\begin{align}\label{ucp convergence of multipower variations}
 \Delta_n^{1-\frac m2}   SAMPV^n(m_1,...,m_k)\stackrel{u.c.p.}{\longrightarrow} \left(\int_0^t \rho_{\Sigma_s}^{\otimes k}(m_1,...,m_k)ds\right)_{t\in [0,T]}.
\end{align}
\end{theorem}
Let us consider the important example of bipower variation:
\begin{example}[Bipower variation]
Let $m_1=m_2=k=2$, i.e., $m=4$, and define the bipower variation
\begin{equation}
   SAMPV_t^n(2,2)= \sum_{i=1}^{\ul-1} \tilde{\Delta}_i^n Y^{\otimes 2}\otimes \tilde{\Delta}_{i+1}^n Y^{\otimes 2}.
\end{equation}
Observe that $\rho_{\Sigma_s}^{\otimes 2}(2,2)=\rho_{\Sigma_s}\otimes\rho_{\Sigma_s}=\Sigma_s^{\otimes 2}$ by Example \ref{Ex: LLN for SARCV as a corollary of the LLN for power variations}.
\end{example}

\section{Outline of the Proofs}\label{sec: Outline of Proofs}
We will now provide an outline of the proofs of the main results (i.e. Theorems \ref{T: LLN for the SARCV}, \ref{T: Central limit theorem for functionals of the quadratic covariation}, \ref{T:Infeasible Central Limit Theorem}, \ref{T: Law of large numbers for tensor power variations} and \ref{T: Law of large numbers for Multitensorpower variations}). The remaining results Theorem \ref{T: Limit Theorems for nonadjusted RV}, Theorem \ref{T:Finite dimensional limit theorems for functionals of nonadjusted covariation}, Lemma \ref{L: Consistency of discretised truncated SARCV} and Lemma \ref{L: Evaluation functionals in Sobolev spaces have 1/2 regularity} as well as Examples \ref{Ex: counterexample for the Spatial regularity assumption} and \ref{Counterexample for RV} are consequences of these limit theorems. The detailed proofs are relegated to the Appendices \ref{sec: technical tools}-\ref{sec: Proofs of CLTs}.  

Throughout this section, we let $p_N$ be the projection onto $v^N:=\overline{lin(\{e_j:j\geq N\})}$, for some orthonormal basis $(e_j)_{j\in\mathbb N}$ that is contained in $D(\mathcal A^*)$, and $P_N^m$ denote the projection onto $\overline{lin(\{ \bigotimes_{l=1}^{m} e_{k_l}: k_l\geq N\})}$ (where $m$ is variable, corresponding to the particular case). In the special case $m=2$ we write $P_N^2=:P_N$.

First, it is important to note that we can appeal to localised versions of the assumptions of Theorems \ref{T: LLN for the SARCV}, \ref{T: Central limit theorem for functionals of the quadratic covariation}, \ref{T:Infeasible Central Limit Theorem}, \ref{T: Law of large numbers for tensor power variations} and \ref{T: Law of large numbers for Multitensorpower variations}. This is a common procedure that  follows the arguments presented in Section 4.4.1 in \cite{JacodProtter2012}, which enables us to prove all theorems stated in this work under such localised assumptions. The localised assumptions essentially impose boundedness instead of almost sure finiteness, in order to ensure the existence of all necessary moments.

 The first important observation is the following: By the localisation procedure, we can assume  there is a constant $ A$, such that 
\begin{equation}\label{Very very Weak localised integrability Assumption on the moments}
   \int_0^T \|\alpha_s\|^{\frac m2}+\|\sigma_s \|_{L_{\text{HS}}(U,H)}^m ds < A.
\end{equation} 
In this case, the $SAMPV$, when projected onto functionals of the form $\bigotimes_{l=1}^m e_{j_l}$, for an orthonormal basis $(e_j)_{j\in \mathbb N}$ which is contained in $D(\mathcal A^*)$, $m\in\mathbb N$ and $j_1,...,j_m\in \mathbb N$, corresponds asymptotically to the tensor multipower variations of the semimartingale 
$$S_t:= \int_0^t \alpha_s ds+ \int_0^t \sigma_s dW_s.$$
We find that
\begin{align}\label{FINITE DIMENSIONAL DISTRIBUTIONS ARE ESSENTIALLY SEMIMARTINGALE POWER VARIATIONS}
 & \langle SAMPV_t^n(m_1,...,m_k),\bigotimes_{l=1}^m e_{j_l}\rangle_{\mathcal H^m}\notag\\
 &\qquad=\sum_{i=1}^{\ul}  \langle\bigotimes_{j=1}^k\Delta_{i+j-1}^n S^{\otimes m_j}, \bigotimes_{l=1}^m e_{j_l}\rangle_{\mathcal H^m}+ \mathcal O_p( \Delta_n^{\frac m2}).
\end{align}
As the left-hand side of \eqref{FINITE DIMENSIONAL DISTRIBUTIONS ARE ESSENTIALLY SEMIMARTINGALE POWER VARIATIONS} corresponds to a multivariate continuous semimartingale, the limit theorems from \cite{JacodProtter2012} are readily available.

Now we come to the second important observation: For that, define the two sequences
\begin{equation}\label{outline: volatility and drift have equismall tails}
    a_N(z):=\sup_{n\in\mathbb N}\mathbb E\left[\int_0^T\|p_N\alpha_s^{\mathcal S_n}\|_{\mathcal H}^zds\right],\quad b_N(z):=\sup_{n\in\mathbb N}\mathbb E\left[\int_0^T\|p_N\sigma_s^{\mathcal S_n}\|_{L_{\text{HS}}(U,H)}^zds\right],
\end{equation}
for $z\leq m$,
$\sigma_s^{\mathcal S_n}=\mathcal S(i\Delta_n-s)\sigma_s$
and
$\alpha_s^{\mathcal S_n}=\mathcal S(i\Delta_n-s)\alpha_s$
with $s\in((i-1)\Delta_n,i\Delta_n]$. Observe that
\begin{equation*}
  \Sigma_s^{\mathcal S_n}= \sigma_s^{\mathcal S_n}(\sigma_s^{\mathcal S_n})^*.
\end{equation*}
Under \eqref{Very very Weak localised integrability Assumption on the moments}
both $a_N(z)$ for $z\leq m/2$ and $b_N(z)$ for $z\leq m$ converge to $0$ as $N\to\infty$ for $z\leq m$, respectively $z\leq \frac m2$. Moreover, we can find for all $m\in\mathbb N$ a universal constant $C>0$ possibly depending on $m$, such that
\begin{equation}\label{outline: whole increment estimate}
 \sum_{i=1}^{\ul}   \mathbb E\left[\|p_N\tilde{\Delta}_i^n Y\|^m\right]\leq C\Delta_n^{\frac m2-1}(a_N(\frac m2)+ b_N(m))=o(\Delta_n^{\frac m2-1}).
\end{equation}
We notice that the Hilbert-Schmidt structure of the volatility is crucial to establish that $b_N(z)$ converges to $0$.

The proofs for limit theorems in this work follow a similar pattern.
For the laws of large numbers:
\begin{itemize}
        \item[(LLNa)] Show that $(\Delta_n^{1-\frac m2}(I-P_N^m)(SAMPV_t^n-\int_0^t \rho^{\otimes k}(m_1,...,m_k)ds))_{t\in [0,T]}$ converges for all $N\in\mathbb N$ to 
        $0$ as $n\to \infty$, due to the available limit theory for finite-dimensional semimartingales.
        \item[(LLNb)] Show that $(\Delta_n^{1-\frac m2}P_N^m SAMPV_t^n)_{t\in [0,T]}$ converges to $0$ uniformly in $n$ and $t$ as $N\to \infty$. Standard arguments then imply that  $$\left(\Delta_n^{1-\frac m2}\left(SAMPV_t^n-\int_0^t \rho^{\otimes k}(m_1,...,m_k)ds\right)\right)_{t\in [0,T]}\stackrel{u.c.p.}{\longrightarrow}0\quad \text{ as }n\to\infty.$$
    \end{itemize}
    For the central limit theorems for the $SARCV$ we have
    \begin{itemize}
        \item[(CLTa)] Show that \begin{equation}
            (\tilde{Z}^{n,2}_t)_{t\in[0,T]}:= \left(\Delta_n^{-\frac 12}\left(\sum_{i=1}^{\ul}\tilde{\Delta}_i^n Y^{\otimes 2}-\int_{(i-1)\Delta_n}^{i\Delta_n} \Sigma_s^{\mathcal S_n}ds\right)\right)_{t\in[0,T]}
        \end{equation}
         for $ n\in\mathbb N$, which is a sequence of sums of martingale differences,
        is tight in $\mathcal D([0,T],\mathcal H)$ provided that the (localised) Assumption  \ref{As: Integrability Assumption on Drift and Volatility} holds.
         \item[(CLTb)] Prove that under (localised) Assumption \ref{As: Integrability Assumption on Drift and Volatility} the finite-dimensional distributions $((I-P_N)\tilde{Z}^{n,2}_t)_{t\in[0,T]}$ converge to an asymptotically conditional Gaussian process with the covariance $(I-P_N)\Gamma_t (I-P_N)$ by virtue of \eqref{Very very Weak localised integrability Assumption on the moments} and the finite-dimensional limit Theorem 5.4.2 in \cite{JacodProtter2012}.
         \item[(CLTc)] In order to prove Theorem \ref{T:Infeasible Central Limit Theorem}, we appeal to Assumption \ref{As: Spatial regularity} to show that $$\Delta_n^{-\frac 12}\sum_{i=1}^{\ul}\int_{(i-1)\Delta_n}^{i\Delta_n} \left(\Sigma_s^{\mathcal S_n}-\Sigma_s \right)ds\stackrel{u.c.p.}{\longrightarrow}0,$$ and for Theorem \ref{T: Central limit theorem for functionals of the quadratic covariation} to the fact that the operator $B$ has its finite-dimensional range in the $1/2$-Favard class of the dual semigroup in order to show that $$\Delta_n^{-\frac 12}\sum_{i=1}^{\ul}\int_{(i-1)\Delta_n}^{i\Delta_n} \langle (\Sigma_s^{\mathcal S_n}-\Sigma_s), B\rangle_{\mathcal H} ds\stackrel{u.c.p.}{\longrightarrow}0.$$
    \end{itemize}

\subsection{Comments on the proof of the laws of large numbers}
The imposed conditions on the law of large numbers Theorems \ref{T: Law of large numbers for tensor power variations} and \ref{T: Law of large numbers for Multitensorpower variations}  state that the finite-dimensional multipower variations $\sum_{i=1}^{\ul}  ((I-P_N^m)\bigotimes_{j=1}^k\Delta_{i+j-1}^n S^{\otimes m_j})$  fulfil the required conditions of the corresponding laws of large numbers. In the case of power variations, that is under the localised version of Assumption \ref{As: for LLN of Power Variations}, Theorem 3.4.1 in \cite{JacodProtter2012} is applicable. For the multipower variations with the localised version of Assumption \ref{As: locally bounded coefficients}, Theorem 8.4.1 in \cite{JacodProtter2012} applies and yields (LLNa).

Now, observe that the triangle inequality yields
\begin{align*}
 &  \left\| P_N^m\left(SAMPV_t^n(m_1,...,m_k)-\int_0^t \rho_{\Sigma_s}^{\otimes k}(m_1,...,m_k)ds\right)\right\|_{\mathcal H^m}\\
    &\qquad\leq \left\| P_N^mSAMPV_t^n(m_1,...,m_k)\right\|_{\mathcal H^m}+\left\|P_N^m\int_0^t \rho_{\Sigma_s}^{\otimes k}(m_1,...,m_k)ds\right\|_{\mathcal H^m}.
\end{align*}
For a given $\epsilon>0$, after appealing to the inequalities of Markov and H\"older together with \eqref{outline: whole increment estimate}, one finds that 
\begin{align*}
    &\sup_{n\in\mathbb N}\mathbb P\left[\sup_{t\leq T} \Delta_n^{1-\frac m2}\|P_N^m SAMPV_t^n(m_1,...,m_k)\|_{\mathcal H^m}>\epsilon\right]
 \to 0 \quad \text{ as }N\to \infty.
\end{align*}
Moreover, straightforward calculations lead to
\begin{align*}
   \left\|\rho_{p_N\Sigma_sp_N}(m)\right\|_{\mathcal H^m}^2
     \leq  |\mathcal P(m)|^2(\sum_{j\geq N}\|\Sigma_s^{\frac 12}e_{j}\|^2)^{ m},
\end{align*}
which converges to $0$ as $N\to\infty$, since $\Sigma_s^{\frac 12}$ is a Hilbert-Schmidt operator. Through Markov's inequality, one finds
\begin{align*}
    &\mathbb P\left[\sup_{t\leq T}\left\|P_N^m \int_0^t \rho_{\Sigma_s}^{\otimes k}(m_1,...,m_k)ds\right\|_{\mathcal H^m}>\epsilon\right]
    \leq \frac {|\mathcal P(m)|}{\epsilon}\int_0^T   \mathbb E\left[\left(\sum_{j\geq N}\|\Sigma_s^{\frac 12}e_{j}\|^2\right)^{ \frac m2}\right]ds.
\end{align*}
Dominated convergence implies that this converges to $0$ as $N\rightarrow\infty$, which shows (LLNb).

\subsection{Comments on the proofs of the central limit theorem}
 In order to show tightness for the sequence $\tilde Z^{n,2}$ we appeal to a criterion from \cite[p.35]{Joffe1986}:
\begin{theorem}\label{outline: Aldous Tightness Theorem} Let $H$ be a separable Hilbert space.
The family of laws $(\mathbb P_{\psi^n})_{n\in\mathbb N}$ of a family of random variables $(\psi^n)_{n\in\mathbb N}$ in $\mathcal D([0,T],H)$ is tight if the following two conditions hold:\begin{itemize}\item[(i)] $(\mathbb P_{\psi^n_t})_{n\in\mathbb N}$ is tight for each $t\in [0,T]$ and
    \item[(ii)](Aldous' condition) For all $\epsilon,\eta>0$ there is an $\delta>0$ and $n_0\in\mathbb N$ such that for all sequences of stopping times $(\tau_n)_{n\in \mathbb N}$ with $ \tau_n\leq T-\delta$ we have 
    \begin{equation}
        \sup_{n\geq n_0}\sup_{\theta\leq\delta}\mathbb P[\|\psi_{\tau_n}^n-\psi_{\tau_n+\theta}^n\|_H > \eta]\leq \epsilon.
    \end{equation}
\end{itemize}
\end{theorem}
After some tedious estimations, one can verify Aldous' condition for  $\tilde Z^{n,2}$ 
under the localised versions of Assumptions \ref{As: Integrability Assumption on Drift and Volatility}. 
 Then it remains to show the spatial tightness, that is tightness of $(\tilde Z^{n,2}_t)_{n\in\mathbb N}$ as random sequences in $\mathcal H$ for each $t\in [0,T]$. In order to do this, we argue under condition \eqref{Very very Weak localised integrability Assumption on the moments} that, without loss of generality, we can assume $\alpha\equiv 0$. 
 Moreover, we will appeal to the following criterion, which is based on the equi-small tails-characterisation of compact sets in Hilbert spaces and is well known (c.f. Lemma 1.8.1 in \cite{vanderVaart1996}): 
 \begin{lemma}\label{outline: Tightness Criterion in Hilbert spaces}
Let $(Y_n)_{n\in \mathbb N}$ be a sequence of random variables on a probability space $(\Omega,\mathcal F,\mathbb P)$ with values in a separable Hilbert space $H$ and having finite second moments.  If for some orthonormal basis $(e_n)_{n\in\mathbb N}$ we have
 \begin{equation}\label{outline: Weak Spatial Tightness Criterion in Hilbert space}
  \lim_{N\to\infty}\sup_{n\in\mathbb N}  \sum_{k\geq N} \mathbb E\left[\langle Y_n, e_k\rangle^2 \right]= 0,
\end{equation}
 then the sequence $( Y_n)_{n\in\mathbb N}$ is tight.
\end{lemma}
 To show the spatial tightness of $\tilde Z^{n,2}$, we observe that
$$\sum_{m,k\geq N}\langle \tilde Z_t^{2,n}, e_k\otimes e_m\rangle_{\mathcal H}^2= \|\sum_{i=1}^{\ul}\tilde Z_n^N(i)\|_{\mathcal H}^2,$$
where 
\begin{align*}
   \tilde Z_n^N(i) :=
  \Delta_n^{-\frac 12}\left( (p_N\tilde{\Delta}_i^nY)^{\otimes 2}-\int_{t_{i-1}}^{t_i} p_N\mathcal S(t_i-s)\Sigma_s\mathcal S(t_i-s)^* p_N ds\right).
\end{align*}
Next note that $t\mapsto \psi_t=\int_{(i-1)\Delta_n}^tp_N\mathcal S(i\Delta_n-s)\sigma_sdW_s$ is a martingale for $t\in[(i-1)\Delta_n,i\Delta_n]$. From \cite[Theorem 8.2, p.~109]{PZ2007} we then deduce that 
the process $(\zeta_t)_{t\geq 0}$ given by $\zeta_t=  \left(\psi_t\right)^{\otimes 2}-\langle\langle \psi\rangle\rangle_t$, where $\langle\langle \psi\rangle\rangle_t= \int_{(i-1)\Delta_n}^t p_N\mathcal S(t_i-s)\Sigma_s\mathcal S(t_i-s)^*p_Nds$,
is a martingale w.r.t.~$(\mathcal{F}_t)_{t\geq 0}$. Therefore 
$\mathbb E[\tilde Z_n^N(i)|\mathcal F_{t_{i-1}}]=0$ and $ \mathbb E[\langle \tilde Z_n^N(i), \tilde Z_n^N(j)\rangle_{\mathcal H} ]=0$, which yields 
 $  \EE  [\Vert\sum_{i=1}^{\ul}\tilde Z_n^N(i) \Vert_{\mathcal H}^2 ] = \sum_{i=1}^{\lfloor T/\Delta_n\rfloor} \EE[\Vert\tilde Z_n^N(i)\Vert_{\mathcal H}^2]$.
Moreover, it holds $$
    \mathbb E[\Vert \tilde Z_n^N(i)\Vert_{\mathcal H}^2]
   \leq 4  \Delta_n\int_{(i-1)\Delta_n}^{i\Delta_n}\mathbb E[\Vert p_N\sigma_s^{\mathcal S_n} \Vert_{L_{\text{HS}}(U,H)}^4]ds,$$
   such that we ultimately obtain
\begin{align*}
\sum_{k,l\geq N} \EE  \left[\langle\tilde Z_t^{n,2}, e_k\otimes e_l\rangle^2 \right]
   \leq & 4 \sup_{n\in\mathbb N}\int_0^T\mathbb E\left[\Vert p_N\sigma_s^{\mathcal S_n} \Vert_{L_{\text{HS}}(U,H)}^4\right]ds,
\end{align*}
which converges to $0$ due to \eqref{outline: volatility and drift have equismall tails}.
Lemma \ref{outline: Tightness Criterion in Hilbert spaces} yields the claim in (CLTa), i.e., we have shown the following intermediate result:
\begin{theorem}\label{outline: Tightness for the quadratic variation}
Let Assumption \ref{As: Integrability Assumption on Drift and Volatility} hold. Then the sequence of processes 
$(\tilde{Z}^{n,2}_t)_{t\in[0,T]}$
is tight in $\mathcal D([0,T],\mathcal H).$
\end{theorem}
We now outline the proof of the stable convergence in law as a process of the finite-dimensional distributions $(\langle \tilde{Z}_t^{n,2},e_k\otimes e_l\rangle)_{k,l=1,...,d}$. 
Due to \eqref{FINITE DIMENSIONAL DISTRIBUTIONS ARE ESSENTIALLY SEMIMARTINGALE POWER VARIATIONS} and after some technical calculations, these finite-dimensional distributions can be asymptotically identified with the ones of the quadratic variation of the associated multivariate semimartingale, i.e., the stable limit of $ (\langle\tilde{Z}_t^{n}e_k,e_l\rangle)_{k,l=1,...,d}$ is the same as the one of 
$$\left(\Delta_n^{-\frac 12} \sum_{i=1}^{\ul}  \left(\langle \Delta_i^n S, e_k\rangle\langle \Delta_i^n S, e_l\rangle- \int_{(i-1)\Delta_n}^{i\Delta_n}\langle \Sigma_s e_k,e_l\rangle ds\right)\right)_{k,l=1,...,d}.$$
The latter is a component of the difference between realised quadratic covariation and the quadratic covariation of the $d$-dimensional continuous local martingale $S_t^d=(\langle S_t,e_1\rangle,...,\langle S_t, e_d\rangle)$. 
Therefore, 
$(\langle \tilde Z_t^{n} e_k,e_l\rangle)_{k,l=1,...,d}$
converges by Theorem 5.4.2 from \cite{JacodProtter2012} stably as a process to  a continuous (conditional on $\mathcal F$)  mixed normal distribution which can be realised on a very good filtered extension as
\begin{align*}
  N_{k,l} = \frac 1{\sqrt 2}\sum_{c,b=1}^d\int_0^t \hat{\sigma}_{kl,bc}(s)+\hat{\sigma}_{lk,bc}(s) dB^{cb}_s.
\end{align*}
Here, $\hat{\sigma}(s)$ is $d^2\times d^2$-matrix, being the  square-root of the matrix $\hat c(s)$ with entries $\hat c_{kl,k'l'}(s)= \langle \Sigma_s e_k,e_{k'}\rangle \langle \Sigma_s e_l,e_{l'}\rangle $. Furthermore, $B$ is a matrix of independent Brownian motions. 
As now all finite-dimensional distributions converge stably and the sequence of measures is tight, we obtain by a modification of Proposition 3.9 in \cite{Hausler2015} that the convergence is indeed stable in the Skorokhod space. 
One can then show that the asymptotic normal distribution has covariance $\Gamma_t$. This gives (CLTb) and thus
an auxiliary central limit theorem, which does not rely on the spatial regularity condition in Assumption  \ref{As: Spatial regularity}:
\begin{theorem}
Let Assumption \ref{As: Integrability Assumption on Drift and Volatility} hold. We have that $ \tilde Z^{n,2} \stackrel{\mathcal L-s}{\Rightarrow} (\mathcal N(0,\Gamma_t))_{t\in[0,T]}$.
\end{theorem}
In order to prove Theorem   \ref{T:Infeasible Central Limit Theorem} we have to show $
     \Delta_n^{-\frac 12}\sum_{i=1}^{\ul}\int_{(i-1)\Delta_n}^{i\Delta_n} \Sigma_s^{\mathcal S_n}-\Sigma_s ds\stackrel{u.c.p.}{\longrightarrow} 0.$
As $e_k\in D(\mathcal A^*)$ and due to the fact that $\|(\mathcal S(\Delta_n)^*-I) e_k\|=\|\int_0^{\Delta_n} \mathcal S(u)^* \mathcal A^* e_k du\|=\mathcal O(\Delta_n) $, it is relatively straightforward to show
 that for all $N\in\mathbb N$
 \begin{equation}\label{Adjusted minus nonadjusted quadratic variation vanishes}
      (I-P_N)  \Delta_n^{-\frac 12}\sum_{i=1}^{\ul}\int_{(i-1)\Delta_n}^{i\Delta_n} \left(\Sigma_s^{\mathcal S_n}-\Sigma_s\right) ds\stackrel{u.c.p.}{\longrightarrow} 0.
 \end{equation}
 Further, we find by the  triangle, Bochner and Hölder inequalities 
 \begin{align*}
   &  \mathbb E\left[\sup_{t\in [0,T]}\left\|P_N  \Delta_n^{-\frac 12}\sum_{i=1}^{\ul}\int_{(i-1)\Delta_n}^{i\Delta_n} \left(\Sigma_s^{\mathcal S_n}-\Sigma_s\right) ds\right\|_{\mathcal H}\right]\\
       &\leq   \left(\int_0^T\mathbb E\left[\left\|\Delta_n^{-\frac 12}  (\mathcal S(\lfloor s/\Delta_n\rfloor\Delta_n-s)-I)\sigma_s \right\|_{\text{op}}^2 \right]ds\right)^{\frac 12}\\
       & \qquad \times\left(\int_0^T\sqrt 2\mathbb E\left[\left\|p_N\sigma_s \right\|_{L_{\text{HS}}(U,H)}^2+\left\|p_N\mathcal S(\lfloor s/\Delta_n\rfloor\Delta_n-s)\sigma_s \right\|_{L_{\text{HS}}(U,H)}^2 \right]ds\right)^{\frac 12}.
 \end{align*}
 The first factor is finite by Assumption \eqref{As: Spatial regularity}(i), whereas the second one converges to 0 as $N\to\infty$ by \eqref{outline: volatility and drift have equismall tails}. By combining this with \eqref{Adjusted minus nonadjusted quadratic variation vanishes} the claim follows and Theorem \ref{T:Infeasible Central Limit Theorem} is proved.

In order to prove Theorem \ref{T: Central limit theorem for functionals of the quadratic covariation} we can argue similarly as for Theorem \ref{T:Infeasible Central Limit Theorem} that 
we just have to show $
     \Delta_n^{-\frac 12}\sum_{i=1}^{\ul}\int_{(i-1)\Delta_n}^{i\Delta_n} \langle\Sigma_s^{\mathcal S_n}-\Sigma_s,B\rangle_{\mathcal H} ds\stackrel{u.c.p.}{\longrightarrow} 0.$
 We can argue componentwise, which is why we assume without loss of generality that $B=h\otimes g$ and split the expression into two integral terms:
\begin{align*}
    & \Delta_n^{-\frac 12}\sum_{i=1}^{\ul}\int_{(i-1)\Delta_n}^{i\Delta_n} \langle (\Sigma_s^{\mathcal S_n}-\Sigma_s) h,g\rangle ds\\
     &\qquad=\Delta_n^{-\frac 12}\sum_{i=1}^{\ul}\int_{(i-1)\Delta_n}^{i\Delta_n} \langle ((\mathcal S(i\Delta_n-s)-I)\Sigma_s\mathcal S(i\Delta_n-s)^*) h,g\rangle ds\\
     &\qquad\qquad +\Delta_n^{-\frac 12}\sum_{i=1}^{\ul}\int_{(i-1)\Delta_n}^{i\Delta_n}\langle (\Sigma_s (\mathcal S(i\Delta_n-s)-I)^*) h,g\rangle ds\\
     &\qquad=(1)_n+(2)_n.
\end{align*}
We can show the convergence for $(1)_n$ only, as the convergence for $(2)_n$ is analogous. It holds that
\begin{align*}
   (1)_n = & \Delta_n^{-\frac 12}\sum_{i=1}^{\ul}\int_{(i-1)\Delta_n}^{i\Delta_n} \langle (I-p_N)(\Sigma_s\mathcal S(i\Delta_n-s)^*) h,(\mathcal S(i\Delta_n-s)-I)^*g\rangle ds\\
   & +\Delta_n^{-\frac 12}\sum_{i=1}^{\ul}\int_{(i-1)\Delta_n}^{i\Delta_n} \langle p_N(\Sigma_s\mathcal S(i\Delta_n-s)^*) h,(\mathcal S(i\Delta_n-s)-I)^*g\rangle ds\\
   = & (1.1)_{n,N}+(1.2)_{n,N}.
\end{align*}
Using that  $(\mathcal S(i\Delta_n-s)-I)e_j=\int_s^{i\Delta_n}\mathcal S(u-s)\mathcal A e_j ds$ and that the projection $(I-P_N)$ has the form $(I-P_N)=\sum_{j=1}^{N-1}\langle \cdot,e_j\rangle e_j$, we can show that 
\begin{equation}\label{outline: (1.1) of the functional remainder is AN}
\sup_{t\in[0,T]}|(1.1)_{n,N}|
      \leq  \Delta_n^{\frac 12}\sum_{j=1}^{N-1} \int_0^T\|\Sigma_s\|_{\text{op}}ds \|h\|\|g\|\sup_{t\in[0,T]}\|\mathcal S(t)\|_{\text{op}}^2,
\end{equation}
which converges to $0$ as $n\to\infty$. In particular,
$\sup_{t\in[0,T]}|(1.1)_{n,N}|\stackrel{u.c.p.}{\to}0$ as $n\to \infty$.
From the assumption that $g\in F^{\mathcal S^*}_{1/2}$ we can derive a 
finite constant 
$$
K:=\sup_{t\in[0,T]}\|\mathcal S(t)\|_{\text{op}}\left(\int_0^T\mathbb E\left[\|\sigma_s^* \|_{\text{op}}^2\right]\|h\|^2ds\right)^{\frac 12}\sup_{t\leq T}\|t^{-\frac 12}(\mathcal S(t)-I)^* g\|<\infty
$$
such that
\begin{align*}
 \mathbb E\left[    \sup_{t\in[0,T]}|(1.2)_{n,N}|\right]
    \leq K \left(\int_0^T \mathbb E\left[\| p_N\sigma_s\|_{\text{op}}^2\right]ds\right)^{\frac 12} ,
\end{align*}
which converges to $0$ as $N\to\infty$ by \eqref{outline: volatility and drift have equismall tails}. Thus, combining this uniform  convergence result with \eqref{outline: (1.1) of the functional remainder is AN} 
and the analogous argumentation for $(2)_n$ yields the assertion and, thus, (CLTc).

\section{Conclusion}\label{sec: Conclusion}
In this article, we introduced feasible central limit theorems for the semigroup-adjusted realised covariations and, thus, provided a basis for functional data analysis of mild solutions to a large number of semilinear stochastic partial differential equations. We also addressed the issue of how this can be translated into a fully discrete setting, whereby we assumed a regular spatio-temporal sampling grid. In general, finding closed forms for the semigroup-adjusted multipower variations is a task that must be addressed for each semigroup (or equivalently each infinitesimal generator), each sampling grid and any precise application separately. Certainly, the Hilbert space approach is well suited to account for potentially any sampling grid. 

To gain an overview of the infinite-dimensional limit theory introduced for both $SARCV^n$ and $RV^n$ in this article, it might be helpful to give a systematic summary.
For the sake of presentation, it is tedious and eventually not very instructive to repeat all assumptions in full technical detail so instead we make a distinction on the basis of the magnitude of 
$p_n:=\int_0^T \|(\mathcal S(\Delta_n)-I)\sigma_s\|_{L_{\text{HS}}(U,H)}ds$ in terms of $\Delta_n$ and assume the volatility $\sigma$ of a mild It{\^o} process of the form \eqref{mild Ito process} to be deterministic. In this regard we can distinguish four cases:
\begin{align*}
   (i)& \text{ If } p_n=o(\Delta_n^{\frac 34}),\text{ then } SARCV^n \text{ and }RV^n \text{ satisfy LLN and CLT }.\\
    (ii)& \text{ If }p_n=o(\Delta_n^{\frac 12}),\text{ then } RV^n \text{ satisfies LLN, }SARCV^n\text{satisfies LLN and CLT}.\\
    (iii)&\text{ If } p_n=\mathcal O(\Delta_n^{\frac 12}),\text{ then } SARCV^n \text{ satisfies LLN and CLT}.\\
     (iv)& \text{ In general }SARCV^n \text{ satisfies LLN},
\end{align*}
where satisfying $LLN$ (law of large numbers) means convergence to the integrated volatility in probability and satisfying $CLT$ (central limit theorem) means asymptotic normality of the normalised estimator.
Observe that Example \ref{Ex: counterexample for the Spatial regularity assumption} in Section \ref{sec: Limit Theorems for the SARCV} yields that we cannot reduce the regularity in (iii), if we want to guarantee the validity of a general central limit theorem for $SARCV^n$.
Example \ref{Counterexample for RV} shows that $RV^n$ does not have to satisfy a central limit theorem if $p_n= o(\Delta_n^{ 1/2})$ is not valid and underlines the necessity of the adjustment by the semigroup. If even $p_n= o(\Delta_n^{ 1/4})$ does not hold, then $RV^n$ does not even have to satisfy the LLN.

Moreover, it is likely that in many realistic scenarios, the distribution underlying the data and the sampling itself yield some additional challenges, which can be approached in our setting. Let us comment on some of these points:

 \textit{Functional sampling}: In infinite dimensions, we witness sampling schemes that have no counterpart in finite dimensions. For instance, data could be sampled as averages (or in general smooth functionals) of the process of interest over certain time periods in the future or within a demarcated area. This is for instance  the case for energy swap prices or meteorological forecasting data. Our framework yields an ideal basis to derive inferential statistical tools in these situations.
    
     \textit{Jumps}: Many processes are not considered to be continuous in time. In fact, many financial time series show jumps and spikes on a regular basis, which is, in particular, the case in energy markets, a potential application of our theory. This suggests the inclusion of a pure-jump component to our framework, such as in the framework of \cite{FilipovicTappeTeichmann2010b}. However, as in finite dimensions, jumps will considerably complicate expressions, applications and proofs and, thus, more effort has to go into the task of making inference on non-continuous behaviour in infinite-dimensional models. Arguably, the structure of our proof, which appeals to tightness and already existing limit theorems from finite dimensions, yields a promising approach. 
    
    \textit{Asynchronous sampling}: It could very well be, that we sample at high frequency in time, but sparsely and irregularly in space. Ignoring this (for instance by na{\"i}ve rearrangement to refresh times) can have unpleasant consequences such as the Epps effect, c.f. \cite[Sec.~9.2.1]{Ait-SahaliaJacod2014}. Again, energy intraday market prices, in which all available maturities are unlikely traded at the exact same time instances, can be prone to this. Infinite-dimensionality and the potentially necessary adjustment by the semigroup might make it harder than in the finite-dimensional case to deal with this issue, as in addition to asynchronous sampling, one has to deal with the problem of smoothing the adequately aggregated data in space.
  
   \textit{Noise}: The task of accounting for noise in the samples, often called \textit{market microstructure noise} in financial applications has received much attention by the research community (c.f., for example, \cite{zhang2005}, \cite{barndorff2008} or \cite{jacod2009}), as noise lets the quadratic variation severely overshoot the integrated volatility in the presence of data sampled at very high frequency. In combination with the problem of smoothing (and asynchronous sampling) this appears to be a delicate question in infinite-dimensional applications. 
    However, both finite-dimensional high-frequency statistics and functional data analysis have several tools available to deal with noise and it is intriguing to find out how they can be exploited to overcome this problem in the future.

\subsection*{Acknowledgements} F. E. Benth and A. E. D. Veraart would like to thank the Isaac Newton Institute for Mathematical Sciences for their support and hospitality during the programme  \emph{The Mathematics of Energy Systems} when parts of the work on this paper were undertaken. 
We thank 
two anonymous referees for their careful reading
and constructive comments. We also wish to thank Sascha Gaudlitz for pointing out errors in an earlier version of this article.

This work was supported by:
EPSRC grant number EP/R014604/1.
D. Schroers and F. E. Benth gratefully acknowledge financial support from the STORM project 274410, funded by the Research Council of Norway, and the thematic research group SPATUS, funded by UiO:Energy at the University of Oslo. 

\bibliographystyle{agsm}

\newpage
\appendix

\section{Notation}
For convenience, we list some of the frequently used notation throughout this appendix:
\begin{itemize}
    \item $L_{\text{HS}}(U,H)$: the space of Hilbert-Schmidt operators from $U$ to $H$.
    \item $\mathcal H$: equals $L_{\text{HS}}(H,H)$, when $H$ is the initial Hilbert space in which the mild It{\^o} process $Y$ takes its values.
    \item $\mathcal H^m$: for $m\geq 3$ this is the space of operators spanned by the orthonormal basis $(e_{j_1}\otimes\cdots\otimes e_{j_m})_{j_1,...,j_m\in\mathbb N}$ for an orthonormal basis $(e_j)_{j\in\mathbb N}$ of $H$ with respect to the Hilbert-Schmidt norm. 
    \item $\sigma_s^{\mathcal S_n}=\mathcal S(i\Delta_n-s)\sigma_s $ for $s\in((i-1)\Delta_n,i\Delta_n]$.
\item
$\alpha_s^{\mathcal S_n}=\mathcal S(i\Delta_n-s)\alpha_s$
for $s\in((i-1)\Delta_n,i\Delta_n]$. 
\item $\Sigma_s^{\mathcal S_n}:=\sigma_s^{\mathcal S_n}(\sigma_s^{\mathcal S_n})^*$.
\item $p_N$ is the projection onto $v^N:=\overline{lin(\{e_j:j\geq N\})}$, for some orthonormal basis $(e_j)_{j\in\mathbb N}$.
\item $P_N^m$ is the projection onto $\overline{lin(\{ \bigotimes_{l=1}^{m} e_{k_l}: k_l\geq N\})}$.
\item $P_N=P_N^2$ for the special case $m=2$.
\end{itemize}

We start by giving several auxiliary results that are needed to prove the limit theorems.

\section{Technical tools}\label{sec: technical tools}
\subsection{Localisation}
 For both the laws of large numbers Theorems \ref{T: Law of large numbers for tensor power variations} and \ref{T: Law of large numbers for Multitensorpower variations} and the central limit Theorems \ref{T:Infeasible Central Limit Theorem}, \ref{T: Central limit theorem for functionals of the quadratic covariation}  we can work under the following stronger assumptions: 
\begin{assumption}\label{As: Weakened Assumption for the CLT for the SARCV} 
There is a constant $A>0$ such that almost surely
$$\int_0^T\|\alpha_s\|^2+\|\sigma_s\|_{L_{\text{HS}}(U,H)}^4 ds\leq A.$$
\end{assumption}
\begin{assumption}[m]\label{As: Decomposition of Q} 
There is a constant $A>0$ such that almost surely
$$\int_0^T\|\alpha_s\|^{\frac {2m}{2+m}}+\|\sigma_s\|_{L_{\text{HS}}(U,H)}^m ds\leq A.$$
\end{assumption}
\begin{assumption}\label{as: localised martingale assumption without volatility}
Assumption \ref{As: locally bounded coefficients} holds and there is an $A>0$ such that almost surely
$$\|\alpha_s\|+\|\sigma_s \|_{L_{\text{HS}}(U,H)}\leq A,\quad s\in [0,T].$$
\end{assumption}

We have then the following simplifying localisation result:
\begin{theorem}[Localisation]\label{L: localisation} The following relaxations can be made for the limit theorems in this work:
\begin{itemize}
    \item[(a)](localisation for CLT for functionals of $SARCV$ and $RV$) If the central limit Theorems \ref{T: Central limit theorem for functionals of the quadratic covariation} or \ref{T:Finite dimensional limit theorems for functionals of nonadjusted covariation}(ii) hold under Assumption \ref{As: Weakened Assumption for the CLT for the SARCV}, then they also hold under Assumption \ref{As: Integrability Assumption on Drift and Volatility}.
    \item[(b)](localisation for LLN for power variations) If the law of large numbers Theorem \ref{T: Law of large numbers for tensor power variations}  holds under Assumption \ref{As: Decomposition of Q}(m), then it also holds under Assumption \ref{As: for LLN of Power Variations}(m). Moreover, the laws of large numbers Theorems \ref{T: Limit Theorems for nonadjusted RV}(i) or \ref{T:Finite dimensional limit theorems for functionals of nonadjusted covariation}(i) hold, if they hold under the additional Assumption \ref{As: Decomposition of Q}(2).
     \item[(c)](localisation for LLN for multipower variations)If the law of large numbers, Theorem \ref{T: Law of large numbers for Multitensorpower variations},  holds under Assumption \ref{as: localised martingale assumption without volatility}, then it also holds under Assumption \ref{As: locally bounded coefficients}.
    \item[(d)](localisation for CLT for $SARCV$ and $RV$) If the central limit Theorem \ref{T:Infeasible Central Limit Theorem}  holds under Assumptions \ref{As: Weakened Assumption for the CLT for the SARCV} and \ref{As: Spatial regularity}(i) (respectively Theorem \ref{T: Limit Theorems for nonadjusted RV}(ii) holds under Assumptions \ref{As: Weakened Assumption for the CLT for the SARCV} and \ref{as: strong Spatial regularity for the CLT for RV}), then it also holds under Assumptions \ref{As: Integrability Assumption on Drift and Volatility} together with \ref{As: Spatial regularity}(i) or (ii) (respectively \ref{As: Integrability Assumption on Drift and Volatility} and \ref{as: strong Spatial regularity for the CLT for RV}).
\end{itemize}
\end{theorem}
\begin{proof}
Without loss of generality we may exchange the volatility process $(\sigma_t)_{t\in [0,T]}$ by a left continuous version $(\sigma_{t-})_{t\in [0,T]}$. This is because we assumed the filtration to be right-continuous and the stochastic integrals $\int_a^b H_sdW_s$ and $\int_a^b H_{s-}dW_s$ coincide for any predictable c{\`a}dl{\`a}g process $(H_t)_{t\in [0,T]} $. In particular, in the case of a predictable c{\`a}dl{\`a}g volatility process, we can assume it to be locally bounded since any left-continuous process is locally bounded.

Now, the same localisation procedure as for finite-dimensional semimartingales described in Section 4.4.1 in \cite{JacodProtter2012} holds:
We define under each of the assumptions a different sequence $(\tau_N(i))_{N\in\mathbb N}$, $i=1,...,4$, of stopping times and the corresponding stopped process $Y_t(N,i):=Y_{t\wedge \tau_N(i)}$. Observe that on $\{ t<\tau_N(i)\}$ we have $Y_t(N,i)=Y_t$:
Set 
\begin{align*}
    & X_t(1):=\int_0^t\|\alpha_s\|^2+\|\sigma_s\|_{L_{\text{HS}}(U,H)}^4ds, \\ 
    & X_t(2):= \int_0^t\|\alpha_s\|^{\frac {2m}{2+m}}+\|\sigma_s\|_{L_{\text{HS}}(U,H)}^m ds,\\
    &X_t(3):=\|\alpha_t\|+\|\sigma_t \|_{L_{\text{HS}}(U,H)}, \\
    & X_t(4):= \int_0^t\|\alpha_s\|^2+\|\sigma_s\|_{L_{\text{HS}}(U,H)}^4+\sup_{r\in[0,t]}\|t^{-\frac 12}(\mathcal S(t)-I)\sigma_s\|_{\text{op}}^2 ds.
\end{align*}
We define for $i=1,...,4$
$$\tau_N(i):=\inf \left\{t\in[0,T]: X_t(i)\geq N\right\}, N\in\mathbb N.$$
Then Assumption \ref{As: Integrability Assumption on Drift and Volatility} assures that $\tau_N(1)\uparrow T$, Assumption \ref{As: for LLN of Power Variations} assures that $\tau_N(2)\uparrow T$, Assumption \ref{As: locally bounded coefficients} assures that $\tau_N(3)\uparrow T$, Assumption \ref{As: Integrability Assumption on Drift and Volatility} together with \ref{As: Spatial regularity}(i) or (ii) ensures $\tau_N(4)\uparrow T$.
Moreover, observe that $Y_t(N,1)$ satisfies Assumption \ref{As: Weakened Assumption for the CLT for the SARCV}, $Y_t(N,2)$ satisfies Assumption \ref{As: Decomposition of Q}(m), $Y_t(N,3)$ satisfies Assumption \ref{as: localised martingale assumption without volatility}, $Y_t(N,4)$ satisfies Assumption \ref{As: Weakened Assumption for the CLT for the SARCV} and \ref{As: Spatial regularity}(i).

The localisation works now for all cases analogously:
Observe, that as convergence in probability is a special case of stable convergence in law (when the underlying probability space for the limiting distribution coincides with the one on which the sequence of probability laws is defined), we just have to consider the stable convergence in law.

For any mild 
It{\^o} process $X$, i.e. a process of the same form as $Y$ as defined in \eqref{mild Ito process}, with values in $H$
 we write $U_n(X)$ either for the process of normalised multipower variations (for the central limit theorems) or the supremum over the unnormalised multipower variations (and as a special case, power variations). I.e., for proving $(a)$ it is $U_n(X)_t=\Delta_n^{-1/2}\langle (SARCV_t^n-\int_0^t\Sigma_s ds),B\rangle$ (resp. $U_n(X)_t=\Delta_n^{-1/2}\langle (RV_t^n-\int_0^t\Sigma_s ds),B\rangle$) for $B=\sum_{l=1}^K \mu_l h_l\otimes g_l\in \mathcal H$ for $h_i,g_l\in F_{3/4}^{\mathcal S^*}$ and $U(X)$ is the asymptotic distribution as in Theorem \ref{T: Central limit theorem for functionals of the quadratic covariation}, for $(d)$ it is $U_n(X)_t=\Delta_n^{-1/2} (SARCV_t^n-\int_0^t\Sigma_s ds) $ (resp. $U_n(X)_t=\Delta_n^{-1/2} (RV_t^n-\int_0^t\Sigma_s ds) $) 
 and $U(X)$ is the asymptotic distributions as described in Theorem \ref{T:Infeasible Central Limit Theorem} and $(b)$ and $(c)$ it is $U(X)=0$ and  
 \begin{align*}
 U_n(X)_t
 =  \sup_{s\in[0,t]}\|\Delta_n^{1-\frac m2}\sum_{i=1}^{\lfloor s/\Delta_n\rfloor-k+1}\bigotimes_{j=1}^k\tilde{\Delta}_{i+j-1}^n X^{\otimes m_j}-\int_0^t\rho_{\Sigma_s}^{\otimes k}(m_1,...,m_k)ds\|_{\mathcal H^m}.
 \end{align*}
 Moreover, for the non-adjusted limit theorems, we exchange $\tilde{\Delta}$ by $\Delta$ in the above expressions.
 
 In order to prove the claim, it is enough to show $U_n(Y)\to U(Y)$ stably in law as a process under the respective localised assumptions. I.e., if $\tilde{\Omega}=\Omega\times\Omega'$, $\tilde{\mathcal F}=\mathcal F\otimes \mathcal F'$ and $\tilde{\mathbb P}=\mathbb P[d\omega]\mathbb Q[\omega,d\omega']$ is the extension on which $U(X)$ can be realised, we want to show that
for all bounded, continuous functions $g:\mathcal D([0,T],\mathcal H)\to \mathbb R$ and all bounded $\mathcal F$-measurable real-valued random variables $Z$ we have
\begin{equation}\label{abstract stable convergence term}
    \lim_{n\to\infty}\mathbb E\left[Zg(U_n(Y))\right]=\tilde{\mathbb E}\left[Zg(U(Y))\right],
\end{equation}
where $\tilde{\mathbb E}$ is the expectation with respect to $\tilde{\mathbb P}$. If we write 
$$\mathbb Q_Y(g)(\omega):=\int_{\Omega'} g(U(Y))(\omega,\omega')\mathbb Q(\omega,d\omega'),$$ 
we can express \eqref{abstract stable convergence term} as 
$$\lim_{n\to\infty}\mathbb E\left[Zg(U_n(Y))\right]=\mathbb E\left[Z\mathbb Q_Y(g)\right].$$
Now assume that the convergence results hold for $Y(N,i)$, for which the localised assumptions are valid. We can therefore deduce that
\begin{align*}
\lim_{n\to\infty}\mathbb E\left[Zg(U_n(Y))\indicator_{t\leq \tau_N(i)}\right]&=\lim_{n\to\infty}\mathbb E\left[Zg(U_n(Y(N,i)))\right] \\
&=\mathbb E\left[Z\mathbb Q_{Y(N,i)}(g)\right]=\mathbb E\left[Z\mathbb Q_Y(g)\indicator_{t\leq \tau_N(i)}\right]
\end{align*}
holds for all $N\in\mathbb N$.
 This implies \eqref{abstract stable convergence term} as boundedness of $Z$ and $g$ yield that
 $$\sup_{n\in \mathbb N}\mathbb E\left[Z(g(U_n(Y))-\mathbb Q_Y(g))\indicator_{t > \tau_N}\right]\to 0\quad \text{ as }N\to\infty.$$
 This proves the claim.
\end{proof}

In many occasions, we just need to impose a localised version of the condition
\begin{equation}
   \mathbb P\left[ \int_0^T \|\alpha_s\|^{\frac m2}+\|\sigma_s \|_{L_{\text{HS}}(U,H)}^m ds < \infty\right]=1.
\end{equation} 
In that regard, we introduce the auxiliary assumption:
\begin{assumption}[(m)]\label{In proof: Very very Weak localised integrability Assumption on the moments}
There is a constant $A>0$ such that 
$$\int_0^T \|\alpha_s\|^{\frac m2}+\|\sigma_s \|_{L_{\text{HS}}(U,H)}^m ds\leq A.$$
\end{assumption}
Based upon choosing the right $m$, this is satisfied under any (!) of the assumptions above. 
In particular Assumption \ref{As: Weakened Assumption for the CLT for the SARCV} coincides with Assumption \ref{In proof: Very very Weak localised integrability Assumption on the moments}(4) while Assumption \ref{As: Decomposition of Q}(2) coincides with Assumption \ref{In proof: Very very Weak localised integrability Assumption on the moments}(2). Assumption \ref{As: Decomposition of Q}(m) is stricter than Assumption \ref{In proof: Very very Weak localised integrability Assumption on the moments}(m) if $m >2$.

\subsection{Elimination of the semigroup on finite-dimensional projections}
We will see that if we apply certain functionals, the semigroup-adjusted increments are essentially increments of finite-dimensional semimartingales.
First, we recall the Burkholder-Davis-Gundy inequality (in what follows called the BDG inequality), c.f. \cite{Marinelli2016}.
\begin{theorem}
For an $H$-valued local martingale $(M_t)_{t\in[0,T]}$ we have for all real numbers $m\geq 1$ and all $t\leq T$
\begin{equation}\label{BDG inequality}
\mathbb E\left[\sup_{s\leq t}\|M_s\|^m\right]\leq C_m \mathbb E\left[ [M,M]_t^{\frac m2}\right],
\end{equation}
where the constant $C_m>0$ is just depending on $m$ and $[M,M]$ is the scalar quadratic variation of $M$.
\end{theorem}
For our purposes, the most important application is given in the following example:
\begin{example}
Let $M_t:=\int_0^t \sigma_s dW_s$ be a stochastic integral.
Then the BDG inequality \eqref{BDG inequality} reads
\begin{align*}
\mathbb E\left[\sup_{s\leq t}\|\int_0^t \sigma_s dW_s\|^m\right]&\leq C_m \mathbb E\left[\left(\int_0^t \text{Tr}(\Sigma_s) ds\right)^{\frac m2}\right] \\
&=C_m\mathbb E\left[\left(\int_0^t \|\sigma_s\|_{L_{\text{HS}}(U,H)}^2 ds\right)^{\frac m2}\right],
\end{align*}
for the constant $C_m$ just depending on $m$ 
and for all $t\leq T$.
\end{example}
Observe that there always exists an orthonormal basis $(e_i)_{i\in\mathbb N}$ of $H$ that is contained in the domain $D(\mathcal A^*)$ of the generator $\mathcal A^*$ of the adjoint semigroup $(\mathcal S(t)^*)_{t\geq 0}$. The latter is a semigroup since $H$ is a Hilbert space, see \cite[p.44]{Engel1999}.
Besides tightness, the most important argument to be able to appeal to the finite-dimensional limit theory of semimartingales is the following one:
\begin{lemma}\label{L: Reduction to martingales is possible}
Suppose that Assumption \ref{In proof: Very very Weak localised integrability Assumption on the moments}(m) holds for $m\geq 2$.
We define the $H$-valued semimartingale $S$ by
$$S_t:= \int_0^t \alpha_s ds+\int_0^t \sigma_s dW_s,$$
and fix a basis $(e_j)_{j\in\mathbb N}\subset D(\mathcal A^*)$ of $H$. 
Let furthermore $\Delta_i^n S:=S_{i\Delta_n}-S_{(i-1)\Delta_n}$ denote the non-adjusted increments of the semimartingale $S$. Then for $j_1,...,j_m\in\mathbb N$
\begin{align}\label{Finite Projections of Power variations are essentially functions of differences of Martingales}
  \langle SAMPV_t^n(m_1,...,m_k),&\bigotimes_{l=1}^m e_{j_l}\rangle_{\mathcal H^m}\notag\\
  &=\sum_{i=1}^{\ul}  \langle\bigotimes_{j=1}^k\Delta_{i+j-1}^n S^{\otimes m_j}, \bigotimes_{l=1}^m e_{j_l}\rangle_{\mathcal H^m}+ \mathcal O_p( \Delta_n^{\frac m2}).
\end{align}
\end{lemma}

\begin{proof}
First, fix $i,j\in\mathbb N$. By the stochastic Fubini theorem (and since $Y_0=0$) we have
\begin{align*}
       \int_0^t\int_0^u& \langle \alpha_s,\mathcal A^*\mathcal S(u-s)^*e_j\rangle dsdu \\
       &\qquad\qquad+\int_0^t \int_0^u \langle \sigma_s, \mathcal A^*\mathcal S(u-s)^* e_j\rangle dW_sdu+ \langle S_t,e_j\rangle \\
      &= \int_0^t\int_0^u \frac d{du}\langle \alpha_s,\mathcal S(u-s)^*e_j\rangle dsdu \\
      &\qquad\qquad+\int_0^t \int_0^u \frac d{du}\langle \sigma_s, \mathcal  S(u-s)^* e_j\rangle dW_sdu+ \langle S_t,e_j\rangle\\
      &=\int_0^t\int_s^t \frac d{du}\langle \alpha_s,\mathcal S(u-s)^*e_j\rangle duds+\int_0^t\langle \alpha_s,e_j\rangle ds\\
      &\qquad\qquad+\int_0^t \int_s^t \frac d{du}\langle \sigma_s,  \mathcal S(u-s)^* e_j\rangle dudW_s+\int_0^t \langle \sigma_s,e_j\rangle dW_s\\
      &=\int_0^t\langle \alpha_s,\mathcal S(t-s)^*e_j\rangle ds+\int_0^t \langle \sigma_s,  \mathcal S(t-s)^* e_j\rangle dW_s\\
      &= \langle Y_t,e_j\rangle.
\end{align*}
Now define 
\begin{align}
    \Delta_n^t a^j&:=\int_t^{t+\Delta_n}\int_t^u\frac d{du} \langle \alpha_s,\mathcal S(u-s)^*e_j\rangle dsdu \notag \\
    &\qquad\qquad+ \int_t^{t+\Delta_n} \int_t^u  \langle \sigma_s, \mathcal  A^*\mathcal S(u-s)^* e_j\rangle dW_sdu.
\end{align}
Again, by using the stochastic Fubini theorem we obtain
\begin{align*}
  \langle Y_{t+\Delta_n} ,e_j\rangle&-(\langle S_{t+\Delta_n},e_j\rangle- \langle S_t,e_j\rangle)+ \Delta_n^t a^j\\
= &\int_t^{t+\Delta_n}\int_0^u\frac d{du} \langle \alpha_s,\mathcal S(u-s)^*e_j\rangle dsdu\\
&\qquad-\int_t^{t+\Delta_n}\int_t^u\frac d{du} \langle \alpha_s,\mathcal S(u-s)^*e_j\rangle dsdu\\
&\qquad
+\int_t^{t+\Delta_n} \int_0^u  \frac d{du}\langle \sigma_s,  \mathcal S(u-s)^* e_j\rangle dW_sdu \\
&\qquad-\int_t^{t+\Delta_n} \int_t^u  \frac d{du}\langle \sigma_s, \mathcal  S(u-s)^* e_j\rangle dW_sdu+\langle Y_t,e_j\rangle\\
     = &\int_t^{t+\Delta_n}\int_0^t\frac d{du} \langle \alpha_s,\mathcal S(u-s)^*e_j\rangle dsdu\\
    &\qquad+\int_t^{t+\Delta_n} \int_0^t  \frac d{du}\langle \sigma_s, \mathcal  S(u-s)^* e_j\rangle dW_sdu+\langle Y_t,e_j\rangle\\
      = &\int_0^t\int_t^{t+\Delta_n}\frac d{du} \langle \alpha_s,\mathcal S(u-s)^*e_j\rangle duds\\
      &\qquad+\int_0^t \int_t^{t+\Delta_n} \frac d{du}\langle \sigma_s,  \mathcal S(u-s)^* e_j\rangle dudW_s+\langle Y_t,e_j\rangle\\
      = &\int_0^t\int_s^{t+\Delta_n}\frac d{du} \langle \alpha_s,\mathcal S(u-s)^*e_j\rangle duds\\
      &\qquad+\int_0^t \int_s^{t+\Delta_n} \frac d{du}\langle \sigma_s, \mathcal  S(u-s)^* e_j\rangle dudW_s+ \langle S_t,e_j\rangle \\
      = &\int_0^t \langle \alpha_s,\mathcal S(t+\Delta_n-s)^*e_j\rangle ds+\int_0^t \langle \sigma_s,  \mathcal S(t+\Delta_n-s)^* e_j\rangle dW_s\\
      = & \langle \mathcal S(\Delta_n)Y_t,e_j\rangle.
\end{align*}
Thus,
\begin{align*}
    \langle \tilde{\Delta}_{i}^n Y ,e_{j}\rangle=    \langle \Delta_{i}^n S,e_{j}\rangle + \Delta_n^{(i-1)\Delta_n} a^{j}.
\end{align*}
Now let $i_1,...,i_m\in\mathbb N$. We get
\begin{align*}
     \prod_{k=1}^{m} \langle \tilde{\Delta}_{i_k}^n Y,e_{j_k}\rangle
     = &\prod_{k=1}^{m}\left( \langle \Delta_{i_k}^nS,e_{j_k}\rangle + \Delta_n^{(i_k-1)\Delta_n} a^{j_k}\right)\\
     = &\prod_{k=1}^{m} \langle \Delta_{i_k}^nS,e_{j_k}\rangle+\sum_{y=1,x=2}^m b_{xy}(i_1,...,i_m,j_1,...,j_m),
\end{align*}
where each summand $b_{xy}$ is of the form
$$b_{xy}(i_1,...,i_m,j_1,...,j_m)=\prod_{r=1}^p\langle \Delta_{i_{k_r}}^nS,e_{j_{k_r}}\rangle \prod_{s=1}^q \Delta_n^{(i_{l_s}-1)\Delta_n} a^{j_{l_s}},$$
for $p\leq m-1$, $q\leq m$, $p+q=m$ and 
$$\{ i_{k_1},...,i_{k_p},i_{l_1},...,i_{l_q}\}=\{i_1,...,i_m\},\quad \{ j_{k_1},...,j_{k_p},j_{l_1},...,j_{l_q}\}=\{j_1,...,j_m\}.$$
Then by the generalised Hölder inequality
\begin{align}\label{estimate for bkl summands}
    &\mathbb E[|b_{xy}(i_1,...,i_m,j_1,...,j_m)|]\notag\\
    &\qquad\leq \mathbb E[|\langle \Delta_{i_{k_1}}^n S,e_{j_{k_1}}\rangle|^m]^{\frac 1m}\times\cdots\times\mathbb E[|\langle\Delta_{i_{k_p}}^n S,e_{j_{k_p}}\rangle|^m]^{\frac 1m}\notag \\
    &\qquad\qquad\times\mathbb E[|\Delta_n^{(i_{l_1}-1)\Delta_n} a^{j_{l_1}}|^m]^{\frac 1m}\times\cdots\times\mathbb E[|\Delta_n^{(i_{l_q}-1)\Delta_n} a^{j_{l_q}}|^m]^{\frac 1m}.
\end{align}
This means that in order to estimate $ \mathbb E[|b_{xy}(i_1,...,i_m,j_1,...,j_m)|]$ we have to find bounds on $\mathbb E[|\langle \Delta_{i}^n S,e_{j}\rangle|^m]^{\frac 1m}$ and $\mathbb E[|\Delta_n^{(i-1)\Delta_n} a^{j}|^m]^{\frac 1m}$ for $i,j\in\mathbb N$. We start with the latter term.

For any $t\leq u\leq T$, $j\in\mathbb N$, we have for the quadratic variation
\begin{align*}
\left[\int_0^{\cdot}\indicator_{[t,u]}\langle \sigma_s , \mathcal S(u-s)^*\mathcal A^* e_j\rangle dW_s\right]_T= & \int_u^t \langle \mathcal S(u-s)\Sigma_s \mathcal S(t-s)^*\mathcal A^* e_j,\mathcal A^*e_j\rangle ds \\
= &\int_t^u \| \sigma_s^* \mathcal S(u-s)^*\mathcal A^* e_j\|^2ds.
\end{align*}
Hence, by the BDG inequality \eqref{BDG inequality} we obtain, for $t\leq u\leq t+\Delta_n$,
\begin{align*}
&\mathbb E\left[\left|\int_t^u \langle \alpha_s,\mathcal A^*\mathcal S(t-s)^*e_j\rangle ds+\int_t^u \langle \sigma_s, \mathcal A^*\mathcal S(u-s)^*e_{j}\rangle dW_s\right|^m\right]^{\frac 1m}\\
\leq & \mathbb E\left[\left|\int_t^u \langle \alpha_s,\mathcal A^*\mathcal S(t-s)^*e_j\rangle ds\right|^m\right]^{\frac 1m}+ \mathbb E\left[\left|\int_t^u \langle \sigma_s, \mathcal A^*\mathcal S(u-s)^*e_{j}\rangle dW_s\right|^m\right]^{\frac 1m}\\
\leq &\mathbb E\left[\left|\int_t^u \langle \alpha_s,\mathcal A^*\mathcal S(t-s)^*e_j\rangle ds\right|^m\right]^{\frac 1m}+ C_m^{\frac 1m}\mathbb E\left[\left|\int_t^u \|\sigma_s^* \mathcal S(u-s)^*\mathcal A^* e_j\|^2 ds\right|^{\frac m2}\right]^{\frac 1m}\\
\leq &\Delta_n^{\frac {m-2}{m}}\mathbb E\left[\left(\int_t^{t+\Delta_n} \|\alpha_s\|^{\frac m2}ds\right)^2\right]^{\frac 1m}\sup_{t\in[0,T]}\|\mathcal S(t)\|_{\text{op}} \|\mathcal A^* e_j\|\\
&\qquad+ C_m^{\frac 1m} \Delta_n^{\frac {m-2}{2m}}\left(\int_t^{t+\Delta_n} \mathbb E\left[\|\sigma_s^*\|_{\text{op}}^m\right]ds\right)^{\frac 1m} \sup_{t\in[0,T]}\|\mathcal S(t)\|_{\text{op}} \|\mathcal A^* e_j\|.
\end{align*}
Therefore, as $m\geq 2$
\begin{align}\label{finite dimensional projection of a estimate}
   & \mathbb E[|\Delta_n^{t} a^{j}|^m]^{\frac 1m}\notag\\
   \leq & \Delta_n \left(\Delta_n^{\frac {m-2}{m}}\mathbb E\left[\left(\int_t^{t+\Delta_n} |\|\alpha_s\|^{\frac m2}ds\right)^2\right]^{\frac 1m}+ C_m^{\frac 1m} \Delta_n^{\frac {m-2}{2m}}\left(\int_t^{t+\Delta_n} \mathbb E\left[\|\sigma_s^*\|_{\text{op}}^m\right]ds\right)^{\frac 1m}\right)\notag\\
   &\qquad\times\sup_{t\in[0,T]}\|\mathcal S(t)\|_{\text{op}} \|\mathcal A^*e_j\|\notag\\
   \leq &\Delta_n^{\frac {3m-2}{2m}} \left(\mathbb E\left[\left(\int_t^{t+\Delta_n} |\|\alpha_s\|^{\frac m2}ds\right)^2\right]^{\frac 1m}+ C_m^{\frac 1m}\left(\int_t^{t+\Delta_n} \mathbb E\left[\|\sigma_s^*\|_{\text{op}}^m\right]ds\right)^{\frac 1m}\right)\notag\\
   &\qquad\times\sup_{t\in[0,T]}\|\mathcal S(t)\|_{\text{op}} \|\mathcal A^*e_j\|.
\end{align}
Now for the martingale differences $\mathbb E[|\langle\Delta_i^n S,e_{j}\rangle|^m]^{\frac 1m}$ we can estimate by virtue of the BDG inequality \eqref{BDG inequality}:
\begin{align}\label{finite dimensional projection of M estimate}
 & \mathbb E[|\langle\Delta_i^n S,e_{j}\rangle|^m]^{\frac 1m}\notag\\
 \leq
 &\mathbb E\left[\|\int_{(i-1)\Delta_n}^{i\Delta_n}\alpha_s ds\|^m\right]^{\frac 1m}+C_m^{\frac 1m}\mathbb E\left[(\int_{(i-1)\Delta_n}^{i\Delta_n}\langle \Sigma_s e_j,e_j\rangle ds)^{\frac m2}\right]^{\frac 1m}\notag\\
 \leq & \left(\Delta_n^{\frac {m-2}m}\mathbb E\left[\left(\int_{(i-1)\Delta_n}^{i\Delta_n}\|\alpha_s\|^{\frac m2}ds\right)^2\right]^{\frac 1m}+\Delta_n^{\frac {m-2}{2m}}C_m^{\frac 1m}  \left(\int_{(i-1)\Delta_n}^{i\Delta_n}\mathbb E[\|\sigma_s\|_{\text{op}}^m]ds\right)^{\frac 1{m}}\right)\notag\\
  \leq & \Delta_n^{\frac {m-2}{2m}}\left(\mathbb E\left[\left(\int_{(i-1)\Delta_n}^{i\Delta_n}\|\alpha_s\|^{\frac m2}ds\right)^2\right]^{\frac 1m}+C_m^{\frac 1m}  \left(\int_{(i-1)\Delta_n}^{i\Delta_n}\mathbb E[\|\sigma_s\|_{\text{op}}^m]ds\right)^{\frac 1m}\right).
\end{align}
Combining \eqref{estimate for bkl summands}, \eqref{finite dimensional projection of a estimate} and \eqref{finite dimensional projection of M estimate} yields, as $q\geq 1$
\begin{align*}
  &  \mathbb E[|b_{xy}(i_1,...,i_m,j_1,...,j_m)|]\\
    &\leq \Delta_n^{p\frac {m-2}{2m}}\prod_{l=1}^p\left(\mathbb E\left[\left(\int_{(i_{k_l}-1)\Delta_n}^{i_{k_l}\Delta_n}\|\alpha_s\|^{\frac m2}ds\right)^2\right]^{\frac 1m}+C_m^{\frac 1m}  \left(\int_{(i-1)\Delta_n}^{i\Delta_n}\mathbb E[\|\sigma_s\|_{\text{op}}^m]ds\right)^{\frac 1m}\right)^p \\
    &\quad\times\Delta_n^{q\frac {3m-2}{2m}} \prod_{l=1}^q\left(\mathbb E\left[\left(\int_{(i_{k_l}-1)\Delta_n}^{i_{k_l}\Delta_n}\|\alpha_s\|^{\frac m2}ds\right)^2\right]^{\frac 1m}\right.\\
    &\qquad\qquad\qquad\qquad\qquad\qquad\qquad\qquad\left.+ C_m^{\frac 1m}\left(\int_{(i_{k_l}-1)\Delta_n}^{i_{k_l}} \mathbb E\left[\| \sigma_s^*\|_{\text{op}}^m\right]ds\right)^{\frac 1m}\right)^q\notag\\
   &\qquad\times\sup_{t\in[0,T]}\|\mathcal S(t)\|_{\text{op}}^q \|\mathcal A^*e_j\|^q\\
   \leq & \Delta_n^{\frac m2}\prod_{l=1}^m\left(\mathbb E\left[\left(  \int_{(i_l-1)\Delta_n}^{i_l\Delta_n}\|\alpha_s\|^{\frac m2}ds\right)^2\right]^{\frac 1m}+ C_m^{\frac 1m}\left(\int_{(i_l-1)\Delta_n}^{i_l\Delta_n} \mathbb E\left[\|\sigma_s\|_{\text{op}}^m\right]ds\right)^{\frac 1m}\right)\\
   &\qquad\times\left(\sup_{t\in[0,T]q,j=1,...,m}\|\mathcal S(t)\|_{\text{op}}^q \|\mathcal A^*e_j\|^q\right).
\end{align*}
Now we use introduce the notation 
$b_{x,y}(i_1,...,i_m,j_1,...,j_m)=:b_{x,y}(i)^{m_1,...,m_k}$ 
in the case of $i_1,...,i_{m_1}=i$, $i_{m_1+1},...,i_{m_1+m_2}=i+1$ ... $i_{m_{k-2}+1},...,i_{m_{k-2}+m_{k-1}}=i+k-1$ (to help with the intuition, this is just a formal way to specify that the first $m_1$ components are equal and then the next $m_2$ are again equal and so on). 
Then
\begin{align*}
& \langle SAMPV_t^n(m_1,...,m_k),\bigotimes_{j=1}^me_j\rangle\\
&\qquad=\sum_{i=1}^{\ul} \langle \bigotimes_{j=1}^k \tilde{\Delta}_{i+j-1}S,\bigotimes_{j=1}^m e_j\rangle +\sum_{i=1}^{\ul} \sum_{x=2,y=1}^m b_{x,y}(i)^{m_1,...,m_k}.
\end{align*}
In order to prove the assertion, we just have to show that the latter summand is $\mathcal O(\Delta_n^{\frac m2})$. Therefore,
the generalised Hölder inequality and the elementary inequality $(a+b)^m\leq 2^m(a^m+b^m)$ for positive real numbers $a,b\in\mathbb R_+$ yield
\begin{align*}
   & \sum_{i=1}^{\ul-k+1} \sum_{y=1,x=2}^m\mathbb E[|b_{xy}(i)^{m_1,...,m_k}|]\\
    \leq & m(m-1)\Delta_n^{\frac m2}\left(\sup_{t\in[0,T],q,j=1,...,m}\|\mathcal S(t)\|_{\text{op}}^q \|\mathcal A^*e_j\|^q\right) \\
   &\qquad\times \sum_{i=1}^{\ul-k+1} \prod_{j=1}^k\left(\mathbb E\left[\left(\int_{(i+j-2)\Delta_n}^{(i+j-1)\Delta_n} |\|\alpha_s\|^{\frac m2}ds\right)^2\right]^{\frac 1m}\right.\\
   &\qquad\qquad\qquad\qquad\qquad\qquad+ \left.C_m^{\frac 1m}\left(\int_{(i+j-2)\Delta_n}^{(i+j-1)\Delta_n} \mathbb E\left[\|\sigma_s\|_{\text{op}}^m\right]ds\right)^{\frac 1m}\right)^{m_j}\\
   \leq & m(m-1)\Delta_n^{\frac m2}\left(\sup_{t\in[0,T],q,j=1,...,m}\|\mathcal S(t)\|_{\text{op}}^q \|\mathcal A^*e_j\|^q\right) \\
   &\quad\times \sum_{i=1}^{\ul-k+1} \prod_{j=1}^k 2^m\left(\mathbb E\left[\left(\int_{(i+j-2)\Delta_n}^{(i+j-1)\Delta_n} |\|\alpha_s\|^{\frac m2}ds\right)^2\right]\right.\\
   &\qquad\qquad\qquad\qquad\qquad\qquad\left.+ C_m\int_{(i+j-2)\Delta_n}^{(i+j-1)\Delta_n} \mathbb E\left[\|\sigma_s\|_{\text{op}}^m\right]ds\right)^{\frac {m_j}m}.
   \end{align*}
   As $\mathbb E\left[\left(\int_{(i+j-2)\Delta_n}^{(i+j-1)\Delta_n} \|\alpha_s\|^{\frac m2}ds\right)^2\right]\leq \mathbb E\left[\int_{(i+j-2)\Delta_n}^{(i+j-1)\Delta_n} \|\alpha_s\|^{\frac m2}ds\right]$ for $\Delta_n$ small enough, we obtain 
   \begin{align*}
   &\sum_{i=1}^{\ul-k+1}\sum_{y=1,x=2}^m\mathbb E\left[|b_{xy}(i)^{m_1,...,m_k}|\right]\\
   \leq & m(m-1)\Delta_n^{\frac m2}\left(\sup_{t\in[0,T],q,j=1,...,m}\|\mathcal S(t)\|_{\text{op}}^q \|\mathcal A^*e_j\|^q\right)\\
   &\quad\times \sum_{i=1}^{\ul-k+1} \prod_{j=1}^k 2^m\mathbb E\left[\int_{(i+j-2)\Delta_n}^{(i+j-1)\Delta_n} \|\alpha_s\|^{\frac m2}ds+ C_m \|\sigma_s\|_{\text{op}}^mds\right]^{\frac {m_j}m}\\
   \leq & m(m-1)\Delta_n^{\frac m2}\left(\sup_{t\in[0,T],q,j=1,...,m}\|\mathcal S(t)\|_{\text{op}}^q \|\mathcal A^*e_j\|^q\right)\\
   &\quad\times k!  2^m\mathbb E\left[\int_0^T \|\alpha_s\|^{\frac m2}ds+ C_m \|\sigma_s\|_{\text{op}}^mds\right]\\
   \leq & m(m-1)\Delta_n^{\frac m2}\left(\sup_{t\in[0,T],q,j=1,...,m}\|\mathcal S(t)\|_{\text{op}}^q \|\mathcal A^*e_j\|^q\right)k!  A.
   \end{align*}
This proves the assertion
\end{proof}

\subsubsection{Uniform convergence of the finite-dimensional projections}
We introduce the notation 
$$\sigma_s^{\mathcal S_n}=\mathcal S(i\Delta_n-s)\sigma_s,$$
and
$$\alpha_s^{\mathcal S_n}=\mathcal S(i\Delta_n-s)\alpha_s,$$
for $s\in((i-1)\Delta_n,i\Delta_n]$, such that $\sigma_s^{\mathcal S_n}(\sigma_s^{\mathcal S_n})^*=\Sigma_s^{\mathcal S_n}$. We often need the following technical lemma: 
\begin{lemma}\label{L: Projection convergese uniformly on the range of volatility}
Suppose that for $m\in\mathbb N$ we have
$$\int_0^T \mathbb E\left [\|\alpha_s\|^{\frac m2}+\|\sigma_s \|_{L_{\text{HS}}(U,H)}^m\right] ds<\infty,$$
which holds in particular under Assumption \ref{In proof: Very very Weak localised integrability Assumption on the moments}(m). Let $(e_j)_{j\in\mathbb N}$ be an orthonormal basis of $H$ and $p_N$ be the projection onto $ v^N:=\overline{span\{e_{j}:j\geq N\}}$.
 Then for all natural numbers $p\leq m$, we have
$$
\lim_{N\to\infty}\sup_{n\in\mathbb N}\mathbb E\left[\int_0^T\|p_N\sigma_s^{\mathcal S_n}\|_{L_{\text{HS}}(U,H)}^{p}ds\right]= 0.
$$
Moreover, for all $q\leq \frac m2$, we have
$$\lim_{N\to\infty}\sup_{n\in\mathbb N}\mathbb E\left[\int_0^T\|p_N\alpha_s^{\mathcal S_n}\|^{q}ds\right]= 0.$$
\end{lemma}
\begin{proof}
It is enough to prove the assertion for the first limit 
as the second limit can be treated as a special case, if we replace $\alpha^{\mathcal S_n}$ by the Hilbert-Schmidt operator $e\otimes \alpha$ and using that $$\|p_N\alpha_s^{\mathcal S_n}\|=\|e\otimes p_N\alpha_s^{\mathcal S_n}\|,$$
where $e$ is an arbitrary element in $U$ with $\|e\|=1$.
No observe that the Bochner integrability of $\|p_N\sigma_s^{\mathcal S_n}\|_{L_{\text{HS}}(U,H)}^{p}$ is guaranteed by assumption and 
\begin{align}\label{Substitution estimate for the projection convergence}
    \mathbb E\left[\int_0^{\ulT\Delta_n}\|p_N\sigma_s^{\mathcal S_n}\|_{L_{\text{HS}}(U,H)}^{p}ds\right]\leq & \sum_{i=1}^{\ulT} \mathbb E\left[\int_{(i-1)\Delta_n}^{i\Delta_n}\sup_{r\in[0,T]} \left\| p_N\mathcal S(r)\sigma_{s}\right\|^{p}_{L_{\text{HS}}(U,H)} ds\right]\notag\\
    = & \int_0^T \mathbb E\left[\sup_{r\in[0,T]} \left\| p_N\mathcal S(r)\sigma_{s}\right\|^{p}_{L_{\text{HS}}(U,H)} \right] ds.
\end{align}
Now observe that, for $s,r\in[0,T]$ fixed, we have almost surely
\begin{align*}
   \|p_N\mathcal S(r)\sigma_{s}\|_{L_{\text{HS}}(U,H)}^2=\|\sigma_{s}^*\mathcal S(r)^*p_N\|_{L_{\text{HS}}(H,U)}^2= \sum_{k=N}^{\infty} \|\sigma_{s}^*\mathcal S(r)^*e_k\|^2\to 0,\quad \text{ as }N\to\infty,
\end{align*}
since $\sigma_{t}^*\mathcal S(s)^*$ is almost surely a Hilbert-Schmidt operator.
Moreover, the function $f_s(r)= \|p_N\mathcal S(r)\sigma_{s}\|_{L_{\text{HS}}(U,H)}$ is continuous in $r$, as for
\begin{align*}
    |f_s(r_1)-f_s(r_2)|\leq & \left\|p_N(\mathcal S(r_1)-\mathcal S(r_2))\sigma_s\right\|_{L_{\text{HS}}(U,H)}\\
    \leq & \sup_{r\in[0,T]}\|\mathcal S(r)\|_{\text{op}} \sup_{r\leq|r_1-r_2|}\left\|\left(I-\mathcal S(r)\right)\sigma_s\right\|_{L_{\text{HS}}(U,H)}.
\end{align*}
The latter converges to $0$, as $r_1\to r_2$, by Proposition 5.1 in \cite{Benth2022}. 
As we also have
$$\|p_N\mathcal S(s)\sigma_{t}\|_{L_{\text{HS}}(U,H)}\geq  \| p_{N+1}p_N\mathcal S(s)\sigma_{t}\|_{L_{\text{HS}}(U,H)}=\|p_{N+1}\mathcal S(s)\sigma_{t}\|_{L_{\text{HS}}(U,H)},$$
we find by virtue of Dini's theorem (c.f. Theorem 7.13 in \cite{Rudin1976}) that almost surely
$$\lim_{N\to\infty}\sup_{t\in[0,T]}\left\|p_N\mathcal S(r)\sigma_s\right\|_{L_{\text{HS}}(U,H)}=0.$$
By the dominated convergence theorem, this immediately yields
$$ \mathbb E\left[\int_0^{\ulT\Delta_n}\|p_N\sigma_s^{\mathcal S_n}\|_{L_{\text{HS}}(U,H)}^{p}ds\right]\leq\mathbb E\left[ \int_0^T\sup_{r\in[0,T]}\|p_N\mathcal S(r)\sigma_{s}\|_{L_{\text{HS}}(U,H)}^{p}ds\right]\to 0,
$$
as $N\rightarrow\infty$. This proves the claim. 
\end{proof}

\subsection{Various estimates for increments}\label{Drift elimination section}
In this subsection, we will use the notation 
\begin{align}
    \tilde{\Delta}_i^n A := \int_{(i-1)\Delta_n}^{i\Delta_n}\mathcal S(i\Delta_n-s)\alpha_s ds,\\
    \tilde{\Delta}_i^n M := \int_{(i-1)\Delta_n}^{i\Delta_n}\mathcal S(i\Delta_n-s)\sigma_s dW_s.
\end{align}
We will make use of the following Lemma: 

\begin{lemma}\label{L: Auxiliary Estimates for the increments}
Suppose that Assumption \ref{In proof: Very very Weak localised integrability Assumption on the moments}(m) holds.
Let $(e_j)_{j\in\mathbb N}$ be an orthonormal basis of $H$ and $p_N$ be the projection onto 
$ v^N:=\overline{span\{e_{j}:j\geq N\}}$.
Moreover, let 
$$a_N(z):=\sup_{n\in\mathbb N}\mathbb E\left[\int_0^T\|p_N\alpha_s^{\mathcal S_n}\|^zds\right],$$
and
$$b_N(z):=\sup_{n\in\mathbb N}\mathbb E\left[\int_0^T\|p_N\sigma_s^{\mathcal S_n}\|_{L_{\text{HS}}(U,H)}^zds\right],$$
which both converge to $0$ as $N\to\infty$ for $z\leq m$, respectively $z\leq \frac m2$ by Lemma \ref{L: Projection convergese uniformly on the range of volatility}.  Then we can find for all $m\in\mathbb N$ a universal constant $C=C(m)>0$, such that 
\begin{equation}\label{drift increment estimate}
   \sum_{i=1}^{\ul} \mathbb E\left[\|p_N \tilde{\Delta}_i^n A\|^m\right] \leq C\Delta_n^{m-1} a_N(m),
\end{equation}
\begin{equation}\label{vol increment estimate}
    \sum_{i=1}^{\ul} \mathbb E\left[\|p_N \tilde{\Delta}_i^n M\|^m\right] \leq C\Delta_n^{\frac m2-1} b_N(m),
\end{equation}
and 
\begin{equation}\label{whole increment estimate}
 \sum_{i=1}^{\ul}   \mathbb E\left[\|p_N\tilde{\Delta}_i^n Y\|^m\right]\leq C(\Delta_n^{m-1}a_N(m)+\Delta_n^{\frac m2-1} b_N(m))=o(\Delta_n^{\frac m2-1}).
\end{equation}
\end{lemma}
\begin{proof}
Throughout this proof, we treat $C$ as a generic constant that is chosen appropriately large in each step. The majorisation in 
\eqref{drift increment estimate} is an immediate implication of Bochner's inequality and the boundedness of $\alpha$. For \eqref{vol increment estimate}, we get by the BDG inequality \eqref{BDG inequality} that
\begin{align*}
\sum_{i=1}^{\ul}&\mathbb E\left[\|p_N\int_{(i-1)\Delta_n}^{i\Delta_n} \mathcal S(i\Delta_n-s)\sigma_s dW_s\|^m\right]\\
&\leq \sum_{i=1}^{\ul}
C_m\mathbb E\left[\left(\int_{(i-1)\Delta_n}^{i\Delta_n} \|p_N\mathcal S(i\Delta_n-s)\sigma_s\|_{L_{\text{HS}}(U,H)}^2 ds\right)^{\frac m2}\right]\\
& \leq  \sum_{i=1}^{\ul}C_m \Delta_n^{\frac m2-1} \int_{(i-1)\Delta_n}^{i\Delta_n}\mathbb E\left[ \|p_N\mathcal S(i\Delta_n-s)\sigma_s\|_{L_{\text{HS}}(U,H)}^m \right]ds\\
&= C_m \Delta_n^{\frac m2-1} b_N(m).
\end{align*}
  Moreover, inequality \eqref{whole increment estimate} holds as
\begin{align*}
  \sum_{i=1}^{\ul} \mathbb E\left[\|p_N\tilde{\Delta}_i^n Y\|^m\right]
     &\leq  2^{m-1}  \sum_{i=1}^{\ul}\mathbb E\left[\|p_N\int_{(i-1)\Delta_n}^{i\Delta_n}\mathcal S(i\Delta_n-s)\alpha_s ds\|^m\right] \\
     &\qquad+ 2^{m-1}  \sum_{i=1}^{\ul}\mathbb E\left[\|p_N\int_{(i-1)\Delta_n}^{i\Delta_n}\mathcal S(i\Delta_n-s)\sigma_s dW_s\|^m\right]\\
     &\leq  2^{m-1}  \sum_{i=1}^{\ul} \Delta_n^{\frac m2(\frac m2-1)}\mathbb E\left[\left(\int_{(i-1)\Delta_n}^{i\Delta_n}\|p_N\alpha_s^{\mathcal S_n}\|^{\frac m2} ds\right)^{\frac m2}\right] \\
     &\qquad+ 2^{m-1}  \sum_{i=1}^{\ul}\mathbb E\left[\left(\int_{(i-1)\Delta_n}^{i\Delta_n}\|p_N\sigma_s^{\mathcal S_n}\|_{L_{\text{HS}}(U,H)}^2 ds\right)^{\frac m2}\right]\\
     &\leq  2^{m-1}  \sum_{i=1}^{\ul} \Delta_n^{\frac m2(\frac m2-1)}\mathbb E\left[\int_{(i-1)\Delta_n}^{i\Delta_n}\|p_N\alpha_s^{\mathcal S_n}\|^{\frac m2} \right]ds A^{\frac m2-1}\\
     &\qquad+ 2^{m-1}  \Delta_n^{\frac m2-1}\sum_{i=1}^{\ul}\mathbb E\left[\int_{(i-1)\Delta_n}^{i\Delta_n}\|p_N\sigma_s^{\mathcal S_n}\|_{L_{\text{HS}}(U,H)}^m ds\right]\\
     \leq &  \Delta_n^{\frac m2-1}2^{m-1}\left(a_N(\frac m2)A^{\frac m2-1}+ b_N(m)\right).
\end{align*}
\end{proof}

\section{Proof of the laws of large numbers}

\begin{proof}[Proof of Theorems \ref{T: Law of large numbers for tensor power variations} and \ref{T: Law of large numbers for Multitensorpower variations}]
Let $(e_j)_{j\in\mathbb N}$ be an orthonormal basis of $H$, such that $e_j\in\mathcal D(A^*)$ for all $j\in\mathbb N$. Recall that $P_N^d$ is the projection onto $\mathcal V_N^d$, with 
$$\mathcal V_N^d:=\overline{span\{e_{j_1}\otimes\cdots\otimes e_{j_d}:j_i\geq N, i=1,\ldots,d\}}.$$ 
We write $P_N^d=p_N$ if $d=1$. Then we can identify $(I-P_N^m)SAMPV_t^n(m_1,...,m_k)$ with the ``matrix''
$$(\langle SAMPV_t^n(m_1,...,m_k),e_{j_1}\otimes...\otimes e_{j_d}\rangle_{\mathcal H^m})_{(j_1,...,j_d)\in \{1,...,N\}^d}.
$$
To obtain the asymptotic behaviour of this ``matrix'', and so the convergence 
$$(I-P_N^m)\Delta_n^{1-\frac m2}SAMPV_t^n(m_1,...,m_k)\stackrel{u.c.p.}{\longrightarrow} (I-P_N^m)\int_0^t \rho^{\otimes k}_{\Sigma_s}(m_1,...,m_k)ds,$$
it is enough to check the convergence of the $\langle SAMPV_t^n(m_1,...,m_k),e_{j_1}\otimes...\otimes e_{j_m}\rangle_{\mathcal H^m}$, for each $(j_1,...,j_m)\in \{1,...,N\}^m$ separately. 

Fix therefore some
$(e_{p_{i,j}})_{i=1,...,k,j=1,...,m_k}$ with $p_{i,j}\in\mathbb N$ and 
$e_{p_{i,j}}\in \{e_1,...,e_N\}$. 
Using Lemma \ref{L: Reduction to martingales is possible} and its notation, we find that $$ 
\Delta_n^{1-\frac m2}\langle SAMPV_t^n(m_1,...,m_k),\bigotimes_{i=1}^k\bigotimes_{j=1}^{m_j}e_{p_{i,j}}\rangle_{\mathcal H^m}
$$ 
has the same asymptotic behaviour as $\Delta_n^{1-\frac m2}\sum_{i=1}^{\ul-k+1}\prod_{l=1}^{k}\prod_{j=1}^{m_l}\langle \Delta_{i+l-1}^n S, e_{p_{l,j}}\rangle$, where
we recall that
$$\Delta_{i+l-1}^n S=\int_{(i+j-2)\Delta_n}^{(i+j-1)\Delta_n} \alpha_s ds+\int_{(i+j-2)\Delta_n}^{(i+j-1)\Delta_n} \sigma_s dW_s.$$ 
This is however a component of the multipower variation of the multivariate semimartingale $(\langle S_t,e_1\rangle,...,\langle S_t,e_N\rangle)_{t\in]0,T]}$. Thus, in the case of power variations (i.e., under Assumption \ref{As: Decomposition of Q}), Theorem 3.4.1 in \cite{JacodProtter2012} applies, while in the case of multipower variations (i.e., under Assumption \ref{as: localised martingale assumption without volatility}), Theorem 8.4.1 in \cite{JacodProtter2012} applies. Hence, this yields
\begin{align*}
\Delta_n^{1-\frac m2}\langle SAMPV_t^n(m_1,...,m_k),&\bigotimes_{i=1}^k\bigotimes_{j=1}^{m_j}e_{p_{i,j}}\rangle_{\mathcal H^m} \\
&\stackrel{u.c.p.}{\longrightarrow}\langle \int_0^t \rho_{\Sigma_s}^{\otimes k}(m_1,...,m_k)ds,\bigotimes_{i=1}^k\bigotimes_{j=1}^{m_j}e_{p_{i,j}}\rangle_{\mathcal H^m},
\end{align*}
i.e.,
$$(I-P_N^m)\Delta_n^{1-\frac m2}SAMPV_t^n(m_1,...,m_k)\stackrel{u.c.p.}{\longrightarrow} (I-P_N^m)\int_0^t \rho_{\Sigma_s}^{\otimes k}(m_1,...,m_k)ds.$$
This establishes the result for finite-dimensional projections of the multipower variation. 

The triangle inequality yields
\begin{align*}
   \| (SAMPV_t^n&(m_1,...,m_k)-\int_0^t \rho_{\Sigma_s}^{\otimes k}(m_1,...,m_k)ds)\|_{\mathcal H^m}\\
    \leq & \|(I-P_N^m) (SAMPV_t^n(m_1,...,m_k)-\int_0^t \rho_{\Sigma_s}^{\otimes k}(m_1,...,m_k)ds)\|_{\mathcal H^m}\\
    &+ \| P_N^m SAMPV_t^n(m_1,...,m_k)\|_{\mathcal H^m}+\|P_N^m\int_0^t \rho_{\Sigma_s}^{\otimes k}(m_1,...,m_k)ds\|_{\mathcal H^m}.
\end{align*}
We next show the uniform convergence to zero of the latter two terms.
The Markov and generalised Hölder inequality as well as \eqref{whole increment estimate} 
yield for a given $\epsilon>0$,
\begin{align*}
    &\mathbb P\left[\sup_{t\leq T} \Delta_n^{1-\frac m2}\|P_N^m SAMPV_t^n(m_1,...,m_k)\|_{\mathcal H^m}>\epsilon\right]\\
    &\quad\leq  
    \frac 1{\epsilon}\Delta_n^{1-\frac m2}  \mathbb E\left[ \sup_{t\leq T}\left\|\sum_{i=1}^{\ul-k+1}\bigotimes_{j=1}^k(p_N \tilde{\Delta}_{i+j-1}^n Y)^{\otimes m_j}\right\|_{\mathcal H^m}\right]\notag\\
&\quad\leq   \frac 1{\epsilon}\Delta_n^{1-\frac m2} \sum_{i=1}^{\lfloor T/\Delta_n\rfloor-k+1}\mathbb E\left[\prod_{j=1}^k\left\|p_N \tilde{\Delta}_{i+j-1}^n Y\right\|^{m_j}\right]\notag\\
&\quad\leq \frac 1{\epsilon}\Delta_n^{1-\frac m2} \sum_{i=1}^{\lfloor T/\Delta_n\rfloor-k+1}\prod_{j=1}^k\mathbb E\left[\left\|p_N \tilde{\Delta}_{i+j-1}^n Y\right\|^{m}\right]^{\frac {m_j}m}\\
 &\quad \leq\frac 1{\epsilon}\Delta_n^{1-\frac m2} \prod_{j=1}^k\left(\sum_{i=1}^{\lfloor T/\Delta_n\rfloor-k+1}\mathbb E\left[\left\|p_N \tilde{\Delta}_{i+j-1}^n Y\right\|^{m}\right]\right)^{\frac {m_j}m}\\
 &\quad \leq\frac 1{\epsilon}\Delta_n^{1-\frac m2} \prod_{j=1}^k\left( C\Delta_n^{\frac m2-1}(a_N(\frac m2)+ b_N(m))\right)^{\frac {m_j}m}.
\end{align*}
This converges to $0$ as $N\to\infty$ uniformly in $n$. Now, notice that by definition we have 
$$\langle\rho_{p_N\Sigma_sp_N},e_{j_1}\otimes...\otimes e_{j_m}\rangle =\langle\rho_{\Sigma_s},e_{j_1}\otimes...\otimes e_{j_m}\rangle \delta_{j_1,...,j_n\geq N}. $$ Therefore
\begin{align*}
   \left\|\rho_{p_N\Sigma_sp_N}(m)\right\|_{\mathcal H^m}^2
     = &\sum_{j_1,...,j_m\geq N}\langle \rho_{\Sigma_s}(m),e_{j_1}\otimes ...\otimes e_{j_m}\rangle^2\\
     = &\sum_{j_1,...,j_m\geq N}(\sum_{p\in\mathcal P(m)}\prod_{(k,l)\in p} \langle \Sigma_s e_{j_l},e_{j_k}\rangle)^2\\
     \leq &|\mathcal P(m)|\sum_{p\in\mathcal P(m)}\sum_{j_1,...,j_m\geq N}\prod_{(k,l)\in p} \langle \Sigma_s e_{j_l},e_{j_k}\rangle^2\\
     \leq & |\mathcal P(m)|\sum_{p\in\mathcal P(m)}\sum_{j_1,...,j_m\geq N}\prod_{(k,l)\in p} \| \Sigma_s^{\frac 12} e_{j_l}\|^2\|\Sigma_s^{\frac 12}e_{j_k}\|^2.\\
     = & |\mathcal P(m)|^2(\sum_{j\geq N}\|\Sigma_s^{\frac 12}e_{j}\|^2)^{ m},
\end{align*}
which converges to $0$ almost surely as $N\to\infty$, as $\Sigma_s^{\frac 12}$ is a Hilbert-Schmidt operator.
Hence, by the definition of $\rho_{\Sigma_s}^{\otimes k}(m_1,...,m_k)$, it holds,
\begin{align*}
    &\mathbb P\left[\sup_{t\leq T}\left\|P_N^m \int_0^t \rho_{\Sigma_s}^{\otimes k}(m_1,...,m_k)ds\right\|_{\mathcal H^m}>\epsilon\right]\\
    &\qquad\leq \frac 1{\epsilon}\int_0^T \mathbb E\left[\left\|P_N^m \rho_{\Sigma_s}^{\otimes k}(m_1,...,m_k)\right\|_{\mathcal H^m}\right]ds\\
    &\qquad=\frac 1{\epsilon}\int_0^T \mathbb E\left[\left(\sum_{j_1,...,j_m\in\mathbb N}\langle  \rho_{p_N\Sigma_sp_N}^{\otimes k}(m_1,...,m_k), e_1\otimes...\otimes e_m\rangle^2\right)^{\frac 12}\right]ds\\
    &\qquad=\frac 1{\epsilon}\int_0^T \mathbb E\left[\prod_{l=1}^k\left(\sum_{j_1,...,j_{m_l}\in\mathbb N}\langle  \rho_{p_N\Sigma_sp_N}(m_l), e_1\otimes...\otimes e_{m_l}\rangle^2\right)^{\frac 12}\right]ds \\
    &\qquad=\frac 1{\epsilon}\int_0^T \mathbb E\left[\prod_{l=1}^k\left\|  \rho_{p_N\Sigma_sp_N}(m_l)\right\|_{\mathcal H^m}\right]ds \\
    &\qquad\leq \frac {|\mathcal P(m)|}{\epsilon}\int_0^T  \mathbb E\left[\left(\sum_{j\geq N}\|\Sigma_s^{\frac 12}e_{j}\|^2\right)^{ \frac m2}\right]ds.
\end{align*}
 This converges to $0$ by the Dominated Convergence Theorem.
 
 Summing up, we can, for each $\delta>0$, find an $N\in\mathbb N$ independent of $n$, such that
 \begin{align*}
&\lim_{n\to\infty}  \mathbb P\left[\sup_{t\leq T}\left(\| (SAMPV_t^n(m_1,...,m_k)-\int_0^t \rho_{\Sigma_s}^{\otimes k}(m_1,...,m_k)ds)\|_{\mathcal H^m}\right)>\epsilon\right]\\
\leq & \lim_{n\to\infty} \mathbb P\left[\sup_{t\leq T}\left(\|(I-P_N^m) (SAMPV_t^n(m_1,...,m_k)-\int_0^t \rho_{\Sigma_s}^{\otimes k}(m_1,...,m_k)ds)\|_{\mathcal H^m}\right)>\epsilon\right]\\
      &+\limsup_{n\to\infty}\left(
       \mathbb P\left[\sup_{t\leq T}\left(\|P_N^m SAMPV_t^n(m_1,...,m_k)\|_{\mathcal H^m}\right)>\epsilon\right]\right.\\
       &\qquad + \left.\mathbb P\left[\sup_{t\leq T}\left(\|P_N^m \int_0^t \rho_{\Sigma_s}^{\otimes k}(m_1,...,m_k)ds\|_{\mathcal H^m}\right)>\epsilon\right]\right)
      \\
      \leq & \lim_{n\to\infty} \mathbb P\left[\sup_{t\leq T}\left(\|(I-P_N^m) (SAMPV_t^n(m)-\int_0^t \rho^{\otimes k}_{\Sigma_s}(m_1,...,m_k)ds)\|_{\mathcal H^m}\right)>\epsilon\right]+2\delta\\
      = & 2\delta.
 \end{align*}
 This holds for all $\delta>0$, and hence the assertion follows. 
\end{proof}

\section{Proof of the central limit theorems}\label{sec: Proofs of CLTs}
We prove the central limit theorems by proving the tightness of the laws of the processes in the Skorokhod space $\mathcal D([0,T],H)$ first and then make use of the available finite-dimensional asymptotic limit theory in order to show convergence of the corresponding finite-dimensional distributions.

\subsection{A short primer on tightness}

Recall that a sequence of measures $(\mu_n)_{n\in\mathbb N}$ is tight on a Polish space $B$ equipped with its Borel $\sigma$-algebra $(B,\mathcal B(B))$, if for each $\epsilon>0$ there is a compact set $K_{\epsilon}\subset B$ such that $\sup_{n\in\mathbb N}\mu_n(B\setminus K_{\epsilon})<\epsilon$.
We will say that a sequence $(X_n)_{n\in\mathbb N}$ of Borel-measurable random variables in $B$ (e.g.~stochastic processes) is tight if the underlying sequence of laws $(\mu_{X_n})_{n\in\mathbb N}$ is tight.

\subsubsection{Tightness of random elements in the Skorokhod space $\mathcal D([0,T],H)$}

For the convenience of the reader, we repeat the following tightness criterion from 
\cite[p.35]{Joffe1986}. 
\begin{theorem}\label{T: Aldous Tightness Theorem} Let $H$ be a separable Hilbert space.
The family of laws $(\mathbb P_{\psi^n})_{n\in\mathbb N}$ of a sequence of random variables $(\psi^n)_{n\in\mathbb N}$ in $\mathcal D([0,T],H)$ is tight if the following two conditions hold:
\begin{itemize}
    \item[(i)] $(\mathbb P_{\psi^n_t})_{n\in\mathbb N}$ is tight for each $t\in [0,T]$ and
    \item[(ii)](Aldous' condition) For all $\epsilon,\eta>0$ there is an $\delta>0$ and $n_0\in\mathbb N$ such that for all sequences of stopping times $(\tau_n)_{n\in \mathbb N}$ with $ \tau_n\leq T-\delta$ we have 
    \begin{equation}
        \sup_{n\geq n_0}\sup_{\theta\leq\delta}\mathbb P\left[\|\psi_{\tau_n}^n-\psi_{\tau_n+\theta}^n\|_H > \eta\right]\leq \epsilon.
    \end{equation}
\end{itemize}
\end{theorem}
Regarding point (i) above, to show tightness in the space $\mathcal D([0,T],H)$ it is necessary to find criteria for the tightness in the Hilbert space itself.
This can be approached by an {\it equi-small tails}-argument and is well known (c.f. Lemma 1.8.1 in \cite{vanderVaart1996}): 
\begin{theorem}\label{T: Tightness Criterion in Hilbert spaces}
Let $(Y_n)_{n\in\mathbb N}$ be a sequence of random variables on a probability space $(\Omega,\mathcal F,\mathbb P)$ with values in a separable Hilbert space $H$, such that for all $\delta>0$
\begin{equation}
  \lim_{N\to\infty}\sup_{n\in\mathbb N}  \mathbb P\left[\sum_{k\geq N} \langle Y_n, e_k\rangle^2>\delta\right] = 0,
\end{equation}
for some orthonormal basis $(e_n)_{n\in\mathbb N}$. Then the sequence $(\mathbb P_{Y_n})_{n\in\mathbb N}$ is tight in $H$.
\end{theorem}
\begin{proof}
Fix some $\epsilon>0$.
By assumption we can define two  
increasing sequences of natural numbers $(N_k^{\epsilon})_{k\in\mathbb N}$ and $(l_k)_{k\in\mathbb N}$, such that $N_1^{\epsilon}=1$ and
\begin{equation}
\sup_{n\in\mathbb N}\mathbb P\left[  \sum_{l\geq N_k^{\epsilon}} \langle Y_n, e_l\rangle^2> \frac 1{l_k}\right] \leq \epsilon\frac 1{ l_k^2\sum_{j=1}^{\infty} \frac 1{l_j^2}}.
\end{equation}
Further, we introduce 
\begin{equation}
  A_k^{\epsilon}:=  \left\lbrace h\in H: \sum_{l\geq N_k^{\epsilon}}\langle h,e_l\rangle^2\leq \frac 1{l_k}\right\rbrace. 
\end{equation}
We prove now that $K_{\epsilon}=\bigcap_{k\in\mathbb N}A_k^{\epsilon}$ is compact. It is obviously closed and bounded 
Then we have as $k\to \infty$
\begin{align*}
 \sup_{h\in K_{\epsilon}} \sum_{l\geq N_k^{\epsilon}}\langle h,e_l\rangle^2
    \leq  \frac 1{l_k}\to 0.
\end{align*}
Hence, the set $K_{\epsilon}$ is totally bounded and by the Hausdorff theorem (c.f. Theorem 3.28 in \cite{Aliprantis2006}) compact. 

It is now left to show that $ 1-\mathbb{P}_{Y_n}[K_{\epsilon}]<\epsilon$. But, by Markov's inequality and the choice of $N_k^{\epsilon}$ we have
\begin{align*}
     1-\mathbb{P}_{Y_n}[K_{\epsilon}]
    \leq  \sum_{k=1}^{\infty}\mathbb P_{Y_n}[(A_k^{\epsilon})^c]\leq \epsilon,
\end{align*}
which proves the claim. 
\end{proof}
By Markov's inequality, we have the following Corollary to Theorem \ref{T: Tightness Criterion in Hilbert spaces},
\begin{corollary}\label{C: Tightness Criterion in Hilbert spaces}
Let $(Y_n)_{n\in \mathbb N}$ be a sequence of random variables on a probability space $(\Omega,\mathcal F,\mathbb P)$ with values in a separable Hilbert space $H$ and having finite second moments.  If for some orthonormal basis $(e_n)_{n\in\mathbb N}$ we have
 \begin{equation}\label{Weak Spatial Tightness Criterion in Hilbert space}
  \lim_{N\to\infty}\sup_{n\in\mathbb N}  \sum_{k\geq N} \mathbb E\left[\langle Y_n, e_k\rangle^2 \right]= 0,
\end{equation}
 then the sequence $(\mathbb P_{Y_n})_{n\in\mathbb N}$ is tight in $H$.
\end{corollary}

\subsubsection{Tightness and stable convergence}
Let $E$ be a Polish space, $\mathcal E:=\mathcal B(E)$ its Borel $\sigma$-algebra and $(\Omega,\mathcal F,\mathbb P)$ be a probability space. Recall that a map $K:\Omega\times \mathcal E\to[0,1]$ is called a Markov kernel from $(\Omega,\mathcal F)$ to $(E,\mathcal E)$, if for all $\omega \in \Omega$ the map $K(\omega,\cdot)$ is a Borel probability measure on $E$ and for all $A\in \mathcal E$ the map $K(\cdot,A)$ is an $\mathcal F$-measurable random variable. 

Let $(Y_n)_{n\in\mathbb N}$ be a sequence of random variables with values in the Skorokhod space $\mathcal D([0,T],H)$ defined on a probability space $(\Omega,\mathcal F,\mathbb P)$ and $Y$ a random variable with values in $\mathcal D([0,T],H)$ defined on an extension $(\tilde{\Omega},\tilde{\mathcal F},\tilde{\mathbb P})$ of $(\Omega,\mathcal F,\mathbb P)$. 
Observe, that we can specify a Markov kernel $K$ by the conditional distribution
$$K(\omega, A)=\tilde{\mathbb P}[Y\in A|\mathcal F].$$
One can then see that stable convergence of the sequence $(Y_n)_{n\in\mathbb N}$ to $Y$ can be written as
\begin{equation}\label{stable convergence as convergence to Markov kernel}
    \mathbb E[Zf(Y_n)]\to \tilde{\mathbb E}[Zf(Y)]= \int_{\Omega} Z(\omega) \int_{\mathcal D([0,T],H)} f(x) K(\omega,dx)\mathbb P[d\omega],\quad \text{as } n\to \infty
\end{equation}
for all bounded continuous functions $f:\mathcal D([0,T],H)\to \mathbb R$ and all bounded random variables $Z$ on $(\Omega,\mathcal F)$. In that way we can identify stable convergence of a sequence of random variables as convergence towards a Markov kernel in the sense of \eqref{stable convergence as convergence to Markov kernel}. We will use this in the proof of the next theorem, which can be found in \cite[Proposition 3.9]{Hausler2015}  for continuous processes. Here we extend the proof for processes with values in the Skorokhod space.
\begin{theorem}\label{T: stable convergence of finite-dimensional distributions yields stable convergence}
Let $(Y_n)_{n\in\mathbb N}$ be a sequence of random variables with values in the Skorokhod space $\mathcal D([0,T],H)$ defined on a probability space $(\Omega,\mathcal F,\mathbb P)$ and $Y$ a random variable with values in $\mathcal D([0,T],H)$ defined on an extension $(\tilde{\Omega},\tilde{\mathcal F},\tilde{\mathbb P})$ of $(\Omega,\mathcal F,\mathbb P)$.
If $(Y_n)_{n\in\mathbb N}$ is tight and $(Y_n(t_1),...,Y_n(t_d))\to(Y(t_1),...,Y(t_d)) $ stably for each finite collection $t_1,...,t_d\in [0,T]$, $d\in\mathbb N$, then $Y_n\to Y$ stably as $n\to \infty$.
\end{theorem}
\begin{proof}
Assume $(Y_n)_{n\in\mathbb N}$ is tight and $(Y_n(t_1),...,Y_n(t_d))\to(Y(t_1),...,Y(t_d)) $ stably for each finite collection $t_1,...,t_d\in [0,T]$ and $(Y_n)_{n\in\mathbb N}$ does not converge stably to $Y$. Then we can find a subsequence $(n_k)_{k\in\mathbb N}$, $\epsilon>0$, a bounded real-valued random variable $Z$ on $(\Omega,\mathcal F,\mathbb P)$ and $h\in C_b(\mathcal D([0,T], H))$, such that 
\begin{equation}
    |\mathbb E[Z h(Y_{n_k})]-\tilde{\mathbb E}[Z h(Y) ]|\geq \epsilon.
\end{equation}
Equivalently, this means for the kernel 
$$K(\omega, A):=\tilde{\mathbb P}[Y\in A|\mathcal F](\omega)$$
that 
\begin{equation}\label{Contradiction assumption for tightness and stabe convergence argument}
    |\mathbb E[Z h(Y_{n_k})]-\int_{\Omega} Z(\omega) \int_{\mathcal D([0,T],H)} f(x) K(\omega,dx)\mathbb P[d\omega]|\geq \epsilon.
\end{equation}
By the tightness of $(Y_n)_{n\in\mathbb N}$, we can appeal to Theorem 3.4(a) in \cite{Hausler2015} and obtain a subsequence $(n_{k_l})_{l\in\mathbb N}$ of $(n_k)_{k\in\mathbb N}$, such that $ Y_{n_{k_l}}\to L$ stably for some Markov kernel $L:\Omega\times \mathcal B(\mathcal D([0,T],H))\to \mathbb [0,1]$ as $n_{k_l}\to\infty$.
Then we have for all $F\in\mathcal F$ with $\mathbb P(F)>0$ that $$\mathbb P^F\circ Y_{n_{k_l}}^{-1}\stackrel{d}{\to} \mathbb P^F[L]:=\int_{\Omega} L(\omega, \cdot) \mathbb P^F[d\omega]$$
weakly by Theorem 3.2 (iv) in \cite{Hausler2015}, where $\mathbb P^F[A]:=\frac{\mathbb P[A\cap F]}{\mathbb P[F]}$ is the conditional probability. According to \cite[p.138-139]{Billingsley} there is then for each $F\in\mathcal F$ with $\mathbb P(F)>0$ a dense set $T_{\mathbb P^F[L]}\subset [0,T]$ (depending on the limiting distribution), that contains $0$ and $T$ and
\begin{equation}\label{Eq: finite dim convergence for Skorokhod space on Tp}
    \mathbb P^F\circ Y_{n_{k_l}}^{-1}\circ(\pi_{t_1,...,t_d})\stackrel{d}{\to} \mathbb P^F[L]\circ(\pi_{t_1,...,t_d})
\end{equation}
whenever $t_1,...,t_d\in T_{\mathbb P^F[L]}$ for $d$ arbitrary. Here, $\pi_{t_1,...,t_d}(f)=(f(t_1),...,f(t_d))$ denotes the finite-dimensional projections.
 By Theorem 12.5 in \cite{Billingsley}, the sets 
$$\pi_{t_1,...,t_d}^{-1}(A),\quad A\in\mathcal B( H^d), t_1,...,t_d\in T_{\mathbb P^F[L]}$$
generate $\mathcal B(\mathcal D([0,T],H))$, where $H^d=H\times...\times H$ is equipped with the product topology, and, in particular, two measures $\mathbb Q_1$ and $\mathbb Q_2$ coincide on $\mathcal B(\mathcal D([0,T],H))$, if $\mathbb Q_1\circ\pi_{t_1,...,t_d}^{-1}=\mathbb Q_2\circ\pi_{t_1,...,t_d}^{-1} $ for all $t_1,...,t_d\in T_{\mathbb P^F[L]}$, $d\in\mathbb N$. By assumption we have that $(Y_n(t_1),...,Y_n(t_d))\to(Y(t_1),...,Y(t_d)) $ stably, which is equivalent to
$$\mathbb P^F\circ Y_n^{-1}\circ \pi_{t_1,...,t_d}\stackrel{d}{\to}\tilde{\mathbb P}^F\circ Y^{-1}\circ \pi_{t_1,...,t_d}$$
for all $t_1,...,t_d\in [0,T]$, $d\in\mathbb N$ and all $F\in\mathcal F$ with $\mathbb P(F)>0$ by \cite[Theorem 3.4(iv)]{Hausler2015}. This together with \eqref{Eq: finite dim convergence for Skorokhod space on Tp} yields 
$$\tilde{\mathbb P}^F\circ Y^{-1}\circ \pi_{t_1,...,t_d}=\mathbb P^F[L]\circ(\pi_{t_1,...,t_d})$$
 for all $t_1,...,t_d\in T_{\mathbb P^F[L]}$, $d\in\mathbb N$ and, hence,
 $$\mathbb P^F[K]=\int_{\Omega} K(\omega,\cdot)\mathbb P^F[d\omega]=\tilde{\mathbb P}^F\circ Y^{-1}=\mathbb P^F[L] $$
 for all $F\in\mathcal F$ with $\mathbb P(F)>0$. However, this shows that $K=L$, which is a contradiction, as by construction of $L$ it is
 $$
    \mathbb E[Zf(Y_{n_{k_{l}}})]\to  \int_{\Omega} Z(\omega) \int_{\mathcal D([0,T],H)} f(x) L(\omega,dx)\mathbb P[d\omega],\quad \text{as } l\to \infty.$$
\end{proof}

\subsection{Tightness results for the central limit theorems}
In this section, we are going to prove in several steps the following theorem. Recall the notation $\Sigma_s^{\mathcal S_n}:=\mathcal S(i\Delta_n-s)\Sigma_s\mathcal S(i\Delta_n-s)^*$ for $s\in [(i-1)\Delta_n,i\Delta_n)$, that we will use extensively here.
\begin{theorem}\label{T: Tightness for the quadratic variation}
Let Assumption \ref{As: Weakened Assumption for the CLT for the SARCV} hold. Then the sequence of processes 
$$(\tilde{Z}^{n,2}_t)_{t\in[0,T]}:= \left(\Delta_n^{-\frac 12}\left(\sum_{i=1}^{\ul}\tilde{\Delta}_i^n Y^{\otimes 2}-\int_{(i-1)\Delta_n}^{i\Delta_n} \Sigma_s^{\mathcal S_n}ds\right)\right)_{t\in[0,T]},\quad n\in\mathbb N$$
is tight in $\mathcal D([0,T],\mathcal H).$
\end{theorem}

Despite the rather extensive notation, it is relatively straightforward to show that $(\tilde{Z}^{n,2})_{t\in [0,T]}$ 
satisfies Aldous' condition. 
\begin{theorem}(Temporal tightness)\label{T: Verfication of Aldous' condition}
 Let $(\mathbb P_n)_{n\in\mathbb N}$ be given by $\mathbb P_n=\mathbb P_{( \tilde Z_t^{
n,2})_{t\in [0,T]}}$ and Assumption \ref{As: Weakened Assumption for the CLT for the SARCV} hold. Then $(\mathbb P_n)_{n\in\mathbb N}$ satisfies Aldous' condition.
\end{theorem}
\begin{proof} 
The Markov inequality yields
\begin{align}\label{Tightness for Feasible Estimator:Triangle Inequality applied}
  & \mathbb P\left[\left\|\tilde Z^{n,2}_{\tau_n}-Z^{n,2}_{\tau_n+\theta}\right\|_{\mathcal H}>\eta\right]\notag\\
  \leq & \frac 1{\eta}\mathbb E\left[\left\|\tilde Z^{n,2}_{\tau_n}-\tilde Z^{n,2}_{\tau_n+\theta}\right\|_{\mathcal H}\right]\notag\\
  \leq  & \frac {1}{\eta} \left(\Delta_n^{-\frac 12}\mathbb E\left[ \left\|\sum_{i=\lfloor \tau_n/\Delta_n\rfloor}^{\lfloor (\tau_n+\theta)/\Delta_n\rfloor} \tilde{\Delta}_{i}^nY^{\otimes 2}\right\|_{\mathcal H}+\Delta_n^{-\frac 12}\left\|\int_{ \tau_n}^{ \tau_n+\theta}\Sigma_s ds\right\|_{\mathcal H}\right]\right).
  \end{align}
  Now, set $\theta<\delta<\Delta_n$.
We can estimate further
  \begin{align*}
   & \mathbb E\left[ \left\|\sum_{i=\lfloor \tau_n/\Delta_n\rfloor}^{\lfloor (\tau_n+\theta)/\Delta_n\rfloor} \tilde{\Delta}_{i}^nY^{\otimes 2}\right\|_{\mathcal H}+\left\|\int_{ \tau_n}^{ \tau_n+\theta}\Sigma_s ds\right\|_{\mathcal H}\right]\\
  &\quad\leq \left( \mathbb E \left[\left\|\tilde{\Delta}_{\lfloor \tau_n+\theta/\Delta_n\rfloor}^n Y^{\otimes 2}\right\|_{\mathcal H}+\left\|\tilde{\Delta}_{\lfloor \tau_n/\Delta_n\rfloor}^n Y^{\otimes 2}\right\|_{\mathcal H}\right]\right)\\
  &\qquad\qquad+\mathbb E [\int_{ \tau_n}^{ \tau_n+\theta}\left\|\Sigma_s\right\|_{\mathcal H}ds]\\
  &\quad= \left( \mathbb E \left[\left\|\tilde{\Delta}_{\lfloor \tau_n+\theta/\Delta_n\rfloor}^n Y\right\|^2+\left\|\tilde{\Delta}_{\lfloor \tau_n/\Delta_n\rfloor}^n Y\right\|^2\right]\right)+\mathbb E [\int_{ \tau_n}^{ \tau_n+\theta}\|\Sigma_s\|_{\mathcal H}ds]\\
  &\quad= (1)_n+(2)_n.
\end{align*}
We obtain
\begin{align}\label{Martingale Difference Stopping inequality}
   &\Delta_n^{-\frac 12} \mathbb E \left[\left\|\tilde{\Delta}_{\lfloor \tau_n+\theta/\Delta_n\rfloor}^n Y\right\|^2\right]\notag\\
   &\leq \Delta_n^{-\frac 12} 2\left( \mathbb E\left[\left\|\int_{(\lfloor\tau_n+\theta/\Delta_n\rfloor)\Delta_n}^{(\lfloor\tau_n+\theta/\Delta_n\rfloor+1)\Delta_n}\alpha_s^{\mathcal S_n}  ds\right\|^2\right]\right.\notag\\
    &\qquad+ \left.\mathbb E\left[\left\|\int_{(\lfloor\tau_n+\theta/\Delta_n\rfloor)\Delta_n}^{(\lfloor\tau_n+\theta/\Delta_n\rfloor+1)\Delta_n}\sigma_s^{\mathcal S_n}  dW_s\right\|^2\right]\right)\notag\\
   &\leq  \Delta_n^{-\frac 12} 2\left( \mathbb E\left[\Delta_n^{-1}\int_{(\lfloor\tau_n+\theta/\Delta_n\rfloor)\Delta_n}^{(\lfloor\tau_n+\theta/\Delta_n\rfloor+1)\Delta_n}\left\|\alpha_s^{\mathcal S_n}\right\|^2  ds\right]\right.\notag\\
   &\left.\qquad + \mathbb E\left[\int_{(\lfloor\tau_n+\theta/\Delta_n\rfloor)\Delta_n}^{(\lfloor\tau_n+\theta/\Delta_n\rfloor+1)\Delta_n}\|\sigma_s \|_{L_{\text{HS}}(U,H)}^2 ds\right]\right)\notag\\
    &\leq 2 \Delta_n^{\frac 12} \sup_{r\in [0,T]}\|\mathcal S(r)\|_{\text{op}}^2\int_0^T \mathbb E\left[\|\alpha_s\|^2\right] ds  \notag\\
    &\qquad+2 
    \mathbb E\left[\left(\int_{(\lfloor\tau_n+\theta/\Delta_n\rfloor)\Delta_n}^{(\lfloor\tau_n+\theta/\Delta_n\rfloor+1)\Delta_n}\|\sigma_s \|_{L_{\text{HS}}(U,H)}^4 ds\right)^{\frac 12}\right],
\end{align}
which converges to $0$ as $n\to \infty$ since the function $t\mapsto\int_0^t \|\sigma_s\|_{L_{\text{HS}}(U,H)}^4 ds$ is uniformly continuous on $[0,T]$ and bounded.

Analogously we obtain $$\Delta_n^{-\frac 12}\mathbb E \left[\|\tilde{\Delta}_{\lfloor \tau_n/\Delta_n\rfloor}^n Y^{\otimes 2}\|_{\mathcal H}\right]\to 0,$$ as $n\to\infty$. This yields $\lim_{n\to \infty}\Delta_n^{-\frac 12}(1)_n=0$. 

It remains
 to show $\lim_{n\to \infty}\Delta_n^{-\frac 12}(2)_n= 0$. Observe that since $s\mapsto\Sigma_s$ is bounded by assumption,
 the following convergence holds almost surely as $n\to\infty$:
\begin{align*}
\Delta_n^{-\frac 12}\int_{ \tau_n}^{ \tau_n+\theta}\|\Sigma_s\|_{\mathcal H}ds
  \leq  \left(\int_{ \tau_n}^{ \tau_n+\theta}\|\Sigma_s\|_{\mathcal H}^2ds\right)^{\frac 12}
&\leq  \sup_{t\in[0,T-\theta]}\left(\int_{ t}^{ t+\theta}\|\Sigma_s\|_{\mathcal H}^2ds\right)^{\frac 12}\\
\leq &\sup_{t,s \in [0,T], t-s\leq \Delta_n} \left(\int_s^t\|\Sigma_s\|_{\mathcal H}^2ds\right)^{\frac 12}
\end{align*}
converges to zero, as $t\mapsto \int_0^{ t}\|\Sigma_s\|_{\mathcal H}^2ds$ is uniformly continuous on $[0,T]$ almost surely, as $\theta\leq \Delta_n$. 
Moreover, we have the integrable majorant $$\int_{\tau_n}^{\tau_n+\theta}\|\Sigma_s\|_{\mathcal H}ds\leq \int_0^T\|\Sigma_s\|_{\mathcal H}ds,$$
such that we obtain by dominated convergence $\Delta_n^{-\frac 12}(2)\to 0$ as $n\to \infty$. Hence, we conclude the Aldous condition in Theorem \ref{T: Aldous Tightness Theorem}.
\end{proof}

In order to show tightness of the sequences $(\tilde Z^{n,2})_{n\in\mathbb N}$ in $\mathcal D([0,T],\mathcal H)$ under the conditions of Theorem \ref{T: Tightness for the quadratic variation}, we have to verify the tightness of each $( \tilde Z^{n,2}_t)_{n\in\mathbb N}$ in $\mathcal H$ for each $t\in [0,T]$ separately. 
This is what we do in the remainder of this subsection.
We first argue that we can assume $\alpha\equiv 0$ in the proof of tightness of $\tilde{Z}_t^{n,2}$ in $\mathcal  H$. Observe that with the notation of Section \ref{Drift elimination section}, we have 
\begin{align}\label{Decomposition into semimartingale and residual term of SARCV}
&\Delta_n^{-\frac 12}\left(SARCV_t^n-\int_0^t \Sigma_s^{\mathcal S_n} ds\right)\notag\\
= & \Delta_n^{-\frac 12}\left(\sum_{i=1}^{\ul} \left(\tilde{\Delta}_i^n A+\tilde{\Delta}_i^n M\right)^{\otimes 2}-\int_0^t \Sigma_s^{\mathcal S_n}ds\right)\notag\\
   =&  \Delta_n^{-\frac 12}\left(\sum_{i=1}^{\ul}\tilde{\Delta}_i^n A^{\otimes 2}+\tilde{\Delta}_i^n A\otimes \tilde{\Delta}_i^n M+\tilde{\Delta}_i^n M\otimes \tilde{\Delta}_i^n A\right)\notag\\
   &+\Delta_n^{-\frac 12}\left(\sum_{i=1}^{\ul}\tilde{\Delta}_i^n M^{\otimes 2}-\int_0^t \Sigma_s^{\mathcal S_n}ds\right)\notag\\
   =: & (I)_t^n+(II)_t^n.
\end{align}
We obtain: 
\begin{theorem}[Elimination of the drift]\label{T: Drift elimination for tightness}
Suppose that Assumption \ref{In proof: Very very Weak localised integrability Assumption on the moments}(2m) holds.
In this case, the first summand in \eqref{Decomposition into semimartingale and residual term of SARCV}, that is $(I)_t^n$ is tight. In particular, in order to show Theorem \ref{T: Tightnes for realised covariation} we can assume $\alpha\equiv 0$.
\end{theorem}
\begin{proof}
We show first that 
\begin{equation}\label{Tightness moment vanishing criterion for the drift elimination}
   \lim_{N\to\infty}\sup_{n\in\mathbb N}\mathbb E\left[\sup_{t\in [0,T]}\left\|P_N^2 (I)_t^n\right\|_{\mathcal H}\right]=0. 
\end{equation}
We can compute, using Hölder's inequality 
\begin{align*}
 &   \mathbb E\left[\sup_{t\in [0,T]}\left\|P_N^2 (I)_t^n\right\|_{\mathcal H}\right]\\
    \leq & \Delta_n^{-\frac 12}\sum_{i=1}^{\ulT} \mathbb E\left[\|P_N^2\tilde{\Delta}_i^n A\|^2+2\|P_N^2\tilde{\Delta}_i^n A\|\|P_N^2\tilde{\Delta}_i^n M\|\right]\\
    = & \Delta_n^{-\frac 12}\sum_{i=1}^{\ulT} \mathbb E\left[\|P_N^2\tilde{\Delta}_i^n A\|^2\right]\\
    &\qquad+2\Delta_n^{-\frac 12}\left(\sum_{i=1}^{\ulT}\mathbb E\left[\|P_N^2\tilde{\Delta}_i^n A\|^2\right]\right)^{\frac 12}\left(\sum_{i=1}^{\ulT}\mathbb E\left[\|P_N^2\tilde{\Delta}_i^n M\|\right]^2\right)^{\frac 12}.
\end{align*}
Now using \eqref{drift increment estimate} and \eqref{vol increment estimate} as well as the corresponding notation from Lemma \ref{L: Auxiliary Estimates for the increments} 
we find
\begin{align*}
    \mathbb E\left[\sup_{t\in [0,T]}\left\|P_N^2 (I)_t^n\right\|_{\mathcal H}\right]
    \leq  \Delta_n^{\frac 12}C \Delta_n a_N(2)
    +2\left(C a_N(2)\right)^{\frac 12}\left(C b_N(2)\right)^{\frac 12}
\end{align*}
The latter converges to $0$, as $N\to\infty$, uniformly in $n$. Thus \eqref{Tightness moment vanishing criterion for the drift elimination} holds.
 A straightforward application of Markov's inequality and Theorem \ref{T: Tightness Criterion in Hilbert spaces} yield the assertion.
\end{proof}

In the subsections below we make the remaining steps in order to prove Theorem \ref{T: Tightnes for realised covariation}, i.e., in view of Theorem \ref{T: Aldous Tightness Theorem} and Theorem \ref{T: Verfication of Aldous' condition} we have to prove that $(\tilde Z_t^{n,2})_{n\in\mathbb N}$ is tight in $\mathcal H$. In view of Theorem \ref{T: Drift elimination for tightness} we further assume that $\alpha\equiv 0$ throughout these subsections.

\subsubsection{Spatial tightness for quadratic variation}\label{sec: Spatial tightness for realised covariation}

Recalling that Assumption \ref{As: Weakened Assumption for the CLT for the SARCV} is satisfied under the assumptions of Theorem \ref{T: Tightnes for realised covariation}, the tightness of the sequence of laws corresponding to  $(\tilde{Z}^{n,2})_{n\in\mathbb N}$ is tight in $\mathcal D([0,T],\mathcal H)$ by 
\begin{theorem}\label{T: Tightnes for realised covariation}
Assume that
$$\int_0^T\mathbb E\left[\left\|\sigma_s^{\mathcal S_n}\right\|_{L_{\text{HS}}(U,H)}^4\right] ds<\infty,$$
which is in particular the case, if Assumption \ref{As: Weakened Assumption for the CLT for the SARCV} holds. We have
\begin{align}\label{slightly better than tightness for SARCV-centered by conditional expectation}
    & \lim_{N\to\infty}\sup_{t\in[0,T]}\sup_{n\in\mathbb N}\sum_{m,k\geq N}\mathbb E\left[ \langle \tilde Z_t^{2,n}, e_k\otimes e_m\rangle_{\mathcal H}^2\right]\notag\\
    = & \lim_{N\to\infty}\sup_{t\in[0,T]}\sup_{n\in\mathbb N}\mathbb E\left[\|p_N\tilde Z_t^{2,n}\|_{\mathcal H}^2\right]=0,
\end{align}
and thus, the sequence $(\tilde{Z}_t^{n,2})_{n\in\mathbb N}$ 
is tight in $\mathcal H$. Hence, the sequence $(\tilde{Z}^{n,2})_{n\in\mathbb N}$ is tight in $\mathcal D([0,T],\mathcal H)$.
\end{theorem}

\begin{proof}
We define
\begin{align*}
   \tilde Z_n^N(i) := & \Delta_n^{-\frac 12}\left((p_N\tilde{\Delta}_i^nY)^{\otimes 2}-\langle\langle p_N\tilde{\Delta}_i^nY\rangle\rangle\right)\\
   = & \Delta_n^{-\frac 12}\left((p_N\tilde{\Delta}_i^nY)^{\otimes 2}-\int_{t_{i-1}}^{t_i} p_N\mathcal S(t_i-s)\Sigma_s\mathcal S(t_i-s)^* p_N ds\right).
\end{align*}
First we show that $\sup_{t\in[0,T]}\Vert\sum_{i=1}^{\ul}\tilde Z_n^N(i)\Vert_{\text{HS}}$  
has finite second moment.
Note that, by the BDG inequality \eqref{BDG inequality}, we have
\begin{align}\label{Fourth Moment Increment inequality}
    \mathbb E\left[\left\|p_N\tilde{\Delta}_i^nY\right\|^4\right]\leq & \mathbb E\left[\left(\int_{(i-1)\Delta_n}^{i\Delta_n} \left\|\sigma_s^{\mathcal S_n}\right\|_{L_{\text{HS}}(U,H)}^2ds\right)^2\right]\\
   \leq & \Delta_n \mathbb E\left[\int_{(i-1)\Delta_n}^{i\Delta_n} \left\|\sigma_s^{\mathcal S_n}\right\|_{L_{\text{HS}}(U,H)}^4ds\right].
\end{align}
Therefore, by the triangle and Cauchy-Schwarz inequalities, we have
\begin{align*}
  \mathbb E&\left[ \left \Vert\sum_{i=1}^{\ul}\tilde Z^N_n(i)\right \Vert_{\mathcal H}^2\right]\\
  & \leq  \mathbb E\left[\left(\sum_{i=1}^{\ul}\left \Vert\tilde Z_n^N(i)\right \Vert_{\mathcal H}\right)^2\right]\\
 & \leq  \mathbb E\left[\sum_{i=1}^{\ul}\left \Vert\tilde Z_n^N(i)\right \Vert_{\mathcal H}^2\right] \ul\\
   &\leq \Delta_n^{-1} \sum_{i=1}^{\ul} \mathbb E\left[\left(\left\|p_N\tilde{\Delta}_i^nY\right\|^2+\int_{(i-1)\Delta_n}^{i\Delta_n}\left\|\Sigma_s^{\mathcal S_n}\right\|_{L_{\text{HS}}(U,H)} ds\right)^2\right] \ul\\
 &  \leq \Delta_n^{-1} \sum_{i=1}^{\ul} 2\left(\mathbb E\left[\left\|p_N\tilde{\Delta}_i^nY\right\|^4\right]+\int_{(i-1)\Delta_n}^{i\Delta_n}\mathbb E\left[\left\|\sigma_s^{\mathcal S_n}\right\|_{L_{\text{HS}}(U,H)}^4\right] ds\right) \ul\\
 &\leq \Delta_n^{-1}\sum_{i=1}^{\ul} 4\Delta_n\int_{(i-1)\Delta_n}^{i\Delta_n}\mathbb E\left[\left\|\sigma_s^{\mathcal S_n}\right\|_{L_{\text{HS}}(U,H)}^4\right] ds \ul\\
 = &\int_0^{\ul}\mathbb E\left[\left\|\sigma_s^{\mathcal S_n}\right\|_{L_{\text{HS}}(U,H)}^4\right] ds \ul<\infty,
\end{align*}
where the finiteness is due to the assumption. 

Now note that $t\mapsto \psi_t=\int_{(i-1)\Delta_n}^tp_N\mathcal S(t_i-s)\sigma_sdW_s$ is a martingale for $t\in[(i-1)\Delta_n,i\Delta_n]$. From \cite[Theorem 8.2, p.~109]{PZ2007} we deduce that 
the process $(\zeta_t)_{t\geq 0}$, with
\begin{align*}
    \zeta_t=  \left(\psi_t\right)^{\otimes 2}-\langle\langle \psi\rangle\rangle_t,
\end{align*}
is a martingale with respect to $(\mathcal{F}_t)_{t\geq 0}$, and hence, 
\begin{align*}
\mathbb E\left[(p_N\tilde{\Delta}_i^nY)^{\otimes 2}|\mathcal F_{(i-1)\Delta_n}\right]=&\mathbb E\left[(\psi_{i\Delta_n})^{\otimes 2}|\mathcal F_{(i-1)\Delta_n}\right]\\
= & \mathbb E\left[\langle\langle \psi\rangle\rangle_{i\Delta_n}|\mathcal F_{(i-1)\Delta_n}\right]\\
= &\int_{(i-1)\Delta_n}^{i\Delta_n}\mathbb E\left[p_N\mathcal S(t_i-s)\Sigma_s\mathcal S(t_i-s)^*p_N|\mathcal F_{t_{i-1}}\right]ds.
\end{align*}
So, $\mathbb E\left[\tilde Z_n^N(i)\right|\mathcal F_{t_{i-1}}]=0$. 
Moreover, for $j<i$, as each $\tilde Z_n^N(i)$ is $\mathcal F_{(i-1)\Delta_n}$ measurable and the conditional expectation commutes with bounded linear operators, we find by using the tower property of conditional expectation that
\begin{align*}
    \mathbb E\left[\langle \tilde Z_n^N(i), \tilde Z_n^N(j)\rangle_{\mathcal H} \right]=&\mathbb E\left[\mathbb E\left[\langle \tilde Z_n^N(i), \tilde Z_n^N(j)\rangle_{\mathcal H} |\mathcal F_{(i-1)\Delta_n}\right]\right]\\
=& \mathbb E\left[\langle \mathbb E\left[\tilde Z_n^N(i)|\mathcal F_{(i-1)\Delta_n}\right], \tilde Z_n^N(j)\rangle_{\mathcal H} \right]=0.
\end{align*}
Thus, we obtain 
 \begin{align*}
  \EE  \left[\left\Vert\sum_{i=1}^{\ul}\tilde Z_n^N(i) \right\Vert_{\mathcal H}^2 \right] \leq \sum_{i=1}^{\lfloor T/\Delta_n\rfloor} \EE\left[\Vert\tilde Z_n^N(i)\Vert_{\mathcal H}^2\right].
\end{align*}
Applying the  triangle and Bochner inequalities, the basic inequality $(a+b)^2\leq 2(a^2+b^2)$ and appealing to \eqref{Fourth Moment Increment inequality}, we find
\begin{align*}
    &\mathbb E\left[\Vert \tilde Z_n^N(i)\Vert_{\mathcal H}^2\right]\\
    &\leq 2\Delta_n^{-1}\EE\left[\Vert(p_N\tilde{\Delta}_i^nY)^{\otimes 2}\Vert_{\mathcal H}^2 + \left(\int_{(i-1)\Delta_n}^{i\Delta_n}\Vert p_N\mathcal S (i\Delta_n-s)\Sigma_s\mathcal S (i\Delta_n-s)^* p_N \Vert_{\mathcal H}ds \right)^2\right]\\
    &\leq 4 \int_{(i-1)\Delta_n}^{i\Delta_n}\mathbb E\left[\Vert p_N\sigma_s^{\mathcal S_n} \Vert_{L_{\text{HS}}(U,H)}^4\right]ds. 
\end{align*}
Summing up, we have
\begin{align*}
 \EE  \left[\sup_{t\in[0,T]}\left\Vert\sum_{i=1}^{\ul}\tilde Z_n^N(i) \right\Vert_{\mathcal H}^2 \right] 
   \leq 4 \sup_{n\in\mathbb N}\int_0^T\mathbb E\left[\Vert p_N\sigma_s^{\mathcal S_n} \Vert_{L_{\text{HS}}(U,H)}^4\right]ds,
\end{align*}
which converges to $0$ by Lemma \ref{L: Projection convergese uniformly on the range of volatility}.
Hence, as
$$\sup_{t\in[0,T]}\sup_{n\in \mathbb N}\sum_{m,k\geq N}\mathbb E\left[ \langle \tilde Z_t^{2,n}, e_k\otimes e_m\rangle_{\mathcal H}^2\right]= \|\sum_{i=1}^{\ul}\tilde Z_n^N(i)\|_{\mathcal H}^2,$$
the Theorem follows by Corollary \ref{C: Tightness Criterion in Hilbert spaces}.
\end{proof}

\subsection{Convergence of finite-dimensional distributions and remainders}

\subsubsection{Proof of the central limit theorem for realised covariation}

Before we can finally prove the central limit theorem for realised covariation, we need the following auxiliary Lemma:
\begin{lemma}\label{L: Adjusted Quadratic Variation is essentially Quadratic Variation after finite projection}
Let $(e_j)_{j\in\mathbb N}$ be an orthonormal basis of $H$ that is contained in $D(\mathcal A^*)$. Then for any $k,l\in\mathbb N$ we have
$$ \int_{(i-1)\Delta_n}^{i\Delta_n} \langle \Sigma_s^{\mathcal S_n} e_k,e_l\rangle ds= \int_{(i-1)\Delta_n}^{i\Delta_n} \langle \Sigma_s e_k,e_l\rangle ds+ \psi_n^{i,k,l},$$
where $\psi^{i,k,l}_n$ is a sequence of random variables such that,
$$
\mathbb E[|\psi^{i,k,l}_n|]\leq K \Delta_n^2,
$$
for a constant $K=K_2(k,l)>0$ independent of $i$.
\end{lemma}
\begin{proof}
Since 
$\mathcal S(i\Delta_n-s)^*e_k=\int_s^{i\Delta_n} \mathcal S(u-s)^*\mathcal A^* e_kdu+e_k$ for all $k\in\mathbb N$, we have that,
\begin{align*}
    \int_{(i-1)\Delta_n}^{i\Delta_n}& \langle \Sigma_s^{\mathcal S_n} e_k,e_l\rangle ds\\
   &= \int_{(i-1)\Delta_n}^{i\Delta_n} \langle \Sigma_s\int_s^{i\Delta_n} \mathcal S(u-s)^*\mathcal A^* e_kdu+\Sigma_se_k,\int_s^{i\Delta_n} \mathcal S(u-s)^*\mathcal A^* e_l du+e_l\rangle ds\\
   &= \int_{(i-1)\Delta_n}^{i\Delta_n} \langle \Sigma_s e_k,e_l\rangle + \langle \Sigma_s\int_s^{i\Delta_n} \mathcal S(u-s)^*\mathcal A^* e_kdu,e_l\rangle\\
   &\qquad\qquad\qquad+ \langle \Sigma_se_k,\int_s^{i\Delta_n} \mathcal S(u-s)^*\mathcal A^* e_l du\rangle\\
   &\qquad\qquad\qquad +  \langle \Sigma_s\int_s^{i\Delta_n} \mathcal S(u-s)^* \mathcal A^* e_kdu,\int_s^{i\Delta_n} \mathcal S(u-s)^*\mathcal A^* e_l du\rangle ds.
\end{align*}
It holds 
by Bochner's inequality
\begin{align*}
    \mathbb E\left[ \left|\langle \Sigma_se_k,\int_s^{i\Delta_n} \mathcal S(u-s)^*\mathcal A^* e_l du\rangle\right| \right]\leq & \int_s^{i\Delta_n} \mathbb E\left[ \| \Sigma_s\mathcal S(u-s)^* \mathcal A^* e_l \| \right]du\\
    \leq & \Delta_n \sup_{t,s\in[0,T]} \mathbb E\left[ \|\Sigma_s \mathcal S(t)^*\mathcal A^* e_l\|\right],
\end{align*}
and analogous by swapping $k$ and $l$
\begin{align*}
    \mathbb E\left[ \left|\langle \Sigma_se_l,\int_s^{i\Delta_n} \mathcal S(u-s)^* \mathcal A^* e_k du\rangle\right| \right]\leq & \int_s^{i\Delta_n} \mathbb E\left[ \| \Sigma_s\mathcal S(u-s)^*\mathcal A^* e_k \| \right]du\\
    \leq & \Delta_n \sup_{t,s\in[0,T]} \mathbb E\left[ \|\Sigma_s \mathcal S(t)^*\mathcal A^* e_k\|\right].
\end{align*}
Moreover 
\begin{align*}
    &\mathbb E\left[ \left|\langle \Sigma_s\int_s^{i\Delta_n} \mathcal S(u-s)^*\mathcal A^* e_k du,\int_s^{i\Delta_n} \mathcal S(u-s)^*\mathcal A^* e_l du\rangle\right| \right]\\
    &\qquad=  \int_s^{i\Delta_n}\int_s^{i\Delta_n}  \mathbb E\left[|\langle \Sigma_s \mathcal S(u-s)^*\mathcal A^* e_k , \mathcal S(v-s)^*\mathcal A^* e_l \rangle |\right]dudv\\
    &\qquad\leq \Delta_n^2 \sup_{t,s\in[0,T],j=k,l} \mathbb E\left[ \|\sigma_s \mathcal S(t)^*\mathcal A^* e_j\|^2\right].
\end{align*}
By the choice $K=\sup_{t,s\in[0,T],j=k,l} \mathbb E\left[ \|\sigma_s \mathcal S(t)^*\mathcal A^* e_j\|^2\right]$, the assertion follows. 
\end{proof}

We can now prove an auxiliary central limit theorem, which essentially does not rely on the spatial regularity condition in Assumption \ref{As: Spatial regularity}.
\begin{theorem}\label{T: Central limit theorem for covariation minus its quadratic variations}
Let Assumption \ref{As: Weakened Assumption for the CLT for the SARCV} hold. We have that $ \tilde Z^{n,2} \stackrel{\mathcal L-s}{\Rightarrow} (\mathcal N(0,\Gamma_t))_{t\in[0,T]}$.
\end{theorem}
\begin{proof} 
Let $(e_i)_{i\in\mathbb N}$ be an orthonormal basis of $H$ that is contained in the domain $D(\mathcal A^*)$ of the adjoint of the generator $\mathcal A$. 
Since tightness of $(\mathbb P_{\tilde{Z}^{n,2}})_{n\in\mathbb N}$ in the Skorokhod topology is guaranteed by Theorems \ref{T: Verfication of Aldous' condition} and \ref{T: Tightnes for realised covariation} in combination with Theorem \ref{T: Aldous Tightness Theorem},
it is enough to show the stable convergence in law as a process of the finite-dimensional distributions $\tilde{Z}_t^{n,2}(d):=(\langle \tilde{Z}_t^{n,2},e_k\otimes e_l\rangle)_{k,l=1,\ldots,d}$.

Therefore, for $k,l=1,\ldots,d$ 
we find by Lemma \ref{L: Reduction to martingales is possible} and Lemma \ref{L: Adjusted Quadratic Variation is essentially Quadratic Variation after finite projection} (and using the same notation) that
\begin{align*}
   \langle\tilde{Z}_t^{n,2} , e_k\otimes e_l\rangle_{\mathcal H} = & \Delta_n^{-\frac 12} \sum_{i=1}^{\ul} \left(\langle\tilde{\Delta}_i^n Y, e_k\rangle\langle\tilde{\Delta}_i^n Y, e_l\rangle- \int_{(i-1)\Delta_n}^{i\Delta_n}\langle \Sigma_s^{\mathcal S_n} e_k,e_l\rangle ds\right)\\
    = &  \Delta_n^{-\frac 12} \sum_{i=1}^{\ul}  \left(\langle \Delta_i^n S, e_k\rangle\langle \Delta_i^n S, e_l\rangle- \int_{(i-1)\Delta_n}^{i\Delta_n}\langle \Sigma_s e_k,e_l\rangle ds\right)\\
    &\qquad\qquad+\Delta_n^{-\frac 12} \sum_{i=1}^{\ul} \left(\zeta_n^{i,k,l}+\psi_n^{i,k,l}\right).
\end{align*}
The second summand converges to $0$ in probability uniformly on compacts, which is why we have that the stable limit of $ \langle\tilde{Z}_t^{n,2}e_k,e_l\rangle$ 
is the same as the one of 
$$\Delta_n^{-\frac 12} \sum_{i=1}^{\ul}  \left(\langle \Delta_i^n S, e_k\rangle\langle \Delta_i^n S, e_l\rangle- \int_{(i-1)\Delta_n}^{i\Delta_n}\langle \Sigma_s e_k,e_l\rangle ds\right).$$
The latter is a component of the difference of realised quadratic covariation and the quadratic covariation of the $d$-dimensional continuous local martingale $S_t^d=(\langle S_t,e_1\rangle,...,\langle S_t, e_d\rangle)$. 
Therefore, 
$\tilde{Z}^{n,2}(d)=(\langle \tilde Z^{n},e_k,e_l\rangle)_{k,l=1,...,d}$
converges by Theorem 5.4.2 from \cite{JacodProtter2012} stably as a process to a process that is defined on a very good filtered extension $(\tilde{\Omega},\tilde{\mathcal F},\tilde{\mathcal F}_t,\tilde{\mathbb P})$ of $(\Omega,\mathcal F,\mathcal F_t,\mathbb P)$. This limiting process is conditionally on $\mathcal F$ a centered Gaussian which can be realised on the very good filtered extension as
\begin{align*}
  N_{k,l} = \frac 1{\sqrt 2}\sum_{c,b=1}^d\int_0^t \left(\hat{\sigma}_{kl,bc}(s)+\hat{\sigma}_{lk,bc}(s)\right)dB^{cb}_s.
\end{align*}
Here, $\hat{\sigma}(s)$ is a $d^2\times d^2$-matrix, being the  square-root of the matrix $\hat c(s)$ with entries $\hat c_{kl,k'l'}(s)= \langle \Sigma_s e_k,e_{k'}\rangle \langle \Sigma_s e_l,e_{l'}\rangle $. Furthermore, $B$ is a matrix of independent Brownian motions. 
This corresponds to the covariance $\Gamma_t$, as by the It{\^o} isometry we obtain
\begin{align*}
    \mathbb E[ N_{k,l} N_{k',l'}|\mathcal F]
    &=\frac 12\sum_{c,b=1}^d \int_0^t \left(\hat{\sigma}_{kl,bc}(s)+\hat{\sigma}_{lk,bc}(s)\right)\left(\hat{\sigma}_{k'l',bc}(s)+\hat{\sigma}_{l'k',bc}(s)\right)ds \\
     &=\frac 12\sum_{c,b=1}^d \int_0^t \left(\hat{\sigma}_{kl,bc}(s)\hat{\sigma}_{k'l',bc}(s)+\hat{\sigma}_{kl,bc}(s)\hat{\sigma}_{l'k',bc}(s) \right. \\
     &\qquad+\left.\hat{\sigma}_{lk,bc}(s)\hat{\sigma}_{k'l',bc}(s)+\hat{\sigma}_{lk,bc}(s)\hat{\sigma}_{l'k',bc}(s)\right)ds \\
       &=\frac 12 \int_0^t \hat c_{kl,k'l'}(s)+\hat c_{kl,l'k'}(s)+\hat c_{lk,k'l'}(s)+\hat c_{lk,l'k'}(s)ds \\
       &=\frac 12 \int_0^t \langle \Sigma_s e_k,e_{k'}\rangle\langle \Sigma_s e_l,e_{l'}\rangle+\langle \Sigma_s e_k,e_{l'}\rangle\langle \Sigma_s e_l,e_{k'}\rangle \\
       &\qquad+\langle \Sigma_s e_l,e_{k'}\rangle\langle \Sigma_s e_k,e_{l'}\rangle+\langle \Sigma_s e_l,e_{l'}\rangle\langle \Sigma_s e_k,e_{k'}\rangle ds \\
        &=\int_0^t \langle \Sigma_s e_k,e_{k'}\rangle\langle \Sigma_s e_l,e_{l'}\rangle+\langle \Sigma_s e_k,e_{l'}\rangle\langle \Sigma_s e_l,e_{k'}\rangle ds \\
     &=\int_0^t \left(\langle \Sigma_s (e_{k'}\otimes e_{l'})\Sigma_s,e_k\otimes e_l\rangle_{\mathcal H} +\langle \Sigma_s (e_{l'}\otimes e_{k'})\Sigma_s,e_k\otimes e_l\rangle_{\mathcal H}\right)ds\\
     &=\langle \Gamma_t e_{k'}\otimes e_{l'},e_k\otimes e_l\rangle_{\mathcal H}.
\end{align*}
As now all finite-dimensional distributions converge stably and the sequence of measures is tight, we obtain by Theorem \ref{T: stable convergence of finite-dimensional distributions yields stable convergence} that the convergence is indeed stable to a process $Z$ in the Skorokhod space, whose corresponding finite-dimensional components $Z(d):=(\langle Z,e_k,e_l\rangle)_{k,l=1,...,d}$ 
are conditionally on $\mathcal F$ a centred Gaussian process.
It is itself conditionally centred Gaussian. Moreover, the process is continuous as well, since for $Z$ in $\mathcal D([0,T],\mathcal H)$ we have
    \begin{align*}
       \| Z_t-Z_s\|_{\mathcal H}\leq \|P_d Z_t\|_{\mathcal H}+\|Z_t(d)-Z_s(d)\|_{\mathbb R^{d\times d}}+\|P_d Z_s\|_{\mathcal H}.
    \end{align*}
  The outer terms can be made arbitrarily small since by \eqref{slightly better than tightness for SARCV-centered by conditional expectation} it holds
    $$ \lim_{d\to\infty}\sup_{t\in[0,T]}\mathbb E\left[\|P_d\tilde Z_t^{2,n}\|_{\mathcal H}^2\right]=0.$$  
    The middle term converges for each fixed $d$ to $0$ as $|t-s|\to 0$, as $P_d Z$ is continuous as an It{\^o} integral. 
The proof is complete.
\end{proof}
 Now we are in the position to prove the central limit theorems \ref{T:Infeasible Central Limit Theorem} and  \ref{T: Central limit theorem for functionals of the quadratic covariation} for realised covariations.
 \begin{proof}[Proof of Theorem \ref{T:Infeasible Central Limit Theorem}]
 It is clear that $\Delta_n^{-\frac 12}\int_{\ul}^t\Sigma_s ds$ converges to 0 in the $u.c.p.$-sense, and
 since 
 $$
 \tilde X_t^n = \tilde Z_t^{n,2}+\Delta_n^{-\frac 12}\sum_{i=1}^{\ul}\int_{(i-1)\Delta_n}^{i\Delta_n} \left(\Sigma_s^{\mathcal S_n}-\Sigma_s \right)ds+\Delta_n^{-\frac 12}\int_{\ul}^t\Sigma_s ds,
 $$ 
 we have to show 
 \begin{equation*}
     \Delta_n^{-\frac 12}\sum_{i=1}^{\ul}\int_{(i-1)\Delta_n}^{i\Delta_n} \left(\Sigma_s^{\mathcal S_n}-\Sigma_s\right) ds\stackrel{u.c.p.}{\longrightarrow} 0.
 \end{equation*}
 Recall that $P_N$ denotes the projection onto $\overline{\{ e_k\otimes e_l: k,l\geq N\}}$, where again $(e_j)_{j\in\mathbb N}$ is an orthonormal basis of $H$ that is contained in $D(\mathcal A^*)$. Notice that for each $A\in \mathcal H$
\begin{align*}
    \|(I-P_N)A\|_{\mathcal H}&=\|\sum_{k,l\leq N-1} \langle A, e_k\otimes e_l\rangle_{\mathcal H} e_k\otimes e_l\|_{\mathcal H} \\
    &\leq \sum_{k,l\leq N-1} |\langle A, e_k\otimes e_l\rangle_{\mathcal H}| \|e_k\otimes e_l\|_{\mathcal H}\\
    &= \sum_{k,l\leq N-1} |\langle A, e_k\otimes e_l\rangle_{\mathcal H}|\\
    &= \sum_{k,l\leq N-1} |\langle A e_k,e_l\rangle|.
\end{align*}
Then, by Lemma \ref{L: Adjusted Quadratic Variation is essentially Quadratic Variation after finite projection}
 \begin{align}\label{finite dimensional projections of adjusted quadratic variations converge}
     & \mathbb E\left[\sup_{t\in [0,T]}\left\|(I-P_N)  \Delta_n^{-\frac 12}\sum_{i=1}^{\ul}\int_{(i-1)\Delta_n}^{i\Delta_n} \left(\Sigma_s^{\mathcal S_n}-\Sigma_s\right) ds\right\|_{\mathcal H}\right]\notag\\
     &\quad\leq  \mathbb E\left[\sup_{t\in [0,T]}\sum_{k,l=1}^{N-1}\left|\Delta_n^{-\frac 12}\sum_{i=1}^{\ul}\int_{(i-1)\Delta_n}^{i\Delta_n} \langle \left(\Sigma_s^{\mathcal S_n}-\Sigma_s\right)e_k,e_l\rangle ds\right|\right]\notag\\
     &\quad=   \mathbb E\left[\sup_{t\in [0,T]}\sum_{k,l=1}^{N-1}\left|\Delta_n^{-\frac 12}\sum_{i=1}^{\ul}\psi_{i,k,l}\right|\right]\notag\\
     &\quad\leq  \sum_{k,l=1}^{N-1}\Delta_n^{-\frac 12}\sum_{i=1}^{\ulT} K(k,l)\Delta_n^2\notag\\
     &\quad\leq  \left(\sum_{k,l=1}^{N-1}K(k,l)\right)T\Delta_n^{\frac 12},
 \end{align}
which converges to $0$ as $n\to \infty$. 
Thus, for all $N\in\mathbb N$,
 $$(I-P_N)  \Delta_n^{-\frac 12}\sum_{i=1}^{\ul}\int_{(i-1)\Delta_n}^{i\Delta_n} \left(\Sigma_s^{\mathcal S_n}-\Sigma_s\right) ds\stackrel{u.c.p.}{\longrightarrow} 0.$$
 Further, we have, by using   the triangle and Bochner inequalities,  
 \begin{align*}
   &  \mathbb E\left[\sup_{t\in [0,T]}\left\|P_N  \Delta_n^{-\frac 12}\sum_{i=1}^{\ul}\int_{(i-1)\Delta_n}^{i\Delta_n} \left(\Sigma_s^{\mathcal S_n}-\Sigma_s\right) ds\right\|_{\mathcal H}\right]\\
     \leq &  \Delta_n^{-\frac 12}\int_0^T\mathbb E\left[\left\|P_N \left(\Sigma_s^{\mathcal S_n}-\Sigma_s \right)\right\|_{\mathcal H}\right]ds\\
       \leq &  \Delta_n^{-\frac 12}\int_0^T\mathbb E\left[\left\|P_N (\mathcal S(\lfloor s/\Delta_n\rfloor \Delta_n-s)-I)\Sigma_s \mathcal S(\lfloor s/\Delta_n\rfloor \Delta_n-s)^*\right\|_{\mathcal H}\right.\\
       &\qquad\qquad\qquad\left.+\left\|P_N \Sigma_s (\mathcal S(\lfloor s/\Delta_n\rfloor \Delta_n-s)^*-I)\right\|_{\mathcal H}\right]ds.
        \end{align*}
Now, estimating further using the Cauchy-Schwarz inequality along with the fact that $\|AB\|_{\mathcal H}\leq \|A\|_{\text{op}}\|B\|_{L_{\text{HS}}(U,H)}$ for any Hilbert Schmidt operator $B:H\to U$ and continuous linear operator $A:U\to H$,
        \begin{align*}
        &  \mathbb E\left[\sup_{t\in [0,T]}\left\|P_N  \Delta_n^{-\frac 12}\sum_{i=1}^{\ul}\int_{(i-1)\Delta_n}^{i\Delta_n} \left(\Sigma_s^{\mathcal S_n}-\Sigma_s\right) ds\right\|_{\mathcal H}\right]\\
       \leq &  \Delta_n^{-\frac 12}\int_0^T\mathbb E\left[\|p_N (\mathcal S(\lfloor s/\Delta_n\rfloor\Delta_n-s)-I)\sigma_s \|_{\text{op}}\right.\\
       &\qquad\qquad\left.\left(\|p_N\sigma_s \|_{L_{\text{HS}}(U,H)}+\|p_N\mathcal S(\lfloor s/\Delta_n\rfloor\Delta_n-s)\sigma_s \|_{L_{\text{HS}}(U,H)}\right)\right]ds\\ 
       \leq &  \left(\int_0^T\mathbb E\left[\|\Delta_n^{-\frac 12}  (\mathcal S(\lfloor s/\Delta_n\rfloor\Delta_n-s)-I)\sigma_s \|_{\text{op}}^2 \right]ds\right)^{\frac 12} \\
       &\qquad\qquad\times\left(\int_0^T\sqrt 2\mathbb E\left[\|p_N\sigma_s \|_{L_{\text{HS}}(U,H)}^2+\|p_N\mathcal S(\lfloor s/\Delta_n\rfloor\Delta_n-s)\sigma_s \|_{L_{\text{HS}}(U,H)}^2 \right]ds\right)^{\frac 12}.
 \end{align*}
 The first factor is finite by Assumption \ref{As: Spatial regularity}, whereas the second one converges to 0 as $N\to\infty$ by Lemma \ref{L: Projection convergese uniformly on the range of volatility}. Observe that we used that the Lemma holds also in the special case $\mathcal S(t)=I$ on $H$ for all $t\geq 0$. Thus, we have
 $$\lim_{N\to\infty}\sup_{n\in\mathbb N}\mathbb E\left[\sup_{t\in [0,T]}\left\|P_N  \Delta_n^{-\frac 12}\sum_{i=1}^{\ul}\int_{(i-1)\Delta_n}^{i\Delta_n} \left(\Sigma_s^{\mathcal S_n}-\Sigma_s\right) ds\right\|_{\mathcal H}\right]=0.$$
 Therefore, we can find for each $\delta>0$ an $N_{\delta}\in\mathbb N$ such that  for all $N\geq N_{\delta}$ 
 \begin{align*}
  & \mathbb E\left[\sup_{t\in [0,T]}\left\|  \Delta_n^{-\frac 12}\sum_{i=1}^{\ul}\int_{(i-1)\Delta_n}^{i\Delta_n} \left(\Sigma_s^{\mathcal S_n}-\Sigma_s\right) ds\right\|_{\mathcal H}\right]\\
  \leq & \mathbb E\left[\sup_{t\in [0,T]}\left\|(I-P_N)  \Delta_n^{-\frac 12}\sum_{i=1}^{\ul}\int_{(i-1)\Delta_n}^{i\Delta_n} \left(\Sigma_s^{\mathcal S_n}-\Sigma_s\right) ds\right\|_{\mathcal H}\right]\\
  & \qquad\qquad+ \sup_{n\in\mathbb N}\mathbb E\left[\sup_{t\in [0,T]}\left\|P_N  \Delta_n^{-\frac 12}\sum_{i=1}^{\ul}\int_{(i-1)\Delta_n}^{i\Delta_n} \left(\Sigma_s^{\mathcal S_n}-\Sigma_s\right)ds\right\|_{\mathcal H}\right]\\
  \leq & \mathbb E\left[\sup_{t\in [0,T]}\left\|(I-P_N)  \Delta_n^{-\frac 12}\sum_{i=1}^{\ul}\int_{(i-1)\Delta_n}^{i\Delta_n} \left(\Sigma_s^{\mathcal S_n}-\Sigma_s\right) ds\right\|_{\mathcal H}\right]+\delta\\
  &\to \delta.
 \end{align*}
 As this holds for all $\delta >0$, we obtain that
 $$\mathbb E\left[\sup_{t\in [0,T]}\left\|  \Delta_n^{-\frac 12}\sum_{i=1}^{\ul}\int_{(i-1)\Delta_n}^{i\Delta_n} \left(\Sigma_s^{\mathcal S_n}-\Sigma_s\right) ds\right\|_{\mathcal H}\right]\to 0,$$
 as $n\to\infty$, and the assertion follows.
 \end{proof}

\begin{proof}[Proof of Theorem \ref{T: Central limit theorem for functionals of the quadratic covariation}]
For any $B\in\mathcal H$, $$\langle \tilde X_t^n ,B\rangle_{\mathcal H} = \langle \tilde Z_t^{n,2},B\rangle_{\mathcal H}+\langle \Delta_n^{-\frac 12}\sum_{i=1}^{\ul}\int_{(i-1)\Delta_n}^{i\Delta_n} \left(\Sigma_s^{\mathcal S_n}-\Sigma_s\right) ds , B\rangle_{\mathcal H}.$$ 
Since the stable convergence with respect to the Hilbert-Schmidt norm as proven in Theorem \ref{T: Central limit theorem for covariation minus its quadratic variations} implies the stable convergence in law with respect to the (analytically) 
weak topology,
we only have to show 
 \begin{equation*}
     \Delta_n^{-\frac 12}\sum_{i=1}^{\ul}\int_{(i-1)\Delta_n}^{i\Delta_n} \langle\left(\Sigma_s^{\mathcal S_n}-\Sigma_s\right),B\rangle_{\mathcal H} ds\stackrel{u.c.p.}{\longrightarrow} 0.
 \end{equation*}
 We can argue componentwise, which is why we assume without loss of generality that $B=h\otimes g$ for $h,g\in F^{\mathcal S^*}_{1/2}$. 
We split into two terms as follows:
\begin{align*}
    & \Delta_n^{-\frac 12}\sum_{i=1}^{\ul}\int_{(i-1)\Delta_n}^{i\Delta_n} \langle (\Sigma_s^{\mathcal S_n}-\Sigma_s) h,g\rangle ds\\
     &\qquad=\Delta_n^{-\frac 12}\sum_{i=1}^{\ul}\int_{(i-1)\Delta_n}^{i\Delta_n} \langle ((\mathcal S(i\Delta_n-s)-I)\Sigma_s\mathcal S(i\Delta_n-s)^*) h,g\rangle ds\\
     &\qquad\qquad +\Delta_n^{-\frac 12}\sum_{i=1}^{\ul}\int_{(i-1)\Delta_n}^{i\Delta_n}\langle (\Sigma_s (\mathcal S(i\Delta_n-s)-I)^*) h,g\rangle ds\\
     &\qquad=(1)_n+(2)_n.
\end{align*}
We only show the convergence for $(1)_n$ since the argument for $(2)_n$ is analogous. It holds
\begin{align*}
   (1)_n = & \Delta_n^{-\frac 12}\sum_{i=1}^{\ul}\int_{(i-1)\Delta_n}^{i\Delta_n} \langle (I-p_N)(\Sigma_s\mathcal S(i\Delta_n-s)^*) h,(\mathcal S(i\Delta_n-s)-I)^*g\rangle ds\\
   & +\Delta_n^{-\frac 12}\sum_{i=1}^{\ul}\int_{(i-1)\Delta_n}^{i\Delta_n} \langle p_N(\Sigma_s\mathcal S(i\Delta_n-s)^*) h,(\mathcal S(i\Delta_n-s)-I)^*g\rangle ds\\
   = & (1.1)_{n,N}+(1.2)_{n,N},
\end{align*}
where again we denoted by $p_N$ the projection onto $\overline{\{e_j: j\geq N\}}$ for an orthonormal basis $(e_j)_{j\in\mathbb N}$ of $H$
 that is contained in $D(\mathcal A)$. We have 
 \begin{align*}
    \mathcal S(t-s) e_i- e_i= \int_s^t \mathcal S(u-s)\mathcal A e_i du,
\end{align*}
and therefore it holds for the first summand that,
\begin{align}\label{Auxiliary result: functionals of finite dimensional projections of the remainder of the SARCV discretisation converge very fast to 0}
& \sup_{t\in[0,T]}|(1.1)_{n,N}|\notag\\
   &= \sup_{t\in[0,T]}|\Delta_n^{-\frac 12}\sum_{i=1}^{\ul}\int_{(i-1)\Delta_n}^{i\Delta_n} \langle (I-p_N)(\Sigma_s\mathcal S(i\Delta_n-s)^*) h,(\mathcal S(i\Delta_n-s)-I)^*g\rangle ds|\notag\\
     &= \sup_{t\in[0,T]}|\Delta_n^{-\frac 12}\sum_{i=1}^{\ul}\int_{(i-1)\Delta_n}^{i\Delta_n} \sum_{j=1}^{N-1}\langle (\Sigma_s\mathcal S(i\Delta_n-s)^*) h, e_j\rangle \langle e_j, (\mathcal S(i\Delta_n-s)-I)^*g\rangle ds|\notag\\
      &=  \sup_{t\in[0,T]}|\sum_{j=1}^{N-1}\Delta_n^{-\frac 12}\sum_{i=1}^{\ul}\int_{(i-1)\Delta_n}^{i\Delta_n} \langle (\Sigma_s\mathcal S(i\Delta_n-s)^*) h, e_j\rangle \langle \int_s^{i\Delta_n} \mathcal S(u-s)\mathcal A e_j ds, g\rangle ds \notag\\
      \leq & \sum_{j=1}^{N-1}\Delta_n^{-\frac 12}\int_0^T \|\Sigma_s\|_{\text{op}}\|h\| \Delta_n\|\mathcal A e_j\| \|g\| ds \sup_{t\in[0,T]}\|\mathcal S(t)\|_{\text{op}}^2\notag\\
      \leq & \Delta_n^{\frac 12}\sum_{j=1}^{N-1} \int_0^T\|\Sigma_s\|_{\text{op}}ds \|h\|\|g\|\sup_{t\in[0,T]}\|\mathcal S(t)\|_{\text{op}}^2.
\end{align}
The last expression converges to $0$ as $n\to\infty$ almost surely. 
In particular, we have as $n\to \infty$ that 
$$ |(1.1)_{n,N}|\stackrel{u.c.p.}{\to}0.$$
It follows from the fact that $g\in  F_{ 1/2}^{\mathcal S^*}$
that we can find a constant $K:=\sup_{t\leq T}\|t^{-\frac 12}(\mathcal S(t)-I)^* g\|<\infty$ 
such that
\begin{align}\label{Auxiliary result: functionals of tails of the remainder of the SARCV discretisation converge very fast to 0}
& \mathbb E\left[    \sup_{t\in[0,T]}|(1.2)_{n,N}|\right]\notag\\
\leq & \Delta_n^{-\frac 12}\sum_{i=1}^{\ulT}\int_{(i-1)\Delta_n}^{i\Delta_n} \mathbb E\left[|\langle p_N(\Sigma_s\mathcal S(i\Delta_n-s)^*) h,(\mathcal S(i\Delta_n-s)-I)^*g\rangle|\right] ds\notag\\
    \leq & \Delta_n^{-\frac 12}\sup_{t\leq \Delta_n}\|(\mathcal S(t)-I)^*g\|\int_0^T \sup_{t\leq \Delta_n}\mathbb E\left[\| p_N\sigma_s\|_{\text{op}}\|\sigma_s^*\mathcal S(t)^* \|_{\text{op}}\|h\|\right]ds\notag\\
    \leq & K \left(\int_0^T \mathbb E\left[\| p_N\sigma_s\|_{\text{op}}^2\right]ds\right)^{\frac 12}\left(\int_0^T\sup_{t\leq \Delta_n}\mathbb E\left[\|\sigma_s^*\mathcal S(t)^*) \|_{\text{op}}^2\right]\|h\|^2ds\right)^{\frac 12}\notag\\
    \leq & K \left(\int_0^T \mathbb E\left[\| p_N\sigma_s\|_{\text{op}}^2\right]ds\right)^{\frac 12}\left(\int_0^T\mathbb E\left[\|\sigma_s^* \|_{\text{op}}^2\right]\|h\|^2ds\right)^{\frac 12} \sup_{t\in[0,T]}\|\mathcal S(t)\|_{\text{op}}.
\end{align}
As the first factor converges to $0$ as $N\to\infty$ by Lemma \ref{L: Projection convergese uniformly on the range of volatility}, we obtain the convergence of $ \mathbb E [\sup_{t\in[0,T]}|(1.2)_{n,N}|]$ to $0$ as $N\to\infty$ uniformly in $n$. 
Therefore, we can find for each $\delta>0$ an $N\in\mathbb N$, such that by Markov's inequality
\begin{align*}
  & \lim_{n\to\infty} \mathbb P\left[\sup_{t\in[0,T]}|(1)_n|>\epsilon\right]\\
    \leq & \lim_{n\to\infty} \mathbb P\left[\sup_{t\in[0,T]}|(1.1)_{n,N}|>\epsilon\right]+\sup_{n\in \mathbb N} \mathbb P\left[\sup_{t\in[0,T]}|(1.2)_{n,N}|>\epsilon\right]\\
    = & 0+ \frac 1{\epsilon}\sup_{n\in \mathbb N} \mathbb E\left[\sup_{t\in[0,T]}|(1.2)_{n,N}|\right]\leq \delta.
\end{align*}
As this holds for all $\delta>0$ we obtain
$(1)_n\stackrel{u.c.p.}{\longrightarrow} 0$ as $n\to\infty$.
The assertion for $(2)_n\stackrel{u.c.p.}{\longrightarrow} 0$ follows analogously.
\end{proof}

\section{Remaining proofs}\label{sec: Remaining proofs}

We will now prove the remaining results, i.e.~Theorem \ref{T: Limit Theorems for nonadjusted RV}, Theorem \ref{T:Finite dimensional limit theorems for functionals of nonadjusted covariation}, Lemma \ref{L: Consistency of discretised truncated SARCV} and Lemma \ref{L: Evaluation functionals in Sobolev spaces have 1/2 regularity} as well as Examples \ref{Ex: counterexample for the Spatial regularity assumption} and \ref{Counterexample for RV}.
We start with the proof of Example \ref{Ex: counterexample for the Spatial regularity assumption}.

\begin{proof}[Proof of Example \ref{Ex: counterexample for the Spatial regularity assumption}]
We start with some general observations.
If the central limit theorem should be valid, i.e., if $\Delta_n^{-\frac 12}\left(SARCV_t^n-\int_0^t \Sigma_s ds\right)$ should converge in distribution, we must necessarily have that it is tight. As the sum of two tight sequences is tight itself and since we have
\begin{align*}
    &\Delta_n^{-\frac 12}\left(SARCV_t^n-\int_0^t \Sigma_s ds\right)-\Delta_n^{-\frac 12}\left(SARCV_t^n-\int_0^t \Sigma_s^{\mathcal S_n} ds\right)\\
    = &\left(\Delta_n^{-\frac 12}\sum_{i=1}^{\ul}\int_0^t \Sigma_s^{\mathcal S_n}-\Sigma_s ds\right),
\end{align*}
we find that $\left(\Delta_n^{-\frac 12}\sum_{i=1}^{\ul}\int_0^t \Sigma_s^{\mathcal S_n}-\Sigma_s ds\right)$ must be tight, due to Theorem 
\ref{T: Central limit theorem for covariation minus its quadratic variations}. I.e.
for all $\epsilon>0$ there is a compact set $K_{\epsilon}\subset \mathcal H$ such that $\sup_{n\in\mathbb N}\mathbb P\left[X_n\notin K_{\epsilon}\right]<\epsilon$. All compact sets $K\subset \mathcal H$ are bounded and hence contained in a ball with radius large enough. Hence, from tightness we obtain that for all $\epsilon>0$ there is an $M_{\epsilon}>0$
such that 
$$\sup_{n\in\mathbb N}\mathbb P\left[\left\|\left(\Delta_n^{-\frac 12}\sum_{i=1}^{\ul}\int_0^t \Sigma_s^{\mathcal S_n}-\Sigma_s ds\right)\right\|_{\mathcal H}>M_{\epsilon}\right]<\epsilon.$$
Thus, if for some specification of  $\sigma$, there is an $\epsilon_0>0$ such that
\begin{equation}\label{Criterion for invalidity of the CLT}
  \limsup_{M\uparrow\infty} \sup_{n\in\mathbb N}\mathbb P\left[\left\|\left(\Delta_n^{-\frac 12}\sum_{i=1}^{\ul}\int_0^t \Sigma_s^{\mathcal S_n}-\Sigma_s ds\right)\right\|_{\mathcal H}>M\right]\geq\epsilon_0
\end{equation}
we necessarily have that the central limit theorem cannot hold.
For some $f_1\in H$ let the volatility have the form
$$\sigma_s = e\otimes f_1,$$
where $e\in H$ is such that $\|e\|=1$. 
Moreover, we let $f_2 \in H$ be such that 
\begin{equation}\label{invalid central limit theorem: technical  Condition on the example III}
    |\langle (\mathcal S(x)-I)f_1,f_2\rangle| \leq C x^{\frac 14}, 
\end{equation}
for some constant $C>0$, which does not depend on $x\geq 0$ and
\begin{equation}\label{invalid central limit theorem: technical  Condition on the example IV}
  \limsup_{n\to\infty}\left|\Delta_n^{-\frac 12}\langle f_1,f_2\rangle \sum_{i=1}^n \int_{(i-1)\Delta_n}^{i\Delta_n} \langle (\mathcal S(i\Delta_n-s)-I) f_1,f_2\rangle ds\right|=\infty.  
\end{equation}
Moreover, we have
$\Sigma_s= f_1^{\otimes 2}$ and hence
\begin{align*}
& \Delta_n^{-\frac 12}\langle\sum_{i=1}^{\ul}\int_0^t \Sigma_s^{\mathcal S_n}-\Sigma_s ds, f_2^{\otimes 2}\rangle\\
= & \Delta_n^{-\frac 12}\sum_{i=1}^{\ul}\int_{(i-1)\Delta_n)}^{i\Delta_n}\langle \mathcal S(i\Delta_n-s)f_1,f_2\rangle^2 -\langle f_1,f_2\rangle^2ds\\
= & \Delta_n^{-\frac 12}\sum_{i=1}^{\ul}\int_{(i-1)\Delta_n)}^{i\Delta_n}\langle (\mathcal S(i\Delta_n-s)-I)f_1,f_2\rangle\langle (\mathcal S(i\Delta_n-s)+I)f_1,f_2\rangle ds\\
= & \Delta_n^{-\frac 12}\sum_{i=1}^{\ul}\int_{(i-1)\Delta_n)}^{i\Delta_n}\langle (\mathcal S(i\Delta_n-s)-I)f_1,f_2\rangle^2 ds\\
& +2\langle f_1,f_2\rangle\Delta_n^{-\frac 12}\sum_{i=1}^{\ul}\int_{(i-1)\Delta_n)}^{i\Delta_n}\langle (\mathcal S(i\Delta_n-s)-I)f_1,f_2\rangle ds.
\end{align*}
Due to \eqref{invalid central limit theorem: technical  Condition on the example III}, it is simple to see that the first term converges to $0$ as $n\to \infty$. 
Now we have that \eqref{Criterion for invalidity of the CLT} holds, since
\begin{align}
&\|f_2\|^2\limsup_{n\in\mathbb N}\left\|\left(\Delta_n^{-\frac 12}\sum_{i=1}^{\ul}\int_0^t \Sigma_s^{\mathcal S_n}-\Sigma_s ds\right)\right\|_{\mathcal H}\notag\\
\geq & \limsup_{n\to\infty}\Delta_n^{-\frac 12}\left|\langle\sum_{i=1}^{\ul}\int_{(i-1)\Delta_n)}^{i\Delta_n} (\mathcal S(i\Delta_n-s)-I)f_1 ds, f_2\rangle \right||\langle f_1, f_2\rangle|= \infty.\notag
\end{align}

In order to show that Example \ref{Ex: counterexample for the Spatial regularity assumption} is indeed a valid counterexample, it is, thus, left to show that for the choice $H=L^2[0,2]$,  $(\mathcal S(t))_{t\geq 0}$ the nilpotent semigroup of left-shifts and $f_1$ a path of a fractional Brownian motion, we can find an $f_2\in H$ such that 
\eqref{invalid central limit theorem: technical  Condition on the example III} and \eqref{invalid central limit theorem: technical  Condition on the example IV} hold. We do this as follows:
We define $(B_1(t),B_2(t)
)_{t\in \mathbb R}$ to be a multivariate fractional Brownian motion on some probability space $(\bar{\Omega},\bar{\mathcal F},\bar{\mathbb P})$, i.e.~a bivariate Gaussian stochastic process with stationary increments 
such that the multivariate self-similarity
$$(B_1(\lambda t),B_2(\lambda t))\sim(\lambda^{\mathfrak H}B_1( t),\lambda^{\epsilon}B_2( t)) \quad \forall \lambda >0,t\in\mathbb R$$
holds for $0<\mathfrak H<\frac 12$, $0<\epsilon< \frac 12 -\mathfrak H$ and $\max(\mathfrak H,\epsilon)>\frac 14 $. Moreover, we assume that $\mathbb E\left[B_1( t)B_2(t)\right]=:\rho>0$ for $t\in\mathbb R$. We also assume that $(B_1(t),B_2(t))_{t\in \mathbb R}$ is time-reversible, i.e. $(B_1(t),B_2(t))=B_1(-t),B_2(-t))$ for all $t\in\mathbb R$. In that case, the covariance structure of this process is given by
\begin{align*}
    \mathbb E\left[B_i(t)B_j(s)\right]=\frac{\rho_{i,j}}2 (|s|^{\mathfrak H_{i,j}}+|t|^{\mathfrak H_{i,j}}-|t-s|^{\mathfrak H_{i,j}}),
\end{align*}
where $$\rho_{i,j}=\begin{cases} 1,& i=j,\\
\rho, & i\neq j,
\end{cases}
$$ and $$\mathfrak H_{i,j}=\begin{cases} \mathfrak H,& i=j=1,\\
\epsilon, & i=j=2,\\
\mathfrak H+\epsilon, & i\neq j.
\end{cases}
$$
Observe that we can always find a $\rho>0$ that guarantees the existence of such a process (c.f.~Proposition 9 in \cite{Lavancier2013})
We want to find an $\omega\in \bar{\Omega}$, such that with the choice 
\begin{equation}\label{abstract omega specification for counterexample of CLT}
    (f_1(x),f_2(x))=(B_1(x)(\omega),B_2(x)(\omega)),\quad x\in [0,2],
\end{equation}
 we have \eqref{invalid central limit theorem: technical  Condition on the example III} and \eqref{invalid central limit theorem: technical  Condition on the example IV}. For that, observe that we have for $t<2$, as the fractional Brownian motion is globally Hölder on the compact interval $[0,2]$ that for almost all $\omega \in \bar{\Omega}$ there is a $\bar C_{\omega}>0$ such that
\begin{align*}
 \|(\mathcal S(t)-I)B^{\mathfrak H}(\omega)\|^2= &  \int_0^{2-t}(B_{t+x}^{\mathfrak H}(\omega)-B_x^{\mathfrak H}(\omega))^2dx+\int_{2-t}^2(B_x^{\mathfrak H}(\omega))^2dx\leq \bar C_{\omega} t^{2\mathfrak H}. 
\end{align*}
By analogous reasoning and since the adjoint semigroup $(\mathcal S(t)^*)_{t\geq 0}$ is given by the right-shift
$$\mathcal S(t)^*f(x)= f(x-t) \indicator_{[0,2]}(x-t),$$
we obtain for almost all $\omega \in \bar{\Omega}$ a $\bar C_{\omega}>0$ such that
\begin{align*}
     |\langle (\mathcal S(x)-I)f_1,f_2\rangle| \leq & \min(\|(\mathcal S(x)-I)B_1(\omega)\|\|B_2(\omega)\|,\|(\mathcal S(x)^*-I)B_2(\omega)\|\|B_1(\omega)\|)\\
     \leq & \bar C_{\omega} x^{\frac 14},
\end{align*}
and hence \eqref{invalid central limit theorem: technical  Condition on the example III} holds. 
It is now enough to prove that \eqref{invalid central limit theorem: technical  Condition on the example IV} holds for all $\omega\in A$ which are in a set $A\in\bar{\mathcal F}$  such that $\bar{\mathbb P}[A]>0$.
For that, it is enough to prove that there is a $c>0$ and an $N\in\mathbb N$, such that for all $n\geq N$ we have
\begin{equation}\label{Criterion: Expectation of divergent sequence for CLT is negative}
    \left|\mathbb E\left[\langle B_1,B_2\rangle\sum_{i=1}^n \int_{(i-1)\Delta_n}^{i\Delta_n} \Delta_n^{-(\mathfrak H+\epsilon)}\langle (\mathcal S(i\Delta_n-s)-I) B_1,B_2\rangle ds\right]\right|>c,
\end{equation}
since in this case, for $n\geq N$ we have $\bar{\mathbb P}\left[A\right]>0$ for the choice 
\begin{align*}
    A
=& \left\lbrace\left|\langle B_1,B_2\rangle\sum_{i=1}^n \int_{(i-1)\Delta_n}^{i\Delta_n} \Delta_n^{-\frac 12}\langle (\mathcal S(i\Delta_n-s)-I) B_1,B_2\rangle ds\right|>c\right\rbrace.
\end{align*}
If this would not be the case, there would be a subsequence $(n_k)_{k\in\mathbb N}$ such that on a full $\bar{\mathbb P}$-measure set $\bar{\Omega}_1$ we have
$$\left|\langle B_1,B_2\rangle\sum_{i=1}^{n_k} \int_{(i-1)\Delta_{n_k}}^{i\Delta_{n_k}} \Delta_{n_k}^{-\frac 12}\langle (\mathcal S(i\Delta_{n_k}-s)-I) B_1,B_2\rangle ds\right|\leq c.$$
Letting $\mathbb E_{\bar{\mathbb P}}$ denote the expectation with respect to the probability measure $\bar{\mathbb P}$,
we would have 
\begin{align*}
     &\left|\mathbb E_{\bar{\mathbb P}}\left[\langle B_1,B_2\rangle \sum_{i=1}^{n_k} \int_{(i-1)\Delta_{n_k}}^{i\Delta_{n_k}} \Delta_{n_k}^{-(\mathfrak H+\epsilon)}\langle (\mathcal S(i\Delta_{n_k}-s)-I) B_1,B_2\rangle ds\right]\right|\\
     = & \left|\int_{\bar{\Omega}_1}\langle B_1,B_2\rangle\sum_{i=1}^{n_k} \int_{(i-1)\Delta_{n_k}}^{i\Delta_{n_k}} \Delta_{n_k}^{-(\mathfrak H+\epsilon)}\langle (\mathcal S(i\Delta_{n_k}-s)-I) B_1,B_2\rangle ds(\omega)\bar{\mathbb P}[d\omega]\right|\\
     \leq & \int_{\bar{\Omega}_1}\left|\langle B_1,B_2\rangle\sum_{i=1}^{n_k} \int_{(i-1)\Delta_{n_k}}^{i\Delta_{n_k}} \Delta_{n_k}^{-\frac 12}\langle (\mathcal S(i\Delta_{n_k}-s)-I) B_1,B_2\rangle ds(\omega)\right|\bar{\mathbb P}[d\omega]\\
      \leq & \int_{\bar{\Omega}_1}c\bar{\mathbb P}[d\omega]\\
      = c.
\end{align*}
This would contradict \eqref{Criterion: Expectation of divergent sequence for CLT is negative}. That yields, in particular, that if \eqref{Criterion: Expectation of divergent sequence for CLT is negative} is valid we obtain that \eqref{invalid central limit theorem: technical  Condition on the example IV} holds for the choice $f_1:=B_1(\omega)$ and $f_2:=B_2(\omega)$ for any $\omega \in A$, as in this case
\begin{align*}
   & \limsup_{n\to\infty}\left|\Delta_n^{-\frac 12}\langle f_1,f_2\rangle \sum_{i=1}^n \int_{(i-1)\Delta_n}^{i\Delta_n} \langle (\mathcal S(i\Delta_n-s)-I) f_1,f_2\rangle ds\right|\\
    =& \limsup_{n\to\infty}\Delta_n^{-(\frac 12-(\mathfrak H+\epsilon))}\left|\Delta_n^{-(\mathfrak H+\epsilon)}\langle f_1,f_2\rangle \sum_{i=1}^n \int_{(i-1)\Delta_n}^{i\Delta_n} \langle (\mathcal S(i\Delta_n-s)-I) f_1,f_2\rangle ds\right|\\
    \geq & \limsup_{n\to\infty}\Delta_n^{-(\frac 12-(\mathfrak H+\epsilon))}c\\
    =&  \infty.
\end{align*}

Let us now prove that \eqref{Criterion: Expectation of divergent sequence for CLT is negative} holds to complete the proof.
By the Isserlis-Wick formula we obtain
\begin{align*}
    &\mathbb E_{\bar{\mathbb P}}\left[\langle B_1,B_2\rangle\sum_{i=1}^n \int_{(i-1)\Delta_n}^{i\Delta_n} \Delta_n^{-(\mathfrak H+\epsilon)}\langle (\mathcal S(i\Delta_n-s)-I) B_1,B_2\rangle ds\right]\\
    = & \sum_{i=1}^n \int_{(i-1)\Delta_n}^{i\Delta_n}\int_0^2\int_{2-i\Delta_n+s}^2 \Delta_n^{-(\mathfrak H+\epsilon)}\mathbb E_{\bar{\mathbb P}}\left[B_1(x)(B_1(y+i\Delta_n-s)\right.\\
    &\left.\qquad\qquad\qquad\qquad\qquad\qquad\qquad\qquad\qquad\qquad-B_1(y)) B_2(y)B_2(x)\right] dydxds \\
    & +\sum_{i=1}^n \int_{(i-1)\Delta_n}^{i\Delta_n}\int_0^2\int_0^{2-i\Delta_n+s} \Delta_n^{-(\mathfrak H+\epsilon)}\mathbb E_{\bar{\mathbb P}}\left[B_1(x)(B_1(y+i\Delta_n-s)-B_1(y))\right]\\
    &\qquad\qquad\qquad\qquad\qquad\qquad\qquad\qquad\qquad\qquad\quad\qquad\times\mathbb E_{\bar{\mathbb P}} \left[B_2(y)B_2(x)\right] dydxds \\
    & +\sum_{i=1}^n \int_{(i-1)\Delta_n}^{i\Delta_n}\int_0^2\int_0^{2-i\Delta_n+s} \Delta_n^{-(\mathfrak H+\epsilon)}\mathbb E_{\bar{\mathbb P}}\left[B_2(x)(B_1(y+i\Delta_n-s)-B_1(y))\right]\\
    &\qquad\qquad\qquad\qquad\qquad\qquad\qquad\qquad\qquad\qquad\qquad\quad\times\mathbb E \left[B_2(y)B_1(x)\right] dydxds \\
     & +\sum_{i=1}^n \int_{(i-1)\Delta_n}^{i\Delta_n}\int_0^2\int_0^{2-i\Delta_n+s}\Delta_n^{-(\mathfrak H+\epsilon)}\mathbb E_{\bar{\mathbb P}}\left[B_2(y)(B_1(y+i\Delta_n-s)-B_1(y))\right]\\
    &\qquad\qquad\qquad\qquad\qquad\qquad\qquad\qquad\qquad\qquad\qquad\quad\times\mathbb E_{\bar{\mathbb P}} \left[B_2(x)B_1(x)\right] dydxds\\
     =: & (1)_t^n+(2)_t^n+(3)_t^n+(4)_t^n.
    \end{align*}
    Clearly, since 
   \begin{equation}
     \sup_{r_{1,1},...,r_{1,d_1},r_{2,1},...,r_{2,d_2}\in [0,2]}\left|\mathbb E\left[\prod_{k=1}^{d_1}B_1(r_{i,1})\prod_{k=1}^{d_2}B_2(r_{2,k})\right]\right|<\infty,
   \end{equation}
     the first term goes to $0$ as $n\to \infty$, as 
     \begin{align*}
    &|(1)_t^n|\\
      =  & \left|\sum_{i=1}^n \int_{(i-1)\Delta_n}^{i\Delta_n}\int_0^2\int_{2-i\Delta_n+s}^2 \Delta_n^{-(\mathfrak H+\epsilon)}\mathbb E_{\bar{\mathbb P}}\left[B_1(x)(B_1(y+i\Delta_n-s)\right.\right.\\
      &\qquad\qquad\qquad\left.\left.-B_1(y)) B_2(y)B_2(x)\right] dydxds\right|\\
         \leq & \sum_{i=1}^n \int_{(i-1)\Delta_n}^{i\Delta_n}2(\Delta_n-s) \Delta_n^{-(\mathfrak H+\epsilon)} ds\\
         &\qquad\qquad\qquad\times\sup_{r_{1,1},r_{1,2},r_{2,1},r_{2,2}}|\mathbb E_{\bar{\mathbb P}}\left[B_1(r_{1,1})B_1(r_{1,2})B_2(r_{2,1})B_2(r_{2,1})\right]|\\
         \leq & 2 \Delta_n^{1-(\mathfrak H+\epsilon)} \sup_{r_{1,1},r_{1,2},r_{2,1},r_{2,2}}|\mathbb E_{\bar{\mathbb P}}\left[B_1(r_{1,1})B_1(r_{1,2})B_2(r_{2,1})B_2(r_{2,1})\right]|.
     \end{align*} 
    For the second term, observe that, by the mean value theorem, we have for all $x,y\in (0,2]$ such that $y,y-x,y+i\Delta_n-s-x\neq 0$ for $s\in [(i-1)\Delta_n,i\Delta_n]$ that
    \begin{align*}
       & \left||x-y|^{2\mathfrak H}+|y+i\Delta_n-s|^{2\mathfrak H}-|x-y-(i\Delta_n-s)|^{2\mathfrak H}-|y|^{2\mathfrak H}\right|\\
       \leq & \max(y^{2\mathfrak H-1},|x-y|^{2\mathfrak H-1},|x-y-(i\Delta_n-s)|^{2\mathfrak H-1} ) (i\Delta_n-s)\\
       \leq & (y^{2\mathfrak H-1}+|x-y|^{2\mathfrak H-1}+|x-y-(i\Delta_n-s)|^{2\mathfrak H-1})(i\Delta_n-s).
    \end{align*}
    Hence, 
    \begin{align*}
       & \left|\int_0^{2-i\Delta_n+s}\mathbb E_{\bar{\mathbb P}}\left[B_2(y)B_2(x)\right]\right.\\
       &\qquad\qquad\qquad\left.\times\left(|x-y|^{2\mathfrak H}+(y+i\Delta_n-s)^{2\mathfrak H}-|x-y-(\Delta_n-s)|^{2\mathfrak H}-y^{2\mathfrak H}\right)dy\right|\\
       \leq & \Delta_n \int_0^{2-i\Delta_n+s}(y^{2\mathfrak H-1}+|x-y|^{2\mathfrak H-1}+|x-y-(i\Delta_n-s)|^{2\mathfrak H-1})dy\\
       &\qquad\qquad\qquad\times\sup_{r,s\in [0,2]}|\mathbb E_{\bar{\mathbb P}} \left[B_2(r)B_2(s)\right]|.
    \end{align*}
    It is 
    \begin{align*}
     \int_0^{2-i\Delta_n+s}|x-y|^{2\mathfrak H-1}dy\leq   \int_0^2|x-y|^{2\mathfrak H-1}dy
     =& \int_0^x(x-y)^{2\mathfrak H-1}dy+\int_{x}^{2}(y-x)^{2\mathfrak H-1}dy\\
     = & \frac{x^{2\mathfrak H}+(2-x)^{2\mathfrak H}}{2\mathfrak H}\\
     \leq & \frac{2^{2\mathfrak H}}{\mathfrak H},
    \end{align*}
    and
    \begin{align*}
     \int_0^{2-i\Delta_n+s}|x-y-i\Delta_n+s|^{2\mathfrak H-1}dy=   \int_{i\Delta_n-s}^2|x-y|^{2\mathfrak H-1}dy
     \leq  \int_0^2|x-y|^{2\mathfrak H-1}dy
     \leq  \frac{2^{2\mathfrak H}}{\mathfrak H},
    \end{align*}
    as well as 
    $$ \int_0^{2-i\Delta_n+s} y^{2\mathfrak H-1}dy=\frac{(2-i\Delta_n+s)^{2\mathfrak H}}{\mathfrak H}\leq \frac{2^{2\mathfrak H}}{\mathfrak H}.$$
    Hence,
    \begin{align*}
      & \left|\int_0^{2-i\Delta_n+s}\mathbb E_{\bar{\mathbb P}}\left[B_2(y)B_2(x)\right]\left(|x-y|^{2\mathfrak H}+|y+i\Delta_n-s|^{2\mathfrak H}\right.\right.\\
      &\left.\left.\qquad\qquad\qquad\qquad\qquad-|x-y-(i\Delta_n-s)|^{2\mathfrak H}-|y|^{2\mathfrak H}\right)dy\right|\\
       \leq  & \Delta_n 3\frac{2^{2\mathfrak H}}{\mathfrak H}\sup_{r,s\in [0,2]}|\mathbb E_{\bar{\mathbb P}} \left[B_2(r)B_2(s)\right]|.
    \end{align*}
    This yields
    \begin{align*}
    &| (2)_t^n|\\
 = &\left|\sum_{i=1}^n \int_{(i-1)\Delta_n}^{i\Delta_n}\int_0^2\int_0^{2-i\Delta_n+s} \Delta_n^{-(\mathfrak H+\epsilon)}\mathbb E_{\bar{\mathbb P}}\left[B_1(x)(B_1(y+i\Delta_n-s)-B_1(y))\right]\right.\\
 &\left.\qquad\qquad\qquad\qquad\qquad\qquad\qquad\qquad\qquad\times\mathbb E_{\bar{\mathbb P}}\left[B_2(y)B_2(x)\right] dydxds \right| \\
       = &   \left|\sum_{i=1}^n \int_{(i-1)\Delta_n}^{i\Delta_n}\int_0^2\int_0^{2-i\Delta_n+s} \Delta_n^{-(\mathfrak H+\epsilon)}\frac 12\left(|x-y|^{2\mathfrak H}+|y+i\Delta_n-s|^{2\mathfrak H}\right.\right.\\
       &\left.\qquad\qquad\qquad\qquad\qquad\qquad\qquad\qquad-|x-y-(i\Delta_n-s)|^{2\mathfrak H}-|y|^{2\mathfrak H}\right)\\  &\left.\qquad\qquad\qquad\qquad\qquad\qquad\qquad\qquad\qquad\qquad\times\mathbb E_{\bar{\mathbb P}} \left[B_2(y)B_2(x)\right] dydxds \right|\\
        \leq &   \sum_{i=1}^n \int_{(i-1)\Delta_n}^{i\Delta_n}\int_0^2\Delta_n^{1-(\mathfrak H+\epsilon)} 3\frac{2^{2\mathfrak H}}{\mathfrak H}dxds\sup_{r,s\in [0,2]}| \mathbb E_{\bar{\mathbb P}}\left[B_2(r)B_2(s)\right]|\\
        = &    \Delta_n^{1-(\mathfrak H+\epsilon)} 6\frac{2^{2\mathfrak H}}{\mathfrak H}\sup_{r,s\in [0,2]}|\mathbb E \left[B_2(r)B_2(s)\right]|,
    \end{align*}
    which goes to $0$ as $n\to \infty$. By analogous reasoning we obtain that the third summand $(3)_t^n$ goes to $0$ as $n\to\infty$.
We now come to the fourth term. For that, we find 
\begin{align}\label{Technical decomposition of fourth summand}
&\sum_{i=1}^n \int_{(i-1)\Delta_n}^{i\Delta_n}\int_0^{2-i\Delta_n+s} \Delta_n^{-(\mathfrak H+\epsilon)}\mathbb E\left[B_2(y)(B_1(y+i\Delta_n-s)-B_1(y))\right] dyds\notag\\
    = & \sum_{i=1}^n \int_{(i-1)\Delta_n}^{i\Delta_n}\Delta_n^{-(\mathfrak H+\epsilon)} \int_0^{2-(i\Delta_n-s)} \mathbb E\left[\left(B_1(x+i\Delta_n-s)-B_1(x)\right)B_2(x)\right]dx ds\notag\\
    & -\sum_{i=1}^n \int_{(i-1)\Delta_n}^{i\Delta_n}\Delta_n^{-(\mathfrak H+\epsilon)} \int_{2-(i\Delta_n-s)}^2 \mathbb E\left[B_1(x)B_2(x)\right]dx ds\notag\\
    = & \sum_{i=1}^n \int_{(i-1)\Delta_n}^{i\Delta_n}\frac{\rho}{2\Delta_n^{\mathfrak H+\epsilon}} \int_0^{2-(i\Delta_n-s)} (x+i\Delta_n-s)^{\mathfrak H+\epsilon}-x^{\mathfrak H+\epsilon}-(i\Delta_n-s)^{\mathfrak H+\epsilon}dx ds\notag\\
    & -\sum_{i=1}^n \int_{(i-1)\Delta_n}^{i\Delta_n} \frac{\rho}{\Delta_n^{(\mathfrak H+\epsilon)}}\int_{2-(i\Delta_n-s)}^2 x^{\mathfrak H+\epsilon}dx ds\notag\\
    = & \frac{\rho}2\sum_{i=1}^n \int_{(i-1)\Delta_n}^{i\Delta_n} \int_0^{2-(i\Delta_n-s)} \frac{(x+i\Delta_n-s)^{\mathfrak H+\epsilon}-x^{\mathfrak H+\epsilon}}{\Delta_n^{\mathfrak H+\epsilon}}dx ds\notag\\
     & -\frac{\rho}2\sum_{i=1}^n \int_{(i-1)\Delta_n}^{i\Delta_n} \left(\frac{(i\Delta_n-s)}{\Delta_n}\right)^{\mathfrak H+\epsilon}(2-(i\Delta_n-s)) ds\notag\\
    & -\sum_{i=1}^n \int_{(i-1)\Delta_n}^{i\Delta_n} \frac{\rho}{\Delta_n^{(\mathfrak H+\epsilon)}}\int_{2-(i\Delta_n-s)}^2 x^{\mathfrak H+\epsilon}dx ds.
\end{align}
The first and the third summand converge to $0$ as $n\to \infty$, as by the mean value theorem
\begin{align*}
   & \left|\frac{\rho}2\sum_{i=1}^n \int_{(i-1)\Delta_n}^{i\Delta_n} \int_0^{2-(i\Delta_n-s)} \frac{(x+i\Delta_n-s)^{\mathfrak H+\epsilon}-x^{\mathfrak H+\epsilon}}{\Delta_n^{\mathfrak H+\epsilon}}dx ds\right|\\
   \leq & \frac{\rho}2\sum_{i=1}^n \int_{(i-1)\Delta_n}^{i\Delta_n} \int_0^{2-(i\Delta_n-s)} \Delta_n^{1-(\mathfrak H+\epsilon)}x^{\mathfrak H+\epsilon-1}dx ds\\
   \leq & \Delta_n^{1-(\mathfrak H+\epsilon)} \frac{\rho}2 \int_0^{2} x^{\mathfrak H+\epsilon-1}dx.
\end{align*}
and
\begin{align*}
   & \left|\sum_{i=1}^n \int_{(i-1)\Delta_n}^{i\Delta_n} \frac{\rho}{\Delta_n^{(\mathfrak H+\epsilon)}}\int_{2-(i\Delta_n-s)}^2 x^{\mathfrak H+\epsilon}dx ds\right|\\
    \leq & \rho \Delta_n^{1-(\mathfrak H+\epsilon)} 2^{\mathfrak H+\epsilon}.
\end{align*}
Summing up, we obtain that for any $\eta >0$ there is an $N\in\mathbb N$, such that for all $n\geq N$ we have as $(1)_t^n$, $(2)_t^n$, $(3)_t^n$ and the first and third term in \eqref{Technical decomposition of fourth summand} go to $0$ as $n\to \infty$
\begin{align*}
    &  \left|\mathbb E\left[\langle B_1,B_2\rangle\sum_{i=1}^n \int_{(i-1)\Delta_n}^{i\Delta_n} (i\Delta_n)^{-(\mathfrak H+\epsilon)}\langle (\mathcal S(i\Delta_n-s)-I) B_1,B_2\rangle ds\right]\right|\\
    \geq & \left|\int_0^2 \mathbb E \left[B_1(x)B_2(x)\right]dx\right|\frac{\rho}2\sum_{i=1}^n \int_{(i-1)\Delta_n}^{i\Delta_n} \left(\frac{(i\Delta_n-s)}{\Delta_n}\right)^{\mathfrak H+\epsilon}(2-(i\Delta_n-s)) ds-\eta\\
    \geq & \left|\int_0^2 \mathbb E \left[B_1(x)B_2(x)\right]dx\right|\frac{\rho}{2\Delta_n^{\mathfrak H+\epsilon}}\sum_{i=1}^n \frac{\Delta_n^{\mathfrak H+\epsilon+1}}{\mathfrak H+\epsilon+1}-\eta\\
    = &\left|\int_0^2 \mathbb E \left[B_1(x)B_2(x)\right]dx\right| \frac{\rho}{2(\mathfrak H+\epsilon+1)}-\eta.
\end{align*}
As this holds for all $\eta>0$, we obtain \eqref{Criterion: Expectation of divergent sequence for CLT is negative} and hence the proof.
\end{proof}


We continue by proving the remaining assertions of Example \ref{Counterexample for RV}

   \begin{proof}[Proof of the remaining assertions of Example \ref{Counterexample for RV}]
In order to verify the validity of the counterexample \ref{Counterexample for RV} we still have to show that if $X= B^{\mathfrak H}$,
\begin{itemize}
    \item[(i)]$(RV_t^n-\int_0^t\Sigma_sds-\sum_{i=1}^n ((\mathcal S(\Delta_n)-I)Y_{(i-1)\Delta_n})^{\otimes 2})$ converges in probability to $0$ and
    \item[(ii)] $\sum_{i=1}^{n}\|(\mathcal S(\Delta_n)-I)Y_{(i-1)})\|^4$ is uniformly integrable.
\end{itemize}
Moreover, in the second case, in which $X=\indicator_{[0,1]}$
we must show that
\begin{itemize}
    \item[(iii)] $\Delta_n^{-\frac 12}(SARCV_t^n-\int_0^t\Sigma_sds)$ is uniformly integrable and
    \item[(iv)]$\Delta_n^{-\frac 12}(RV_t^n-\int_0^t\Sigma_sds)$ is uniformly integrable.
\end{itemize}
Observe that \begin{align*}
    RV_t^n =\sum_{i=1}^{\ul}&\Delta_i^n Y^{\otimes 2}\\
    =\sum_{i=1}^{\ul}&\tilde{\Delta}_i^n Y^{\otimes 2}+\tilde{\Delta}_i^n Y\otimes(\mathcal S(\Delta_n)-I)Y_{(i-1)\Delta_n}\\
    &+(\mathcal S(\Delta_n)-I)Y_{(i-1)\Delta_n}\otimes\tilde{\Delta}_i^n Y+[(\mathcal S(\Delta_n)-I)Y_{(i-1)\Delta_n}]^{\otimes 2}.
\end{align*}
Thus, recalling that $\Sigma_s^{\mathcal S_n}=\mathcal S(i\Delta_n-s)\Sigma_s \mathcal S(i\Delta_n-s)^*$
for $s\in((i-1)\Delta_n,i\Delta_n]$
\begin{align*}
   & (RV_t^n-\int_0^t\Sigma_sds)\\ 
    = & (SARCV_t^n-\int_0^t\Sigma_s^{\mathcal S_n}ds)+\Delta_n^{-\frac 12}\int_0^t\Sigma_s^{\mathcal S_n}-\Sigma_s ds\\ 
    & +\sum_{i=1}^{\ul}\tilde{\Delta}_i^n Y\otimes(\mathcal S(\Delta_n)-I)Y_{(i-1)\Delta_n}\\
    &+\sum_{i=1}^{\ul}(\mathcal S(\Delta_n)-I)Y_{(i-1)\Delta_n}\otimes\tilde{\Delta}_i^n Y\\
    &+\sum_{i=1}^{\ul}[(\mathcal S(\Delta_n)-I)Y_{(i-1)\Delta_n}]^{\otimes 2}\\
= & (1)_t^n+(2)_t^n+(3)_t^n+(4)_t^n +(5)_t^n. 
\end{align*}
The law of large numbers \ref{T: LLN for the SARCV} guarantees that $(1)_t^n+(2)_t^n$ converges to $0$ in probability with respect to the Hilbert-Schmidt norm. Moreover,
for $(3)_t^n$ (and analogously $(4)_t^n$) we find in the first case that, with the notation $\Delta_i\mathcal S=\mathcal S(i\Delta_n)-\mathcal S((i-1)\Delta_n)$,
\begin{align*}
   & \mathbb E\left[\|(3)_t^n\|^2\right]\\
    =&\sum_{i,j=1}^n \mathbb E\left[\langle \tilde{\Delta}_i^n Y\otimes (\mathcal S(\Delta_n)-I)Y_{(i-1)\Delta_n},\tilde{\Delta}_j^n Y\otimes (\mathcal S(\Delta_n)-I)Y_{(j-1)\Delta_n}\rangle_{\mathcal H}\right]\\
    = & \sum_{i,j=1}^n \mathbb E\left[(\beta_{i\Delta_n}-\beta_{(i-1)\Delta_n})\beta_{(i-1)\Delta_n}(\beta_{j\Delta_n}-\beta_{(j-1)\Delta_n})\beta_{(j-1)\Delta_n}\right]\\
    &\qquad\times\mathbb E\left[\langle \mathcal S(i\Delta_n) B^{\mathfrak H}\otimes \Delta_i\mathcal S B^{\mathfrak H},\mathcal S(j\Delta_n) B^{\mathfrak H}\otimes \Delta_j\mathcal S B^{\mathfrak H}\rangle_{\mathcal H}\right]\\
     = & \sum_{i=1}^n (i-1)\Delta_n^2\mathbb E\left[\|\mathcal S(i\Delta_n) B^{\mathfrak H}\otimes \Delta_i\mathcal S B^{\mathfrak H}\|_{\mathcal H}^2\right]\\
     \leq & \sum_{i=1}^n (i-1)\Delta_n^2\mathbb E\left[\|\mathcal S(i\Delta_n) B^{\mathfrak H}\|_{\mathcal H}^4\right]^{\frac 12} \mathbb E\left[\|\Delta_i\mathcal S B^{\mathfrak H}\|_{\mathcal H}^4\right]^{\frac 12}\\
     = & \sup_{r\in [0,T]}\|\mathcal S(r)\|_{\text{op}}^4 \mathbb E\left[\| B^{\mathfrak H}\|_{\mathcal H}^4\right]^{\frac 12} \mathbb E\left[\|(\mathcal S(\Delta_n)-I)B^{\mathfrak H}\|_{\mathcal H}^4\right]^{\frac 12}\sum_{i=1}^n (i-1)\Delta_n^2.
\end{align*}
This converges to $0$, as $n\to \infty$ and thus $(3)_t^n\to 0$ (as well as $(4)_t^n\to 0$) as $n\to \infty$ in $L^2(\Omega)$. This shows the first point (i).

In order to prove (ii), observe that as by Jensen's inequality it holds (since the central eight's moment of a Gaussian random variable $Z\sim N(0,\rho^2)$ is $\mathbb E[Z^8]=105\rho^8$)
\begin{align*}
    \mathbb E\left[\|\Delta_j\mathcal SB^{\mathfrak H}\|_H^{8}\right]\leq  & \int_0^2 \mathbb E \left[\left(B^{\mathfrak H}_{x+i\Delta_n}-B^{\mathfrak H}_{x+(i-1)\Delta_n}\right)^8 \right]dx= 210 \Delta_n^{8H}=210 \Delta_n^2,
\end{align*}
we have
\begin{align*}
   & \mathbb E\left[\left(\sum_{i=1}^{n}\|(\mathcal S(\Delta_n)-I)Y_{(i-1)\Delta_n})\|^4\right)^2\right]\\
      = & \sum_{i,j=1}^{n}\mathbb E\left[\beta_{(i-1)\Delta_n}^4\beta_{(j-1)\Delta_n}^4\right]\mathbb E\left[\|\Delta_i\mathcal SB^{\mathfrak H}\|_H^4\|\Delta_j\mathcal S B^{\mathfrak H}\|_H^4\right]\\
     \leq  & \sum_{i,j=1}^{n}\mathbb E\left[\beta_{(i-1)\Delta_n}^{8}\right]^{\frac 12}\mathbb E\left[\beta_{(j-1)\Delta_n}^{8}\right]^{\frac 12}\mathbb E\left[\|\Delta_i\mathcal S B^{\mathfrak H}\|_H^{8}\right]^{\frac 12}\mathbb E\left[\|\Delta_j\mathcal SB^{\mathfrak H}\|_H^{8}\right]^{\frac 12}\\
     \leq
   & 210\times 105 \sum_{i,j=1}^n (i-1)^2(j-1)^2 \Delta_n^4 \Delta_n^2\\
   \leq & 210\times 105,
\end{align*}
which yields the $L^2(\Omega)$-boundedness of $\sum_{i=1}^{n}\|(\mathcal S(\Delta_n)-I)Y_{(i-1)})\|^4$ and, in particular, point (ii).

In order to show uniform integrability for the points (iii) and (iv) in the second case, in which $X=\indicator_{[0,1]}$, it is enough to show (after normalisation by $\sqrt{n}$) that all summands $\sqrt n(1)_t^n-\sqrt n(5)_t^n$ are bounded in $L^2(\Omega)$ uniformly in $n\in\mathbb N$.
The first summand $\sqrt n(1)_t^n$ is bounded in $L^2(\Omega)$, due to Theorem \ref{T: Tightnes for realised covariation}.

For the second term we observe
\begin{align*}
   &  \mathbb E\left[\left\| \sum_{i=1}^{\ul}\int_{(i-1)\Delta_n}^{i\Delta_n} \left(\Sigma_s^{\mathcal S_n}-\Sigma_s\right) ds\right\|_{\mathcal H}^2\right]^{\frac 12}\\
       \leq &  \int_0^T\mathbb E\left[\left\|(\mathcal S(\lfloor s/\Delta_n\rfloor \Delta_n-s)-I)\Sigma_s \mathcal S(\lfloor s/\Delta_n\rfloor \Delta_n-s)^*\right\|_{\mathcal H}^2\right]^{\frac 12}\\
&\qquad\qquad\qquad+\mathbb E\left[\left\| \Sigma_s (\mathcal S(\lfloor s/\Delta_n\rfloor \Delta_n-s)^*-I)\right\|_{\mathcal H}^2\right]^{\frac 12}ds\\
       \leq &  2 \sup_{r\in [0,T]}\|\mathcal S(r)\|_{\text{op}}\int_0^T \sup_{x\in [0,\Delta_n]}\mathbb E\left[\left\|(\mathcal S(x)-I)\sigma_s\|_{\text{op}}^2 \| \sigma_s\right\|_{\mathcal H}^2\right]^{\frac 12}.
        \end{align*}
 It holds for $x\geq 0$, $\|\sigma_s\|_{L_{\text{HS}}(U,H)}=1$ and 
        $$\|(\mathcal S(x)-I)\sigma_s\|_{_{L_{\text{HS}}(U,H)}}^2= \|(\mathcal S(x+s)-\mathcal S(s))\indicator_{[0,1]} \|_H^2 
        = 2x.$$
This yields        \begin{align*}
   &  \Delta_n^{-\frac 12}\mathbb E\left[\left\| \sum_{i=1}^{\ul}\int_{(i-1)\Delta_n}^{i\Delta_n} \left(\Sigma_s^{\mathcal S_n}-\Sigma_s\right) ds\right\|_{\mathcal H}^2\right]^{\frac 12}\\
       \leq &  \Delta_n^{-\frac 12}2 \sup_{r\in [0,T]}\|\mathcal S(r)\|_{\text{op}}\int_0^T \sup_{x\in [0,\Delta_n]}\|(\mathcal S(x)-I)\sigma_s\|_{\text{op}} \| \sigma_s\|_{\mathcal H}ds\\
      \leq & 4 \sup_{r\in [0,T]}\|\mathcal S(r)\|_{\text{op}}^2 T.
        \end{align*}
        Hence $\sqrt n (2)_t^n$ is bounded in $L^2(\Omega)$. We just proved point (iii). 

Now we turn to the $L^2(\Omega)$-boundedness of $\sqrt n(3)_t^n$ ($\sqrt n(4)_t^n$ is analogous). We have
\begin{align*}
&    \langle \left(\mathcal S(i\Delta_n)-\mathcal S(\Delta_n(i-1))\right) \indicator_{[0,1]},\left(\mathcal S(i\Delta_n)-\mathcal S(\Delta_n(i-1))\right) \indicator_{[0,1]}\rangle\\
    = &\int_{\mathbb R} \left(\indicator_{[-i\Delta_n,-(i-1)\Delta_n]}(y)-\indicator_{[1-i\Delta_n,1-(i-1)\Delta_n]}(y)\right)\\ &\qquad\qquad\times\left(\indicator_{[-j\Delta_n,-(j-1)\Delta_n]}(y)-\indicator_{[1-j\Delta_n,1-(j-1)\Delta_n]}(y)\right)dy\\
    = & \delta_{i,j} 2\Delta_n.
\end{align*}
This yields 
\begin{align*}
  &  \langle \left(\mathcal S(\Delta_n)-I\right) Y_{(i-1)\Delta_n},\left(\mathcal S(\Delta_n)-I\right) Y_{(j-1)\Delta_n}\rangle\\
    = & \beta_{(i-1)\Delta_n} \beta_{(j-1)\Delta_n}  \langle \left(\mathcal S(\Delta_ni)-\mathcal S(\Delta_n(i-1))\right) \indicator_{[0,1]},\left(\mathcal S(\Delta_nj)-\mathcal S(\Delta_n(j-1))\right) \indicator_{[0,1]}\rangle\\
    = &  \delta_{i,j} \beta_{(i-1)\Delta_n}^2 2\Delta_n.
\end{align*}
Thus,
\begin{align*}
    &\|\Delta_n^{-\frac 12}(3)_t\|_{\mathcal H}^2\\
    = &\Delta_n^{-1} \sum_{i,j=1}^n \langle\tilde{\Delta}_i^n Y\otimes(\mathcal S(\Delta_n)-I)Y_{(i-1)\Delta_n},\tilde{\Delta}_j^n Y\otimes(\mathcal S(\Delta_n)-I)Y_{(j-1)\Delta_n}\rangle_{\mathcal H} \\
    = & \Delta_n^{-1}\sum_{i,j=1}^n \langle\tilde{\Delta}_i^n Y,\tilde{\Delta}_j^n Y\rangle \langle(\mathcal S(\Delta_n)-I)Y_{(i-1)\Delta_n},(\mathcal S(\Delta_n)-I)Y_{(j-1)\Delta_n}\rangle \\
    = & 2\sum_{i=1}^n \|\tilde{\Delta}_i^n Y\|^2 \beta_{(i-1)\Delta_n}^2.
\end{align*}
This gives, by independence of $\beta_{(i-1)\Delta_n}$ and $\tilde{\Delta}_i^n Y$,
\begin{align*}
    \mathbb E\left[\|\Delta_n^{-\frac 12}(3)_t\|_{\mathcal H}^2\right]
    = & 2\Delta_n\sum_{i=1}^n (i-1)\int_{(i-1)\Delta_n}^{i\Delta_n}\mathbb E\left[\|\sigma_s^{\mathcal S_n}\|_{L_{\text{HS}}(U,H)}^2\right]ds 
    \leq  2\sup_{r\in [0,T]}\|\mathcal S(r)\|_{\text{op}}^2.
\end{align*}
Hence, the $L^2(\Omega)$-boundedness of $\sqrt n (3)_t^n$ (and $\sqrt n (4)_t^n$) follows. 
It remains to show the $L^2(\Omega)$-boundedness of $(5)_t^n$. We find
\begin{align*}
   \mathbb E\left[ \|(5)_t^n\|^2\right]= & \mathbb E\left[\|\Delta_n^{-\frac 12}\sum_{i=1}^{n}[(\mathcal S(\Delta_n)-I)Y_{(i-1)}]^{\otimes 2}\|^2\right]\\
    = & \mathbb E\left[\Delta_n^{-1}\sum_{i,j=1}^{n} \langle \mathcal S(\Delta_n)-I)Y_{(i-1)},\mathcal S(\Delta_n)-I)Y_{(j-1)}\rangle^2\right]\\
    = & \mathbb E\left[\Delta_n^{-1}\sum_{i=1}^n 2 \Delta_n^2 \beta_{(i-1)\Delta_n}^4\right]\\
    \leq  &  6.
\end{align*}
This yields the $L^2(\Omega)$-boundedness of $(5)_t^n$ and, thus, we proved (iv).
\end{proof}

Now we give the proof of Theorem \ref{T: Limit Theorems for nonadjusted RV}.
\begin{proof}[Proof of Theorem \ref{T: Limit Theorems for nonadjusted RV}]
Recall that by Theorem \ref{L: localisation}(b) and (d) we can assume that Assumption \ref{As: Decomposition of Q} holds. This yields for the proof of the law of the large numbers that by the dominated convergence theorem
$$\lim_{t\to 0}\mathbb E\left[\int_0^T\| t^{-\frac 12}(I-\mathcal S(t))\sigma_s\|_{L_{\text{HS}(U,H)}}^2ds\right]=0,$$
and for the proof of the central limit theorem
$$\lim_{t\to 0}\mathbb E\left[\int_0^T\| t^{-\frac 34}(I-\mathcal S(t))\sigma_s\|_{L_{\text{HS}(U,H)}}^2ds\right]=0.$$
Then observe that we have 
\begin{align*}
    RV_t^n =\sum_{i=1}^{\ul}&\Delta_i^n Y^{\otimes 2}\\
    =\sum_{i=1}^{\ul}&\tilde{\Delta}_i^n Y^{\otimes 2}+\tilde{\Delta}_i^n Y\otimes(\mathcal S(\Delta_n)-I)Y_{(i-1)\Delta_n}\\
    &+(\mathcal S(\Delta_n)-I)Y_{(i-1)\Delta_n}\otimes\tilde{\Delta}_i^n Y+[(\mathcal S(\Delta_n)-I)Y_{(i-1)\Delta_n}]^{\otimes 2} \\
    = (1)_t+&(2)_t+(3)_t+(4)_t.
\end{align*}
We know that the first summand converges in u.c.p. to the integrated volatility $\int_0^t \Sigma_s ds$. 
Under this assumption, we obtain for the second and third summand
\begin{align*}
    &\frac 12\sup_{t\in [0,T]} \mathbb E\left[\|(2)_t+(3)_t\|_{\mathcal H}\right]\leq \sum_{i=1}^{\ulT}\mathbb E\left[\|(\mathcal S(\Delta_n)-I)Y_{(i-1)\Delta_n}\| \|\tilde{\Delta}_i^n Y\|\right] \\
    & = \left(\sum_{i=1}^{\ulT}\mathbb E\left[\|\tilde{\Delta}_i^n Y\|^2\right]\right)^{\frac 12}\left(\sum_{i=1}^{\ulT}\mathbb E\left[\|(\mathcal S(\Delta_n)-I)Y_{(i-1)\Delta_n}\|^2\right]\right)^{\frac 12}\\
    & \leq \sup_{r\in[0,T]}\|\mathcal S(r)\|_{\text{op}} \left(\sum_{i=1}^{\ulT}\int_{(i-1)\Delta_n}^{i\Delta_n}\mathbb E\left[\|\sigma_s\|_{L_{\text{HS}(U,H)}}^2\right]ds\right)^{\frac 12}\\
    & \qquad\times\left(\sum_{i=1}^{\ulT}\int_0^T\mathbb E\left[\|(\mathcal S(\Delta_n)-I)\sigma_s\|_{L_{\text{HS}(U,H)}}^2\right]ds\right)^{\frac 12}\\
    &= o(1),
\end{align*}
where the last equality is by assumption.  
Moreover, the last summand fulfills immediately
\begin{align*}
   \sup_{t\in[0,T]} (4)_t\leq \sup_{r\in[0,T]}\|\mathcal S(r)\|_{\text{op}}\sum_{i=1}^{\ul}\int_0^T\mathbb E\left[\|(\mathcal S(\Delta_n)-I)\sigma_s\|_{L_{\text{HS}(U,H)}}^2\right]ds= o(1).
\end{align*}
This proves the claim for the law of large numbers. The central limit theorem follows by analogous reasoning, after normalising by $\Delta_n^{-\frac 12}$.
\end{proof}
Now we give the proof of Theorem \ref{T:Finite dimensional limit theorems for functionals of nonadjusted covariation}
\begin{proof}[Proof of Theorem \ref{T:Finite dimensional limit theorems for functionals of nonadjusted covariation}]
We can argue componentwise, which is why we assume without loss of generality that $B=h\otimes g$ for $h,g\in F^{\mathcal S^*}_{1/2}$ for (i) or $h,g\in F^{\mathcal S^*}_{1/4}$ for (ii) respectively. Again, we appeal to the decomposition
\begin{align*}
    \langle RV_t^n h,g\rangle=\sum_{i=1}^{\ul}&\langle \Delta_i^n Y^{\otimes 2}h ,g\rangle\\
    =\sum_{i=1}^{\ul}&\langle\tilde{\Delta}_i^n Y^{\otimes 2}h,g\rangle+\langle \tilde{\Delta}_i^n Y\otimes(\mathcal S(\Delta_n)-I)Y_{(i-1)\Delta_n}h,g\rangle\\
    &+\langle(\mathcal S(\Delta_n)-I)Y_{(i-1)\Delta_n}\otimes\tilde{\Delta}_i^n Y h,g\rangle+\langle[(\mathcal S(\Delta_n)-I)Y_{(i-1)}]^{\otimes 2}h,g\rangle \\
    = (1)_n^t+&(2)_n^t+(3)_n^t+(4)_n^t.
\end{align*}
The first summand converges to $\int_0^T \langle \Sigma_s h,g \rangle ds$, and after normalisation with $\Delta_n^{\frac 12}$ it is asymptotically centred normal with variance $\langle \Gamma_t h\otimes g,h\otimes g\rangle$. Therefore, it is for the law of large numbers enough to show
$$\left((2)_n^t+(3)_n^t+(4)_n^t\right)\stackrel{u.c.p.}{\longrightarrow} 0\quad \text{ as }n\to\infty.$$
and for the central limit theorem  to show
$$\Delta_n^{-\frac 12}\left((2)_n^t+(3)_n^t+(4)_n^t\right)\stackrel{u.c.p.}{\longrightarrow} 0\quad \text{ as }n\to\infty.$$
For the third (and analogously for the second) summand we have
\begin{align*}
   |(3)_n^t|= & \left| \sum_{i=1}^{\ul}\langle (\mathcal S(\Delta_n)-I)Y_{(i-1)\Delta_n}\otimes\tilde{\Delta}_i^n Y  h,g\rangle\right|\\
    \leq &  \left(\sum_{i=1}^{\ul}\langle (\mathcal S(\Delta_n)-I)Y_{(i-1)\Delta_n}, h\rangle^2\right)^{\frac 12}\left(\sum_{i=1}^{\ul}\langle \tilde{\Delta}_i^n Y  ,g\rangle^2\right)^{\frac 12}\\
    = & \left(\sum_{i=1}^{\ul}\langle (\mathcal S(\Delta_n)-I)Y_{(i-1)\Delta_n}, h\rangle^2\right)^{\frac 12}\left(\sum_{i=1}^{\ul}\langle \tilde{\Delta}_i^n Y^{\otimes 2}g  ,g\rangle\right)^{\frac 12},
\end{align*}
and the second factor converges to an asymptotically normal law. For the fourth summand we have
\begin{align*}
   |(4)_n^t|= &\left|\sum_{i=1}^{\ul}\langle(\mathcal S(\Delta_n)-I)Y_{(i-1)}^{\otimes 2}h,g\rangle\right|\\
    \leq &\left(\sum_{i=1}^{\ul}\langle(\mathcal S(\Delta_n)-I)Y_{(i-1)},h\rangle^2\right)^{\frac 12}\left(\sum_{i=1}^{\ul}\langle(\mathcal S(\Delta_n)-I)Y_{(i-1)},g\rangle^2\right)^{\frac 12},
\end{align*}
and, thus, for the law of large numbers it is enough to show for all $h\in F_{\frac 12}^{\mathcal S^*}$,
\begin{equation}\label{RV-LLN Auxiliary Equation}\sum_{i=1}^{\ul}\langle(\mathcal S(\Delta_n)-I)Y_{(i-1)},h\rangle^2\stackrel{u.c.p.}{\longrightarrow} 0, \quad \text{ as }n\to\infty
\end{equation}
and for the central limit theorem for all $h\in F_{\frac 34}^{\mathcal S^*}$,
\begin{equation}\label{RV-CLT Auxiliary Equation}
    \Delta_n^{-\frac 12}\sum_{i=1}^{\ul}\langle(\mathcal S(\Delta_n)-I)Y_{(i-1)},h\rangle^2\stackrel{u.c.p.}{\longrightarrow} 0, \quad \text{ as }n\to\infty.
\end{equation}
Due to Theorem \ref{L: localisation}(b) (respectively (a) for the central limit theorem) we can suppose that Assumption \ref{As: Decomposition of Q} (or Assumption \ref{As: Weakened Assumption for the CLT for the SARCV} for the central limit theorem) is valid. In that case, we have for the proof of the law of the large numbers by the dominated convergence theorem
$$\lim_{t\to 0}\mathbb E\left[\int_0^T\| t^{-\frac 12}(I-\mathcal S(t))\|_{\mathcal H}^2ds\right]=0$$
and for the proof of the central limit theorem
$$\lim_{t\to 0}\mathbb E\left[\int_0^T\| t^{-\frac 34}(I-\mathcal S(t))\sigma_s\|_{\mathcal H}^2ds\right]=0.$$
Moreover, we have 
\begin{align*}
   &  \sum_{i=1}^{\ulT}\mathbb E\left[\langle(\mathcal S(\Delta_n)-I)Y_{(i-1)},h\rangle^2\right]\\
    = & \sum_{i=1}^{\ulT}\int_0^{(i-1)\Delta_n} \mathbb E\left[\langle (\mathcal S(i\Delta_n)-\mathcal S((i-1)\Delta_n))\Sigma_s^{\mathcal S_n} (\mathcal S(i\Delta_n)-\mathcal S((i-1)\Delta_n))^* h,h\rangle\right] ds\\
    \leq &  T \int_0^T \sup_{t\in[0,T]} \mathbb E\left[ \Delta_n^{-1}\langle (\mathcal S(t+\Delta_n)-\mathcal S(t))\Sigma_s^{\mathcal S_n} (\mathcal S(t+\Delta_n)-\mathcal S(t))\Delta_n))^* h,h\rangle\right] ds.
\end{align*}

Henceforth, in order to show \eqref{RV-LLN Auxiliary Equation} and \eqref{RV-CLT Auxiliary Equation} it is enough to show that, for all $\gamma\in (0,1)$ and $h\in F_{\gamma}^{\mathcal S^*}$, we have, as $n\to\infty$,
\begin{align*}
    \int_0^T\sup_{t\in[0,T]}\mathbb E\left[\Delta^{-2\gamma}\langle (\mathcal S(t+\Delta)-\mathcal S(t))\Sigma_s (\mathcal S(t+\Delta)-\mathcal S(t))^* h,h\rangle\right]ds\stackrel{u.c.p.}{\longrightarrow}0.
\end{align*}
We note that 
\begin{align*}
&\int_0^T\sup_{t\in[0,T]}\mathbb E\left[\Delta^{-2\gamma}\langle (\mathcal S(t+\Delta)-\mathcal S(t))\Sigma_s (\mathcal S(t+\Delta)-\mathcal S(t))^* h,h\rangle\right]ds\\
\leq&\int_0^T\sup_{t\in[0,T]}\mathbb E\left[\Delta^{-2\gamma}\langle (I-p_N)\Sigma_s (\mathcal S(t+\Delta)-\mathcal S(t))^* h,(\mathcal S(t+\Delta)-\mathcal S(t))^*h\rangle\right]ds\\
    &+ \int_0^T \sup_{t\in[0,T]}\mathbb E\left[\Delta^{-2\gamma}\langle p_N\Sigma_s (\mathcal S(t+\Delta_n)-\mathcal S(t))^* h,\Delta_n^{-\gamma}(\mathcal S(t+\Delta)-\mathcal S(t))^*h\rangle^2\right]ds\\
   = & (1)_{n,N}+(2)_{n,N},
\end{align*}
where again we denoted by $p_N$ the projection onto $v^N:=\overline{\{e_j: j\geq N\}}$ for an orthonormal basis $(e_j)_{j\in\mathbb N}$ of $H$
 that is contained in $D(\mathcal A)$. We have 
 \begin{align*}
    \mathcal S(t+\Delta_n) e_i- \mathcal S(t)e_i= \int_t^{t+\Delta_n} \mathcal S(u)\mathcal A e_i du
\end{align*}
and therefore we find for the first summand that 
\begin{align*}
&(1)_{n,N}\\
   \leq  & \sum_{j=1}^{N-1}\left|\int_0^T \sup_{t\in[0,T]}\mathbb E\left[\Delta^{-2\gamma}\langle \Sigma_s (\mathcal S(t+\Delta_n)-\mathcal S(t))^* h, e_j\rangle\right.\right.\\
   &\left.\left.\qquad\qquad\qquad\qquad\qquad\qquad\qquad\qquad\times\langle e_j, \Delta_n^{-\gamma}(\mathcal S(t+\Delta_n)-\mathcal S(t))^*h\rangle\right] ds\right|\\
    =  & \sum_{j=1}^{N-1}\Delta^{-2\gamma}\left|\int_0^T \sup_{t\in[0,T]}\mathbb E\left[\langle \Sigma_s (\mathcal S(t+\Delta_n)-\mathcal S(t))^* h, e_j\rangle \langle \int_t^{t+\Delta_n} \mathcal S(u)\mathcal A e_j ds, h\rangle\right] ds\right| \\
      \leq & \sum_{j=1}^{N-1}\Delta_n^{1-\gamma}\left(\int_0^T \|\sigma_s\|_{\text{op}}^2 ds\right)^{\frac 12} \|h\|^2\|\mathcal A e_j\| \sup_{t\in[0,T]}\|\mathcal S(t)\|_{\text{op}}\\ &\qquad\qquad\qquad\qquad\qquad
      \times\sup_{t\in[0,T]}\|\Delta_n^{-\gamma}(\mathcal S(t+\Delta_n)-\mathcal S(t))^* h\|,\\
\end{align*}
which converges to $0$ as $n\to\infty$. 
Moreover, it follows 
that for the second summand we have
\begin{align*}
& (2)_{n,N}
    \leq \Delta_n^{-2\gamma}\|(\mathcal S(\Delta)-I)^*h\|^2\left(\int_0^T \sup_{t\in [0,T]}\mathbb E\left[\| p_N\mathcal S(t)\sigma_s\|_{\text{op}}^2\right]ds\right)^{\frac 12},
\end{align*}
where the first factor is bounded by Assumption on $h$ and
$\int_0^T \mathbb E\left[\| p_N\mathcal S(t)\sigma_s\|_{\text{op}}^2\right]ds$ converges to $0$ as $N\to\infty$, as it can be shown analogously to the proof of Lemma \ref{L: Projection convergese uniformly on the range of volatility}. We obtain that $ \mathbb E [\sup_{n\in\mathbb N}|(1.2)_{n,N}|]$ converges to $0$ as $N\to\infty$. 
Therefore, we can find for each $\delta>0$ an $N\in\mathbb N$, such that by Markov's inequality
\begin{align*}
  & \lim_{n\to\infty} \mathbb P\left[\int_0^T\sup_{t\in[0,T]}\mathbb E\left[\Delta^{-2\gamma}\langle (\mathcal S(t+\Delta)-\mathcal S(t))\Sigma_s (\mathcal S(t+\Delta)-\mathcal S(t))^* h,h\rangle\right]ds>\epsilon\right]\\
    \leq & \lim_{n\to\infty} \mathbb P\left[|(1)_{n,N}|>\epsilon\right]+\sup_{n\in \mathbb N} \mathbb P\left[|(2)_{n,N}|>\epsilon\right]\\
    = & 0+ \frac 1{\epsilon}\sup_{n\in \mathbb N} \mathbb E\left[|(2)_{n,N}|\right]\leq \delta.
\end{align*}
As this holds for all $\delta>0$, we obtain
the assertion.
\end{proof}

Next, we give the proof of Lemma \ref{L: Evaluation functionals in Sobolev spaces have 1/2 regularity}. 

\begin{proof}[Proof of Lemma \ref{L: Evaluation functionals in Sobolev spaces have 1/2 regularity}]
We recall 
that
$\delta_x(\cdot)=1+\min(x,\cdot)$. 
By the mean value theorem it holds for $x\in (0,1)$ and $t>0$ small enough, such that $t<x< 1-t$
\begin{align*}
    \|\mathcal S(t) \delta_x-\delta_x\|^2
    = & \int_0^1\left(\indicator_{[0,x]}\left((y+t)\wedge 1\right)-\indicator_{[0,x]}(y)\right)^2 dy\\
    = & \int_0^{1-t} \indicator_{[x-t,x]}(y)dy+\int_{1-t}^1 \indicator_{[0,x]}(y)dy\\
    =& (x\wedge(1-t))-(x-t)+x-(1-t)\\
    = &x+2t-1\\
    \in & (t,2t).
\end{align*}
If $x=1$ and $t$ small enough such that $t\leq x$, it is
\begin{align*}
    \|\mathcal S(t) \delta_x-\delta_x\|^2
    = & \int_0^1\left(\indicator_{[0,1]}\left((y+t)\wedge 1\right)-\indicator_{[0,1]}(y)\right)^2 dy\\
    = & \int_0^{1-t} \indicator_{[x-t,x]}(y)dy+t\\
    =& 2t.
\end{align*}
This shows that $x\in (0,1]$ $\delta_x\in F_{\frac 12}^{\mathcal S}$ but $\delta_x\notin \in F_{\gamma}^{\mathcal S}$ for any $\gamma>\frac 12$. Moreover, $\delta_0\in  F_{\gamma}^{\mathcal S}$ for all $\gamma \in [0,1]$ holds as well.

We show that this holds for the adjoint semigroup $(\mathcal S(t))_{t\geq 0}$ as well: For this purpose, we first derive an explicit representation of the adjoint operator $\mathcal S(t)^*$.
 Let $g\in H$ be arbitrary.
 Then for $x<1$, we have, as $\delta_x'(1)=0$,
 \begin{align*}
     S^*(t)g(x)=&\langle \mathcal S(t)\delta_x,g\rangle\\
     = & \delta_x(t)g(0)+\int_0^1 \delta_x'\left((y+t)\wedge 1\right)g'(y) dy\\
    = & \delta_x(t)g(0)+\int_0^{1-t} \delta_x'\left(y+t\right)g'(y) dx+\int_{1-t}^{1} \delta_x'\left(1\right)g'(y) dy\\
    = &\left(\int_0^t \delta_x'(y) dy +\delta_x(0)\right)g(0)+\int_t^{1} \delta_x'\left(y\right)g'(y-t) dx\\
     = &\left(\int_0^t \indicator_{[0,x]}(y) dx +1\right)g(0)+\int_t^{1} \indicator_{[0,x]}\left(y\right)g'(y-t) dx\\
    = &g(0)+ \int_0^1 \indicator_{[0,x]}(y)\left(\indicator_{[0,t]}(y)g(0)+\indicator_{[t,1]}(y) g'(y-t)\right) dy.
\end{align*}
This yields $(\mathcal S(t)^*g)(0)=g(0)$ and, for all $0<t<1$ and $x\in [0,1)$, 
\begin{align*}
    (\mathcal S(t)^*g)'(x)= \left(\indicator_{[0,t]}(x)g(0)+\indicator_{[t,1]}(x) g'(x-t)\right).
\end{align*}
In particular, $\mathcal S(t)^*\delta_{x}(0)=1$ and, for all $0<t<1$ and $y\in [0,1)$,
\begin{align*}
    (\mathcal S(t)^*\delta_x)'(y)= & \left(\indicator_{[0,t]}(y)+\indicator_{[t,1]}(y) \indicator_{[0,x]}(y-t)\right)\\
    = & \left(\indicator_{[0,t]}(y)+ \indicator_{[t,(x+t)\wedge 1]}(y)\right).
\end{align*}
Therefore, 
for $t$ small enough such that $0\leq x<1-t$ 
\begin{align*}
    \| \mathcal S(t)^* \delta_{x}-\delta_{x}\|^2=&  \int_0^1\left(\indicator_{[0,t]}(y)+ \indicator_{[t,x+t]}(y)-\indicator_{[0,x]}(y)\right)^2dy 
    =  \int_0^1 \indicator_{[x,x+t]}(y)dy 
    =  t.
\end{align*}
This shows that all $\delta_{x}$ with $0\leq x<1$ are contained in the Favard space $F_{\frac 12}^{\mathcal S^*}$, but not in $F_{\gamma}^{\mathcal S^*}$ for $\gamma >\frac 12$. Moreover, $\delta_1\in  F_{\gamma}^{\mathcal S^*} $ for all $\gamma \in [0,1]$ holds as well.
\end{proof}

\end{document}